\documentclass[11pt,reqno]{amsart}
\usepackage{amsmath,amssymb,latexsym,esint,cite,mathrsfs}
\usepackage{verbatim,wasysym}
\usepackage[left=1.8cm,right=1.8cm,top=2.4cm,bottom=2.4cm]{geometry}
\usepackage{tikz,enumitem,graphicx, subfig, microtype, color}
\usepackage{epic,eepic}

\usepackage[colorlinks=true,urlcolor=blue, citecolor=red,linkcolor=blue,
linktocpage,pdfpagelabels, bookmarksnumbered,bookmarksopen]{hyperref}
\usepackage[hyperpageref]{backref}
\usepackage[english]{babel}
\usepackage{appendix}

\numberwithin{equation}{section}

\newtheorem{thm}{Theorem}[section]
\newtheorem{lem}[thm]{Lemma}
\newtheorem{cor}[thm]{Corollary}
\newtheorem{Prop}[thm]{Proposition}
\newtheorem{Def}[thm]{Definition}

\newtheorem{step}[thm]{Step}

\newcommand{\dist}{{\rm dist}}

\begin{document}
	\title[Stability for nonlocal Sobolev inequality]
	{Remainder terms,
		profile decomposition and sharp quantitative stability in the fractional nonlocal Sobolev-type inequality with $n>2s$}
	
	\author[Q. Lu]{Qikai Lu}
	\address{\noindent Qikai Lu  \newline School of Mathematical Sciences, Zhejiang Normal University,\newline
		Jinhua 321004, Zhejiang, People's Republic of China.}
	\email{luqikai@zjnu.edu.cn}
	
	\author[M. Yang]{Minbo Yang$^\dag$}
	\address{\noindent Minbo Yang  \newline
		School of Mathematical Sciences, Zhejiang Normal University,\newline
		Jinhua 321004, Zhejiang, People's Republic of China.}\email{mbyang@zjnu.edu.cn}
	
	\author[S. Zhao]{Shunneng Zhao$^\ddag$}
	\address{\noindent Shunneng Zhao  \newline
		School of Mathematical Sciences, Zhejiang Normal University,\newline
		Jinhua 321004, Zhejiang, People's Republic of China.}
	
	\address{Dipartimento di Matematica, Universit\`{a} degli Studi di Bari Aldo Moro,\newline Via Orabona 4, 70125 Bari, Italy.
	}\email{snzhao@zjnu.edu.cn}

\thanks{Date: March 09, 2025.}
	
	
	\thanks{$^\dag$Minbo Yang was partially supported by National Key Research and Development Program of China (No. 2022YFA1005700) and National Natural Science Foundation of China (12471114).}
	\thanks{$^\ddag$Shunneng Zhao is the corresponding author who was partially supported by National Natural Science Foundation of China (12401146, 12261107), Natural Science Foundation of Zhejiang Province (LMS25A010007) and PNRR MUR project PE0000023 NQSTI - National Quantum Science and Technology Institute (CUP H93C22000670006).}
	
	\allowdisplaybreaks
	
	\begin{abstract}
		{\small
			In this paper, we study the following fractional nonlocal Sobolev-type inequality
			\begin{equation*}
				C_{HLS}\bigg(\int_{\mathbb{R}^n}\big(|x|^{-\mu} \ast |u|^{p_s}\big)|u|^{p_s} dx\bigg)^{\frac{1}{p_s}}\leq\|u\|_{\dot{H}^s(\mathbb{R}^n)}^2\quad \mbox{for all}~~u\in \dot{H}^s(\mathbb{R}^n),
			\end{equation*}
			induced by the classical fractional Sobolev inequality and Hardy-Littlewood-Sobolev inequality for $s\in(0,\frac{n}{2})$, $\mu\in(0,n)$ and where $p_{s}=\frac{2n-\mu}{n-2s}\geq2$ is energy-critical exponent. The $C_{HLS}>0$ is a constant depending on the dimension $n$, parameters $s$ and $\mu$,  which can be achieved by $W(x)$, and up to translation and scaling, $W(x)$ is the unique positive and radially symmetric extremal function of the nonlocal Sobolev-type inequality. It is well-known that, up to a suitable scaling,
			\begin{equation*}
				(-\Delta)^{s}u=(|x|^{-\mu}\ast |u|^{p_s})|u|^{p_s-2}u\quad \mbox{for all}~~u\in\dot{H}^s(\mathbb{R}^n),
			\end{equation*}
			is the Euler-Lagrange equation corresponding to the associated minimization problem.
			
			In this paper, we first prove the non-degeneracy of positive solutions to the critical Hartree equation for all $s\in(0,\frac{n}{2})$, $\mu\in(0,n)$ with $0<\mu\leq4s$. Furthermore, we show the existence of a gradient type remainder term and, as a corollary, derive the existence of a remainder term in the weak $L^{\frac{n}{n-2s}}$-norm for functions supported in domains of finite measure, under the condition $s\in(0,\frac{n}{2})$. Finally, we establish a Struwe-type profile decomposition and quantitative stability estimates for critical points of the above inequality in the parameter region $s\in(0,\frac{n}{2})$ with the number of bubbles $\kappa\geq1$, and for $\mu\in(0,n)$ with $0<\mu\leq4s$. In particular, we provide an example to illustrate the sharpness of our result for $n=6s$ and $\mu=4s$.
		}
	\end{abstract}
	
	\vspace{3mm}
	
	\maketitle
	\section{Introduction}
	\subsection{Motivation and main results}
	The classical fractional Sobolev inequality for exponent $s\in(0,\frac{n}{2})$ states that there exists a dimensional constant $S=S_{n,s}>0$ such that
	\begin{equation}\label{bsic}
		\int_{\mathbb{R}^n}|(-\Delta)^{s/2}u(x)|^2dx\geq S_{n,s}\big(\int_{\mathbb{R}^n}|u(x)|^{2^*_{s}}dx\big)^{\frac{2}{2^*_s}}\quad \mbox{for all}~~u\in \dot{H}^s(\mathbb{R}^n),
	\end{equation}
	where $2_s^*=\frac{2n}{n-2s}$ and the homogeneous Sobolev space $\dot{H}^s(\mathbb{R}^n)$ is defined as the completion of $C^\infty_0(\mathbb{R}^n)$ with respect to the norm
	$$\|u\|_{\dot{H}^s(\mathbb{R}^n)}:=\Big(\int_{\mathbb{R}^n}|(-\Delta)^{s/2}u|^2dy\Big)^{1/2}=\Big(\int_{\mathbb{R}^n}|\hat{u}(\xi)|^2|\xi|^{2s}d\xi\Big)^{1/2}.$$ In \cite{Lieb83}, Lieb determined the optimal constant and identified that the extremal functions of \eqref{bsic} are functions of the form (so-called Talenti bubble):
	\begin{equation}\label{minimizer}
		U[\xi,\lambda](x)=c_{n,s}\Big(\frac{\lambda}{1+\lambda^2|x-\xi|^2}\Big)^{\frac{n-2s}{2}},\hspace{4mm}\lambda\in\mathbb{R}^{+},\hspace{4mm}\xi\in\mathbb{R}^n,
	\end{equation}
	for $x\in\mathbb{R}^n$ and $c_{n,s}:=2^{2s}(\Gamma(\frac{n+2s}{2})/\Gamma(\frac{n-2s}{2}))^{\frac{n-2s}{4s}}$.
	It is well known that the Euler-Lagrange equation associated to (\ref{bsic}) is given by
	\begin{equation}\label{bec}
		(-\Delta)^s u=|u|^{2^*_{s}-2}u\quad \mbox{for all}~~u\in \dot{H}^s(\mathbb{R}^n).
	\end{equation}
	Chen, Li and Ou \cite{Chen-ou}
	classified that all the positive solutions to equation (\ref{bec}) are the Talanti bubbles in \eqref{minimizer}.
	For $s\in(0,\frac{n}{2})$, an equivalent reformulation of inequality \eqref{bsic}, known as the (diagonal) Hardy-Littlewood Sobolev (HLS) inequality
	in \cite{H-L-1928,S1963} asserts that:
	\begin{equation}\label{hlsi}
		\int_{\mathbb{R}^n}\int_{\mathbb{R}^n}u(x)|x-y|^{-\mu} v(y)dxdy\leq C_{n,r,t,\mu}\|u\|_{L^r(\mathbb{R}^n)}\|v\|_{L^t(\mathbb{R}^n)},
	\end{equation}
	where $\mu\in(0,n)$, $1<r,t<\infty$ and $\frac{1}{r}+\frac{1}{t}+\frac{\mu}{n}=2$.
	The equivalence of \eqref{bsic} and \eqref{hlsi} follows from a duality argument because $2_{s}^\ast$ is the H\"{o}lder conjugate of $r:=\frac{2n}{n+2s}(\mu=n-2s)$.
	In the general diagonal case $t=r=\frac{2n}{2n-\mu}$, Lieb \cite{Lieb83} classified the extremal functions for the HLS inequality and determined the optimal constant:
	\begin{equation}\label{defhlsbc}
		C_{n,\mu}=\pi^{\mu/2}\frac{\Gamma((n-\mu)/2)}{\Gamma(n-\mu/2)}\left(\frac{\Gamma(n)}{\Gamma(n/2)}\right)^{1-\frac{\mu}{n}}.
	\end{equation}
	Moreover, the equality holds if and only if
	\begin{equation*}
		u(x)=av(x)=a\big(\frac{1}{1+\lambda^2|x-x_0|^2}\big)^{\frac{2n-\mu}{2}}
	\end{equation*}
	for some $a\in \mathbb{C}$, $\lambda\in \mathbb{R}\backslash\{0\}$ and $x_0\in \mathbb{R}^n$.
	
	For the special case $s=1$, that is, the classical first-order Sobolev inequality.
	Brezis and Lieb \cite{BE198585} asked the question whether
	a remainder term proportional to the quadratic distance of the function $u$ to be the smooth manifold of extremal functions $\mathcal{M}$ can be added to the right hand side of (\ref{bsic}). This question was first addressed in the case $s=1$ by Bianchi and Egnell \cite{BE91}. They demonstrated that the quantity $c_{BE}^{-1}\big(\big\|\nabla u\big\|_{L^2}^2-S^2\big\|u\big\|_{L^{2^\ast}}^2\big)$ is bounded
	from below in terms of some natural distance to the manifold of optimizers.
	Subsequently, the result of Bianchi and Egnell was extended to the biharmonic case in \cite{L-W-99} and poly-harmonic operators in \cite{B-W-W,G-W10} as well. The fractional case was dealt with in \cite{CFW13}, while an explicit lower bound for the first-order Sobolev inequality was established in \cite{D-E-F-F-L}. More recently, K\"{o}nig proved that the sharp constant $ c_{BE}$ must be strictly smaller than $\frac{4}{n+4}$ (cf. \cite{T-K-23,T-K-24}). This refinement provides further insight into the stability of the Sobolev inequality and the optimality of associated constants. The quantitative stability of the Sobolev inequality (for the Sobolev space $p\neq2$) was carried out in \cite{C-F-M-P,F-N-18,F-Z-22}, and HLS inequality has also been studied in \cite{Ca17, C-L-T, Dolbeault-1, DE22}.
	
	A natural and more challenging perspective is to consider the quantitative stability of critical points of the Euler-Lagrange equation \eqref{bec}. If $s=1$, Struwe's seminal work \cite{Struwe-1984} established the
	well-known stability of profile decompositions for \eqref{bec}. Starting with \cite{Struwe-1984}, extensive research has been devoted to the stability properties of the Euler-Lagrange equation associated with the embedding $\mathcal{D}^{1,2}(\mathbb{R}^n)\hookrightarrow L^{2n/(n-2)}(\mathbb{R}^n)$ (where $\mathcal{D}^{1,2}(\mathbb{R}^n)$ is the completion of
	$C_0^{\infty}(\mathbb{R}^n)$ with respect to the norm $\|u\|_{L^2(\mathbb{R}^n)}$). A significant breakthrough was made by Ciraolo, Figalli, and Maggi in \cite{CFM18}, where they established the first sharp quantitative stability result for the one-bubble case ($\kappa = 1$) in dimensions $n \geq 3$. Subsequently, Figalli and Glaudo extended this stability result to the multi-bubble case ($\kappa \geq 2$) in dimensions $3 \leq n \leq 5$ in \cite{FG20}. Meanwhile, the authors in \cite{FG20} constructed counter-examples showing that when $n\geq6$, one may have
	\begin{equation*}
		\inf\limits_{\xi_1,\xi_2,\cdots,\xi_\kappa\in\mathbb{R}^n,
			\lambda_1,\lambda_2,\cdots,\lambda_{\kappa}\in\mathbb{R}^+}\big\|\nabla u-\nabla\big(\sum_{i=1}^\kappa U[\xi_i,\lambda_i]\big)\big\|_{L^2(\mathbb{R}^n)}\gg\big\|\Delta u+u^{\frac{n+2}{n-2}}\big\|_{(\mathcal{D}^{1,2}(\mathbb{R}^n))^{-1}}.
	\end{equation*}
	The optimal estimates in dimensions $n\geq6$ were recently established by Deng, Sun and Wei in \cite{DSW21}.
	Following the rigidity result in \cite{Chen-ou}, an analog of Struwe-type profile decompositions was derived by Palatucci and Pisante in \cite{P-P-15} and N. de Nitti and T. Konig in \cite{T-Konig23}. Furthermore, the quantitative stability results in \cite{CFM18,FG20,DSW21} was generalized by Aryan \cite{Aryan} for $s\in(0,1)$, and by Chen et al. \cite{C-K-L-24} for all $n\in\mathbb{N}$ and $s\in(0,\frac{n}{2})$ (see also \cite{T-Konig23}). We also refer to the extension results in \cite{Z} for the stability of fractional Sobolev trace inequality, and references therein.
	
	Recalling that the fractional Sobolev inequality \eqref{bsic} and
	HLS inequality \eqref{hlsi}, there exists an optimal constant $C_{HLS}>0$ depending only on $n$, $s$ and $\mu$ such that
	\begin{equation}\label{Prm}
		\int_{\mathbb{R}^n}|(-\Delta)^{\frac{s}{2}} u|^2dx\geq C_{HLS}\left(\int_{\mathbb{R}^n}(|x|^{-\mu} \ast|u|^{p_s})|u|^{p_s}dx\right)^{\frac{1}{p_s}}\quad \mbox{for all}~~u\in \dot{H}^s(\mathbb{R}^n)\hspace{2mm}\mbox{and}\hspace{2mm}s\in(0,\frac{n}{2}).
	\end{equation}
	It is well-known that the Euler-Lagrange equation associated with \eqref{Prm} is given by
	\begin{equation}\label{ele-1.1}
		(-\Delta)^s u=(|x|^{-\mu}\ast |u|^{p_s})|u|^{p_s-2}u\quad \mbox{for all}~~u\in \dot{H}^s(\mathbb{R}^n).
	\end{equation}
	Furthermore, Le \cite{ple} classified all positive solutions of \eqref{ele-1.1} are functions of the form
	\begin{equation}\label{defU}
		W[\xi,\lambda](x)=\alpha_{n,\mu,s}\big(\frac{\lambda}{1+\lambda^2|x-\xi|^2}\big)^{\frac{n-2s}{2}},\hspace{1mm}\lambda\in\mathbb{R}^{+},\hspace{1mm}\xi\in\mathbb{R}^n,
	\end{equation}
	and the constant $\alpha_{n,\mu,s}$ is given by $$\alpha_{n,\mu,s}:=\Big(\frac{2^{2s}\Gamma(\frac{n+2s}{2})\Gamma(\frac{2n-\mu}{2})}{\pi^{n/2}\Gamma\big((n-2s)/2\big)\Gamma\big((n-\mu)/2\big)}\Big)^{\frac{n-2s}{2(n+2s-\mu)}}.$$
	For $s=1$, the authors of \cite{GY18, DY19, GHPS19} independently computed the optimal constant $C_{HLS}>0$ and classified that the extremal functions of \eqref{Prm} are the bubbles $W[\xi,\lambda]$ described in \eqref{defU}. Moreover, they are all positive solutions to \eqref{ele-1.1}.
	The nondegeneracy of positive solutions for the critical Hartree equation \eqref{ele-1.1} plays an important role in PDE and functional analysis. As
	far as we know, the first nondegeneracy result is due to Li, Tang and Xu in \cite{LTX} for $s=1$, $n=6$ and $\mu=4$. Their approach relied on the expansion of the equations by spherical harmonics. Furthermore, Gao et al. \cite{GMYZ22} demonstrated the same results by using the method different from that in \cite{LTX}. Later on, Li et al. \cite{XLi} extended this nondegeneracy result to all $\mu\in(0,n)$ with $\mu\in(0,4]$ and $s=1$ by exploiting the spherical harmonic decomposition and the Funk-Hecke formula of spherical harmonic functions under the condition $v\in L^{\infty}(\mathbb{R}^n)$ which was established by using Moser's iteration method \cite{LLTX}. The nondegeneracy result for positive solutions $W[\xi,\lambda]$ was extended later to the case $s\in(0,1)$ by Deng et al. in \cite{DLYZ24}, and to the case $s=2$ by Zhang et al. in \cite{Zhang}. The first purpose of the present paper is to establish the nondegeneracy of the positive solutions $W[\xi,\lambda]$ of equation \eqref{ele-1.1} for all $s\in(0,\frac{n}{2})$, which can be summarized as follows:
	\begin{thm}\label{prondgr}
		Let $n\in\mathbb{N}$, $s\in(0,\frac{n}{2})$, $\mu\in(0,n)$ with $0<\mu\leq4s$.
		Then the solution
		$W:=W[\xi,\lambda](x)$ of equation (\ref{ele-1.1})
		is non-degenerate in the sense that all solutions of linearized equation
		\begin{equation*}
			(-\Delta)^{s}v=p_s\left(|x|^{-\mu} \ast (W^{p_s-1}v)\right)W^{p_s-1}
			+(p_s-1)\left(|x|^{-\mu} \ast W^{p_s}\right)W^{p_s-2}v,\hspace{4mm}v\in\dot{H}^s(\mathbb{R}^n),
		\end{equation*}
		are the linear combination of the functions
		\begin{equation*}
			\partial_\lambda W[\xi,\lambda],\hspace{2mm}\partial_{\xi_1}W[\xi,\lambda],\cdots,\partial_{\xi_n}W[\xi,\lambda].
		\end{equation*}
	\end{thm}
	For a further understanding of inequality \eqref{Prm}, it is natural to investigate the quantitative stability of \eqref{Prm}. In \cite{DSB213}, Deng et al. proved a gradient-type remainder term of inequality \eqref{Prm} for the case $s=1$. Later on, this result was extended to $s=2$ by Zhang et al. in \cite{Zhang}, and to the case of $s\in(0,1)$ with $s\neq\frac{1}{2}$ by Deng et al. in \cite{DLYZ24}. In this paper, we establish a corresponding gradient-type remainder term inequality for all values $s\in(0,\frac{n}{2})$. Our result can be stated as follows.
	\begin{thm}\label{remainder terms}
		Assume that $n\in\mathbb{N}$, $s\in(0,\frac{n}{2})$, and $\mu\in(0,n)$ with $0<\mu\leq4s$.
		Let a $(n+2)$-dimensional manifold be
		$$\mathcal{M}:=\{cW[\xi,\lambda]:~c\in\mathbb{R},~\xi\in\mathbb{R}^n, ~\lambda\in\mathbb{R}^+\},$$
		where $W$ is defined in \eqref{defU}.
		Then there exist two constants $A_2>A_1>0$ such that for every $u\in\dot{H}^s(\mathbb{R}^n)$, it holds
		\begin{equation*}
			A_2\dist(u,\mathcal{M})^2
			\geq \int_{\mathbb{R}^n}|(-\Delta)^{\frac{s}{2}} u|^2 dx
			-C_{HLS}\left(\int_{\mathbb{R}^n}(|x|^{-\mu} \ast|u|^{p_s})|u|^{p_s}dx\right)^{\frac{1}{p_s}}
			\geq A_1\dist(u,\mathcal{M})^2,
		\end{equation*}
		where $\dist(u,\mathcal{M}):=\inf_{ c\in\mathbb{R}, \xi\in\mathbb{R}^n, \lambda\in\mathbb{R}^+}\|u-cW[\xi,\lambda]\|_{\dot{H}^s(\mathbb{R}^n)}$.
	\end{thm}
	As a corollary of Theorem \ref{remainder terms},  we will consider a remainder term inequality for all functions $u\in\dot{H}^s(\Omega)\subset\dot{H}^s(\mathbb{R}^n)$ which vanish in $\mathbb{R}^n\setminus\Omega$. In sprite of the work of Br\'{e}zis and Lieb \cite{BE198585}, for each bounded domain $\Omega\subset\mathbb{R}^n$ with $|\Omega|<\infty$, define the weak $L^q$-norm
	\begin{equation}\label{defwn}
		\|u\|_{L^q_w(\Omega)}:=\sup_{D\subset\Omega}|D|^{-\frac{q-1}{q}}\int_{D}|u|dx,\quad\mbox{for}\hspace{2mm}1<q<\infty.
	\end{equation}
	Then we establish the second-type remainder term for inequality (\ref{Prm}) in a bounded domain, which is stated as follows.
	\begin{thm}\label{thmprtb}
		Let $n\in\mathbb{N}$, $s\in(0,\frac{n}{2})$, $\mu\in(0,n)$ with $0<\mu\leq4s$.
		Then there exists a constant $B'>0$ such that for every domain $\Omega\subset \mathbb{R}^n$ with $|\Omega|<\infty$, and every $u\in\dot{H}^s(\Omega)$, it holds that
		\begin{equation}\label{rtbdy46}
			\int_{\mathbb{R}^n}|(-\Delta)^{\frac{s}{2}} u|^2 dx
			-C_{HLS}\left(\int_{\mathbb{R}^n}(|x|^{-\mu} \ast|u|^{p_s})|u|^{p_s}dx\right)^{\frac{1}{p_s}}
			\geq B'|\Omega|^{-\frac{2}{2_{s}^{\ast}}}\big\|u\big\|^2_{L^{\frac{n}{n-2s}}_w(\Omega)},
		\end{equation}
		where $L^{\frac{n}{n-2s}}_w$ denotes the weak $L^{\frac{n}{n-2s}}$-norm as in (\ref{defwn}).
	\end{thm}
	The local stability of inequality \eqref{bsic} for $s=1$ was established by Br\'{e}zis and Lieb in \cite{BE198585}:
	\begin{equation}\label{kouhong}
		\|\nabla u\|^2_{L^2(\Omega)}-S\|u\|^2_{L^{2^*}(\Omega)}\geq C\|u\|^2_{L^{\frac{N}{N-2}}_w(\Omega)}\hspace{3mm}\mbox{for all}\hspace{2mm}u\in \mathcal{D}^{1,2}_0(\Omega),
	\end{equation}
	for some constant $C>0$. Furthermore, Bianchi and Egnell \cite{BE91} gave an alternative proof for inequality \eqref{kouhong} by using the nondegeneracy property of extremal functions. Later on, Chen et al. \cite{CFW13} extended this stability result to all $s\in(0,\frac{n}{2})$. In \cite{DSB213}, Deng et al. also proved that the existence of a weak $L^{\frac{n}{n-2}}$-remainder term for inequality \eqref{Prm} in the case $s=1$ by the rearrangement inequality and Lions' concentration-compactness principle. However, in order to avoid the rearrangement inequality we follow the same idea as in \cite{CFW13} to show Theorem \ref{thmprtb} by combining the conclusion of Theorem \ref{remainder terms}.
	
	Another way to examine the stability issue on inequality \eqref{Prm} is to study the stability of profile decompositions to equation \eqref{ele-1.1} for nonnegative functions. Inspired by the works of Struwe \cite{Struwe-1984}, Palatucci and Pisante \cite{P-P-15} and de Nitti and K\"onig \cite{T-Konig23}, we establish a fractional Hartree-type version of the profile decompositions associated with \eqref{ele-1.1}. More precisely,
	\begin{thm}\label{emm}
		Let $n\in \mathbb{N}$, $s\in(0,\frac{n}{2})$, $\mu\in(0,n)$ with $0<\mu\leq4s$ and $\kappa\geq1$ be positive integers. Let $\{u_m\}_{m=1}^\infty\subseteq \dot{H}^s(\mathbb{R}^n)$ be a sequence of nonnegative functions such that $$(\kappa-\frac{1}{2})C_{HLS}^{\frac{2n-\mu}{n+2s-\mu}}\leq\big\|u_m\big\|_{\dot{H}^s(\mathbb{R}^n)}^2\leq(\kappa+\frac{1}{2})C_{HLS}^{\frac{2n-\mu}{n+2s-\mu}}$$ with $C_{HLS}=C_{HLS}(n,\mu,s)$ defined in \eqref{Prm}, and assume that
		$$
		\Big\|(-\Delta)^{s}u_m-\left(|x|^{-\mu}\ast |u_m|^{p_{s}}\right)|u_m|^{p_{s}-2}u_m\Big\|_{\dot{H}^{-s}(\mathbb{R}^n)}\rightarrow0\hspace{6mm}\mbox{as}\hspace{2mm}m\rightarrow\infty.
		$$
		Then, there exist $\kappa$-tuples of points $\{\xi_{1}^{(m)},\cdots, \xi_{\kappa}^{(m))}\}_{m=1}^\infty$ and positive real numbers $\{\lambda_{1}^{(m)},\cdots,\lambda_{\kappa}^{(m)}\}_{m=1}^\infty$  such that
		\[
		\Big\|u_m-\sum_{i=1}^{\kappa}W[\xi_i^{(m)},\lambda_{i}^{(m)}]\Big\|_{\dot{H}^s(\mathbb{R}^n)}\rightarrow0\hspace{6mm}\mbox{as}\hspace{2mm}m\rightarrow\infty.
		\]
	\end{thm}
	For $s=1$, the above theorem is reduced to one obtained by Piccione, Yang and Zhao in \cite{p-y-z24}. Building on this result, recent contributions have focused on quantitative stability analysis for solutions to the Euler-Lagrange equation \eqref{ele-1.1}.
	The first quantitative stability for
	critical points of equation \eqref{ele-1.1} was established in \cite{p-y-z24}
	for one-bubble case ($\kappa=1$) without dimension restrictions and multi-bubbles case ($\kappa\geq2$) if $3\leq n<6-\mu$ and $\mu\in(0,n)$ with $\mu\in(0,4]$. See also \cite{L}. These works generalized the results by Figalli and Glaudo in \cite{FG20} to a nonlocal version of Sobolev inequality \eqref{Prm} for $s=1$ based on the spectrum analysis. Very recently, Yang and Zhao in \cite{YZ25} gave a delicate quantitative stability estimates for the critical points of equation \eqref{ele-1.1} for $s=1$ by means of constructing the weight functions related to parameters $n$ and $\mu$. However, the strategy of \cite{YZ25} leads to the constraint of space dimension, that is, $\frac{n^2-6n}{n-4}<\mu<4\hspace{2mm}\mbox{and}\hspace{2mm}n\geq6-\mu$. Later, Dai et al. \cite{DHP} established an alternative formulation of the quantitative stability of \eqref{ele-1.1} for $s=1$ when the dimensions $n\geq3$ and the number of bubbles $\kappa\geq1$.
	
	To the best of our knowledge, the problem of quantitative stability of fractional nonlocal Sobolev-type inequality \eqref{Prm}, specifically concerning critical points for all $s\in(0,\frac{n}{2})$ with $s\neq1$ and the parameter $\mu\in(0,n)$ with $0<\mu\leq4s$, has not been investigated so far. We shall address this gap by proving the following result.
	\begin{thm}\label{Figalli}
		Let $n\in\mathbb{N}$, $s\in(0,\frac{n}{2})$, $\mu\in(0,n)$ with $0<\mu\leq4s$
		and the number of bubbles $\kappa\in\mathbb{N}$. Then there exist a small constant $\delta=\delta(n,\mu,s,\kappa)>0$ and a large constant $C=C(n,\mu,s,\kappa)>0$ such that the following statement holds. If $u\in \dot{H}^{s}(\mathbb{R}^n)$ satisfies
		\begin{equation}\label{w-tittle}
			\Big\|\ u-\sum_{i=1}^{\kappa}\widetilde{W}_i\Big\|_{\dot{H}^s(\mathbb{R}^n)}\leq\delta
		\end{equation}
		for some $\delta$-interacting family
		$\{\widetilde{W}_i\}_{i=1}^{\kappa}$,
		then there is a family $\{W_i\}_{i=1}^{\kappa}$ of bubbles such that
		\begin{equation}\label{tnu}
			\Big\|u-\sum_{i=1}^{\kappa}W_i\Big\|_{\dot{H}^s(\mathbb{R}^n)}\lesssim
			\left\lbrace
			\begin{aligned}
				&			\Gamma(u),\hspace{23mm}if\hspace{2mm}n>2s,\hspace{2mm}0<\mu\leq4s\hspace{2mm}\mbox{with}\hspace{2mm}\kappa=1,\\
				&
				\Gamma(u)|\log\Gamma(u)|^{\frac{1}{2}},\hspace{4.5mm}if\hspace{2mm}n=6s,\hspace{2mm}\mu=4s\hspace{2mm}\mbox{with}\hspace{2mm}\kappa\geq2,\\
				&				\Gamma(u),\hspace{23mm}if\hspace{2mm}n=6s,\hspace{2mm}\mu\in(0,4s)\hspace{2mm}\text{or}\hspace{2mm}n\neq6s,\hspace{2mm}\mu\in(0,4s]\hspace{2mm}\mbox{with}\hspace{2mm}\kappa\geq2,\\
			\end{aligned}
			\right.
		\end{equation}
		for $\Gamma(u)=\|(-\Delta)^{s}u-\left(|x|^{-\mu}\ast |u|^{p_{s}}\right)|u|^{p_{s}-2}u\|_{\dot{H}^{-s}(\mathbb{R}^n)}$. Moreover, for any $i\neq j$, the interaction between the bubbles can be estimated as
		\begin{equation}\label{moreover}
			\int_{\mathbb{R}^n}\big(|x|^{-\mu}\ast W_i^{p_{s}}\big)W_i^{p_s-1}W_j\lesssim \Big\|(-\Delta)^{s}u-\left(|x|^{-\mu}\ast |u|^{p_{s}}\right)|u|^{p_{s}-2}u\Big\|_{\dot{H}^{-s}(\mathbb{R}^n)}^{\min\{\frac{n-2s}{\mu},1\}}.
		\end{equation}
	\end{thm}
	Here we say that the family of bubbles $\{W_i\}_{i=1}^{\kappa}:=\{W[\xi_i,\lambda_i]\}_{i=1}^{\kappa}$ is $\delta$-interacting in the following sense:
	\begin{equation*}
		\mathscr{Q}:=\max\big\{Q_{ij}(\xi_i,\xi_j,\lambda_i,\lambda_j): \hspace{3mm} i,~j=1,\cdots,\kappa\big\}\leq\delta,
	\end{equation*}
	where the quantity is given by
	\begin{equation}\label{mianbao}	Q_{ij}(\xi_i,\xi_j,\lambda_i,\lambda_j)=\min\Big(\frac{\lambda_i}{\lambda_j}+\frac{\lambda_j}{\lambda_i}+\lambda_i\lambda_j|\xi_i-\xi_j|^2\Big)^{-\frac{n-2s}{2}}.
	\end{equation}
	
	By using perturbation methods as those in \cite{CFM18}, it is easy to verify that this power $1$ in Theorem \ref{Figalli} is sharp when $\kappa=1$. In what follows, we shall establish the sharpness of our result in \eqref{tnu} for $n=6s$ and $\mu=4s$ with $\kappa\geq2$.
	\begin{thm}\label{xiuzhengdai}
		Assume that $n=6s$ and $\mu=4s$. For sufficiently large $\mathscr{R}>0$, there exists some $\varrho$ such that if $u=W[-\mathscr{R}e_1,1]+W[\mathscr{R}e_1,1]+\varrho$ where $e_1=(1,0,\cdots,0)\in\mathbb{R}^n$, then
		\begin{equation*}
			\inf\limits_{\xi_1,\xi_2,\cdots,\xi_\kappa\in\mathbb{R}^n,
				\lambda_1,\lambda_2,\cdots,\lambda_{\kappa}\in\mathbb{R}^+}\big\|u-\sum_{i=1}^2 W[\xi_i,\lambda_i]\big\|_{\dot{H}^s(\mathbb{R}^n)}
			\gtrsim\Gamma(u)|\log\Gamma(u)|^\frac{1}{2},
		\end{equation*}
		for $\Gamma(u)=\|(-\Delta)^{s}u-\left(|x|^{-\mu}\ast |u|^{p_{s}}\right)|u|^{p_{s}-2}u\|_{\dot{H}^{-s}(\mathbb{R}^n)}$.
	\end{thm}
	As a direct consequence of Theorem \ref{emm} and Theorem \ref{Figalli}, we can prove the following corollary.
	\begin{cor}\label{Figalli2}
		Let $n\in\mathbb{N}$, $s\in(0,\frac{n}{2})$, $\mu\in(0,n)$ with $0<\mu\leq4s$
		and the number of bubbles $\kappa\in\mathbb{N}$. Then there exists a constant $C=C(n,\mu,s,\kappa)>0$ such that the following statement holds. For any nonnegative function $u\in \dot{H}^{s}(\mathbb{R}^n)$ such that
		\begin{equation*}
			\big(\kappa-\frac{1}{2})C_{HLS}^{\frac{2n-\mu}{n+2s-\mu}}\leq\|u\|_{\dot{H}^{s}(\mathbb{R}^n)}^2\leq\big(\kappa+\frac{1}{2}\big)C_{HLS}^{\frac{2n-\mu}{n+2s-\mu}},
		\end{equation*}
		then there exist $\kappa$ bubbles $\{W_i\}_{i=1}^{\kappa}$ such that
		\begin{equation*}
			\Big\|u-\sum_{i=1}^{\kappa}W_i\Big\|_{\dot{H}^s(\mathbb{R}^n)}\lesssim
			\left\lbrace
			\begin{aligned}
				&			\Gamma(u),\hspace{23mm}if\hspace{2mm}n>2s,\hspace{2mm}0<\mu\leq4s\hspace{2mm}\mbox{with}\hspace{2mm}\kappa=1,\\
				&
				\Gamma(u)|\log\Gamma(u)|^{\frac{1}{2}},\hspace{4.5mm}if\hspace{2mm}n=6s,\hspace{2mm}\mu=4s\hspace{2mm}\mbox{with}\hspace{2mm}\kappa\geq2,\\
				&				\Gamma(u),\hspace{23mm}if\hspace{2mm}n=6s,\hspace{2mm}\mu\in(0,4s)\hspace{2mm}\text{or}\hspace{2mm}n\neq6s,\hspace{2mm}\mu\in(0,4s]\hspace{2mm}\mbox{with}\hspace{2mm}\kappa\geq2,\\
			\end{aligned}
			\right.
		\end{equation*}
		for $\Gamma(u)=\|(-\Delta)^{s}u-\left(|x|^{-\mu}\ast u^{p_{s}}\right)u^{p_{s}-1}\|_{\dot{H}^{-s}(\mathbb{R}^n)}$.
		Furthermore, for any $i\neq j$, the interaction between the bubbles can be estimated as
		\begin{equation*}
			\int_{\mathbb{R}^n}\big(|x|^{-\mu}\ast W_i^{p_s}\big)W_i^{p_s-1}W_j\lesssim \Big\|(-\Delta)^{s} u-\left(|x|^{-\mu}\ast u^{p_{s}}\right)u^{p_s-1}\Big\|_{\dot{H}^{-s}(\mathbb{R}^n)}^{\min\{\frac{n-2s}{\mu},1\}}.
		\end{equation*}
	\end{cor}

	\subsection{Structure of the paper}
	The paper is organized as follows. In section \ref{section2}, we first prove Theorem \ref{prondgr}. Furthermore, we analysis spectral properties of the linear operator, and ultimately derive a spectral inequality. In section \ref{section3}, we establish Theorem \ref{remainder terms} by combining Lions's concentration-compactness principle and the discrete spectral information of the linear operator. Based on this result, we then complete the proof of Theorem \ref{thmprtb}. In section \ref{section4}, we establish a profile decomposition for nonnegative solutions of the critical Hartree-type equation \eqref{ele-1.1}. In section \ref{sangshen}, we prove the existence of first approximation for the case $n=6s,\mu\in(0,4s)$ or $n\neq6s,\mu\in(0,4s]$. The remaining case, $n=6s$ and $\mu=4s$, is addressed in section \ref{stabilitysection6}. Finally, in section \ref{section7}, we establish the quantitative stability estimates for $n>2s$, that is, Theorem \ref{Figalli} and Corollary \ref{Figalli2}. In particular, we provide an example to illustrate the sharpness of our result for $n=6s$ and $\mu=4s$ in section \ref{example}. Several technical computations required in the previous sections are provided in the appendix.
	
	Throughout this paper, $c$ and $C$ are indiscriminately used to denote various absolutely positive constants. We say that $a\lesssim b$ if $a\leq Cb$, $a\approx b$ if $a\lesssim b$ and $\gtrsim b$. For $z\in\mathbb{R}^n$, we write $\tau(z)=(1+|z|^2)^{1/2}$. For $x\in\mathbb{R}$, denote the piece-wise function
	\begin{equation}\label{hanshu}
		\mathcal{K}_{n,\mu,s}(x):=
		\left\lbrace
		\begin{aligned}
			&			x,\hspace{17.5mm}if\hspace{2mm}n>2s,\hspace{2mm}0<\mu\leq4s\hspace{2mm}\mbox{with}\hspace{2mm}\kappa=1,\\
			&
			x|\log x|^{\frac{1}{2}},\hspace{4.5mm}if\hspace{2mm}n=6s,\hspace{2mm}\mu=4s\hspace{2mm}\mbox{with}\hspace{2mm}\kappa\geq2,\\
			&				x,\hspace{18mm}if\hspace{2mm}n=6s,\hspace{2mm}\mu\in(0,4s)\hspace{2mm}\text{or}\hspace{2mm}n\neq6s,\hspace{2mm}\mu\in(0,4s]\hspace{2mm}\mbox{with}\hspace{2mm}\kappa\geq2.
		\end{aligned}
		\right.
	\end{equation}
	
	\section{Spectrum of the linear operator}\label{section2}
	\subsection{Non-degeneracy result}
	In this section, we first prove the nondegeneracy of positive solutions $W[\xi,\lambda]$ of equation \eqref{ele-1.1} for all $s\in(0,\frac{n}{2})$. For the simplicity of notations, we write $W$ instead of $W[\xi,\lambda]$ in the sequel. The following key estimate will be used a couple of times in the paper to conclude the proof of main results.
	\begin{lem}\label{B4-1}
		For $n>2s$, $\mu\in(0,n)$ and $0<\mu\leq4$, there exists constant $C>0$
		such that for all $i\in\{1,\cdots,\kappa\}$, there holds
		\begin{equation*}
			\frac{1}{\big|x\big|^{\mu}}\ast \frac{\lambda^{n-\mu/2}}{\tau(z)^{2n-\mu}}=\int	\frac{1}{\big|y\big|^{\mu}}\frac{\lambda^{n-\mu/2}}{(1+\lambda^2|x-\xi-y|^2)^{(2n-\mu)/2}}dy \leq C\frac{\lambda^{\mu/2}}{\tau(z)^{\mu}},
		\end{equation*}
		where $\tau(z)=(1+|z|^2)^{1/2}$ with $z=\lambda(x-\xi)$, and $\int\cdots=\int_{\mathbb{R}^n}\cdots$.
	\end{lem}
	\begin{proof}
		It can be solved in the same manner similar to the proof of Lemma 2.6 in  \cite{YZ25}. Here we omit the details.
	\end{proof}
	We can now prove Theorem \ref{prondgr}.
	\begin{proof}[Proof of Theorem \ref{prondgr}]
		Given $n\in\mathbb{N}$ and $s\in(0,\frac{n}{2})$, there holds
		\begin{equation}\label{nondege}
			v=\frac{\alpha_{n,s}}{|x|^{n-2s}}\ast\big[p_s\left(|x|^{-\mu} \ast (W^{p_s-1}v)\right)W^{p_s-1}
			+(p_s-1)\left(|x|^{-\mu} \ast W^{p_s}\right)W^{p_s-2}v\big]\hspace{4mm}\mbox{in}\hspace{2mm}\mathbb{R}^n,
		\end{equation}
		for some positive constant $\alpha_{n,s}$.
		Following the similar process to \cite{XLi,DLYZ24,Zhang} with minor modifications by exploiting the spherical harmonic decomposition and the Funk-Hecke formula of spherical harmonic functions, it suffices to check that $v\in L^{\infty}(\mathbb{R}^n)$ for all $s\in(0,\frac{n}{2})$.
		
		It is straightforward to verify that there exists a positive constant $C$ such that
		\begin{equation}\label{Non1}
			\big|v(x)\big|\leq\int\frac{C}{|x-w|^{n-2s}}\Big[\int\frac{|v(y)|}{|w-y|^\mu\tau(y)^{n-\mu+2s}} \frac{1}{\tau(w)^{n-\mu+2s}}dy
			+\int\frac{1}{|w-y|^\mu\tau(y)^{2n-\mu}}\frac{|v(y)|}{\tau(w)^{4s-\mu}}dy\Big]dw.
		\end{equation}
		
		$\bullet$ If $n>6s$, denote $$g(w)=\int\frac{v(y)}{|w-y|^\mu\tau(y)^{n-\mu+2s}}dy.$$
		Then by the Hardy-Littlewood-Sobolev inequality we get
		\begin{equation}\label{Non2}
			\begin{split}
				\big\|v(x)\big\|_{L^t(\mathbb{R}^n)}&\lesssim C\Big(\Big\|\int\frac{1}{|x-w|^{n-2s}} \frac{|g(w)|}{\tau(w)^{n-\mu+2s}}dw\Big\|_{L^t(\mathbb{R}^n)}+\Big\|\int\frac{1}{|x-w|^{n-2s}} \frac{|v(w)|}{\tau(w)^{4s}}dw\Big\|_{L^t(\mathbb{R}^n)}\Big)\\&\lesssim
				C\Big(\Big\|\frac{|g(w)|}{\tau(w)^{n-\mu+2s}}\Big\|_{L^{r}(\mathbb{R}^n)}+\big\|v(w)\big\|_{L^{r}(\mathbb{R}^n)}\Big)
			\end{split}
		\end{equation}
		for any $r\in[\frac{2n}{n-2s},\frac{n}{2s})$ and $t=\frac{nr}{n-2sr}$.
		One can easily compute
		\begin{equation}\label{Non3}
			\Big\|\frac{|g(w)|}{\tau(w)^{n-\mu+2s}}\Big\|_{L^{r}(\mathbb{R}^n)}\lesssim C\big\|g(w)\big\|_{L^{t}(\mathbb{R}^n)}\Big\|\frac{1}{\tau(w)^{n-\mu+2s}}\Big\|_{L^{\zeta}(\mathbb{R}^n)}\lesssim C\big\|g(w)\big\|_{L^{t}(\mathbb{R}^n)}
		\end{equation}
		for $t=\frac{nr}{n-2sr}$, $r\in[\frac{2n}{n-2s},\frac{n}{2s})$ and $\zeta=\frac{tr}{t-r}$. Combining \eqref{Non2} and \eqref{Non3}, we get
		\begin{equation}\label{4nn}
			\begin{split}
				\big\|v(x)\big\|_{L^t(\mathbb{R}^n)}&\lesssim C\Big(\Big\|\frac{1}{|x|^{\mu}} \ast\big|\frac{v}{\tau^{n-\mu+2s}}\big|\Big\|_{L^{t}(\mathbb{R}^n)}+\Big\|\frac{1}{|x|^{n-2s}}\ast \big|\frac{v}{\tau^{4s}}\big|\Big\|_{L^{t}(\mathbb{R}^n)}\Big)\\&\lesssim
				C\Big(\big\|v(w)\big\|_{L^{\zeta^{\ast}}(\mathbb{R}^n)}\Big\|\frac{1}{\tau(w)^{n-\mu+2s}}\Big\|_{L^{\beta}(\mathbb{R}^n)}+\big\|v(w)\big\|_{L^{r}(\mathbb{R}^n)}\Big)
			\end{split}
		\end{equation}
		for any $r\in[\frac{2n}{n-2s},\frac{n}{2s})$, $t=\frac{nr}{n-2sr}$ and $\zeta^{\ast}=\frac{2n}{n-2s}$, $\beta=\frac{\zeta\alpha}{\zeta-\alpha}$ with $\alpha=\frac{nr}{(n-\mu)r+n-2sr}\in[\frac{2n}{3n-2\mu-6s},\frac{n}{n-\mu})$.
		
		Note that we can start the iteration. 
		Choosing $r=r_1:=\frac{2n}{n-2s}$, then by \eqref{4nn} we have $v\in L^{r_2}(\mathbb{R}^n)$ for $t=r_2:=\frac{nr_1}{n-2sr_1}$, which means that $v\in L^{r_0^{\prime}}(\mathbb{R}^n)$ for all $r_0^{\prime}\in[r_1,r_2]$.
		In the following, we distinguish two cases depending on whether $r_2\geq\frac{n}{2s}$ or not. If $r_2\geq\frac{n}{2s}$, then we get $v\in L^{r_0^{\prime}}(\mathbb{R}^n)$ for all $r_0^{\prime}\geq\frac{2n}{n-2s}$. If $r_2<\frac{n}{2s}$, then we get $v\in L^{r_0^{\prime}}(\mathbb{R}^n)$ for all $r_0^{\prime}\in[r_2,r_3]$ where $r_3:=\frac{nr_1}{n-2sr_1}$. By iterating the above arguments finite times, then we eventually get $v\in L^{r_0^{\prime}}(\mathbb{R}^n)$ for all $r_0^{\prime}\geq\frac{2n}{n-2s}$, since $r_{k+1}\geq\frac{n-2s}{n-6s}r_k$.
		
		Denote $\gamma=\frac{2n}{2n-\mu}$, then for some fixed $r\gg1$ large enough, together with the Hardy-Littlewood-Sobolev inequality and H\"{o}lder inequality, we can compute
		\begin{equation}\label{v4nv}
			\begin{split}
				\int&\frac{1}{|x-w|^{n-2s}}\int\frac{|v(y)|}{|w-y|^\mu\tau(y)^{n-\mu+2s}} \frac{1}{\tau(w)^{n-\mu+2s}}dydw\\&
				\lesssim\Big(\int\frac{1}{|x-\omega|^{(n-2s)\gamma}}\frac{1}{\tau(w)^{(n-\mu+2s)\gamma}}dw\Big)^{\frac{1}{\gamma}}\Big(\int\frac{1}{\tau(w)^{2n}}dw\Big)^{\frac{n-\mu+2s}{2n}}\big\|v\big\|_{L^{\zeta^{\ast}}(\mathbb{R}^n)}
				\lesssim1\hspace{4mm}\mbox{for}\hspace{2mm}x\in\mathbb{R}^n.
			\end{split}
		\end{equation}
		By Lemma \ref{B4-1} and the H\"{o}lder inequality, we get for $x\in\mathbb{R}^n$,
		\begin{equation}\label{4n1}
			\begin{split}
				\int\frac{1}{|x-w|^{n-2s}}\int\frac{1}{|w-y|^\mu\tau(y)^{2n-\mu}}\frac{|v(w)|}{\tau(w)^{4s-\mu}}dydw\lesssim
				\Big(\int\frac{1}{|x-w|^{(n-2s)r^{\prime}}}\frac{1}{\tau(w)^{4sr^{\prime}}}dw\Big)^{\frac{1}{r^{\prime}}}\big\|v\big\|_{L^{r}(\mathbb{R}^n)}\lesssim1,
			\end{split}
		\end{equation}
		where $r^{\prime}$ is the H\"{o}lder conjugate of $r$. In order to guarantee that the estimates of \eqref{v4nv}-\eqref{4n1} hold true, we establish the following key estimates:
		\begin{equation}\label{w4n}
			\begin{split}
				\int\frac{1}{|x-w|^{(n-2s)\gamma}}\frac{1}{\tau(w)^{(n-\mu+2s)\gamma}}dw\lesssim\frac{1}{\tau(x)^{(2n-\mu)\gamma-n}}\hspace{4mm}\mbox{for}\hspace{2mm}x\in\mathbb{R}^n,
			\end{split}
		\end{equation}
		\begin{equation}\label{xwxw}
			\begin{split}
				\int\frac{1}{|x-w|^{(n-2s)r^{\prime}}}\frac{1}{\tau(w)^{4sr^{\prime}}}dw\lesssim\frac{1}{\tau(x)^{(n+2s)r^{\prime}-n}}\hspace{4mm}\mbox{for}\hspace{2mm}x\in\mathbb{R}^n.
			\end{split}
		\end{equation}
		Indeed, by considering the cases: $A^{\prime}:=\{w:|w|\leq d~\mbox{for}~d:=\frac{1}{2}|x|\geq1\}$, $A^{\prime\prime}:=\{w:|\omega-x|\leq d~\mbox{for}~d\geq1\}$ and $\{w:~\mathbb{R}^n\setminus(A^{\prime}\cup A^{\prime\prime})\}$ separately, we can deduce that \eqref{w4n} holds by the same manner similar to Lemma \ref{B4-1} (cf. \cite{YZ25}) and \eqref{xwxw} can follow the same argument as \eqref{w4n}.
		
		$\bullet$ If $n=6s$, then we have $r\in[3,3)=\emptyset$. By the Hardy-Littlewood-Sobolev inequality and Sobolev inequality we deduce that
		\begin{equation}\label{n6s}
			\begin{split}
				\big\|v(x)\big\|_{L^t(\mathbb{R}^n)}&\lesssim C\Big(\Big\|\int\frac{1}{|x-w|^{4s}} \frac{|g(w)|}{\tau(w)^{8s-\mu}}dw\Big\|_{L^t(\mathbb{R}^n)}+\Big\|\int\frac{1}{|x-w|^{4s}} \frac{|v(w)|}{\tau(w)^{4s}}dw\Big\|_{L^t(\mathbb{R}^n)}\Big)\\&\lesssim
				C\Big(\Big\|\frac{|g(w)|}{\tau(w)^{n-\mu+2s}}\Big\|_{L^{\eta_1}(\mathbb{R}^n)}+
				\Big\|\frac{|v(w)|}{\tau(w)^{4s}}\Big\|_{L^{\eta_1}(\mathbb{R}^n)}\Big)\\&
				\lesssim C\big(\big\|v(w)\big\|_{L^{3}(\mathbb{R}^n)}\Big\|\frac{1}{\tau(w)^{n-\mu+2s}}\Big\|_{L^{\eta_2}(\mathbb{R}^n)}+\big\|v(w)\big\|_{L^{3}(\mathbb{R}^n)}\big)\Big\|\frac{1}{\tau(w)^{4s}}\Big\|_{L^{\eta_2}(\mathbb{R}^n)}
			\end{split}
		\end{equation}
		for $t=\frac{3\eta_1}{3-\eta_1}$ with $\eta_1=\frac{3\eta_2}{3+\eta_2}$, and $t=\frac{3\eta_1}{3-\eta_1}\in(3,\infty)$. Thus, we have $v\in L^{r_0^{\prime}}(\mathbb{R}^n)$ for all $r_0^{\prime}\geq\frac{2n}{n-2s}.$
		
		$\bullet$ If $2s<n<6s$, then a direct computation yields that
		\begin{equation*}
			\begin{split}
				\big\|v(x)\big\|_{L^t(\mathbb{R}^n)}&\lesssim C\Big(\Big\|\frac{1}{|x|^{n-2s}} \ast\big|\frac{g(w)}{\tau^{n-\mu+2s}}\big|\Big\|_{L^{t}(\mathbb{R}^n)}+\Big\|\frac{1}{|x|^{n-2s}}\ast \big|\frac{v}{\tau^{4s}}\big|\Big\|_{L^{r}(\mathbb{R}^n)}\Big)\\&\lesssim
				C\big\|v(w)\big\|_{L^{\zeta^{\ast}}(\mathbb{R}^n)}\Big(\Big\|\frac{1}{\tau(w)^{n-\mu+2s}}\Big\|_{L^{\frac{n}{n-\mu}}(\mathbb{R}^n)}\Big\|\frac{1}{\tau(w)^{n-\mu+2s}}\Big\|_{L^{\theta}(\mathbb{R}^n)}
				+\Big\|\frac{1}{\tau(w)^{4s}}\Big\|_{L^{\theta}(\mathbb{R}^n)}\Big)
			\end{split}
		\end{equation*}
		for $\zeta^{\ast}=\frac{2n}{n-2s}$, $\theta\in(\frac{n}{2s},\frac{2n}{6s-n})$ and $t=\frac{n\zeta^{\ast}\theta}{n\zeta^{\ast}+(n-2s\zeta^\ast)\theta}\in(\frac{2n}{n-2s},\infty)$.
		Hence, we deduce that $v\in L^{r_0^{\prime}}(\mathbb{R}^n)$ for all $r_0^{\prime}\geq\frac{2n}{n-2s}.$
		Combining this bound with \eqref{v4nv}, \eqref{4n1}, \eqref{n6s}, and \eqref{Non1}, we conclude that $v\in L^{\infty}(\mathbb{R}^n)$ and the proof is finished.
	\end{proof}

	\subsection{Existence of discrete spectrums}\label{sectevp1}
	Define the linear operators
	\begin{equation}\label{defanndg}
		L[u]:=(-\Delta)^s u+\Big(|x|^{-\mu} \ast W^{p_s}\Big)W^{p_s-2}u
	\end{equation}
	and
	\begin{equation*}
		R[u]:=\Big(|x|^{-\mu} \ast \big(W^{p_s-1}u\big)\Big)W^{p_s-1}
		+\Big(|x|^{-\mu} \ast W^{p_s}\Big)W^{p_s-2}u.
	\end{equation*}
	We first investigate the following eigenvalue problem:
	\begin{eqnarray}\label{Pwhlep}
		L[u]=\tilde{\lambda}R[u],\quad u\in \dot{H}^{s}(\mathbb{R}^n).
	\end{eqnarray}
	
	Denote
	\begin{equation*}
		X:=\biggl\{u\in\dot{H}^s(\mathbb{R}^n) :  \int|(-\Delta)^\frac{s}{2}u|^2+\int(|x|^{-\mu} \ast W^{p_s})W^{p_s-2}u^2<\infty\biggr\}
	\end{equation*}
	and
	\begin{equation*}
		Y:=\biggl\{u\in\dot{H}^s(\mathbb{R}^n) : \int(|x|^{-\mu}\ast (W^{p_s-1}u))W^{p_s-1}u+ \int(|x|^{-\mu} \ast W^{p_s})W^{p_s-2}u^2<\infty\biggr\}.
	\end{equation*}
	We now show that the space $X$ compactly embeds into $Y$. It is worth noting that related results can be found in \cite{DLYZ24} for $s\in(0,1)$ with $s\neq\frac{1}{2}$. However, we emphasize that our results are valid for all fractional exponents $s\in(0,\frac{n}{2})$.
	\begin{lem}\label{niunai}
		The space $X$ compactly embeds into $Y$ for all $s\in(0,\frac{n}{2})$.
	\end{lem}
	\begin{proof}
		We first claim that the embedding $\dot{H}^s(\mathbb{R}^n)\hookrightarrow L^2(\mathbb{R}^n,W^{2_s^\ast-2}dx)$ is compact. Since $C_0^\infty(\mathbb{R}^n)$ is dense in $\dot{H}^s(\mathbb{R}^n)$, we can assume that $u\in C_0^\infty(\mathbb{R}^n)$. Then by the H\"older inequality, it follows
		\begin{equation}\label{shanshanjia}
			\int_{\mathbb{R}^n}W^{2_s^\ast-2}u^2
			\leq\bigg(\int_{\mathbb{R}^n}W^{2_s^\ast}\bigg)^\frac{2s}{n}\bigg(\int_{\mathbb{R}^n}u^{2_s^\ast}\bigg)^\frac{n-2s}{n}
			\leq C\|u\|_{\dot{H}^s(\mathbb{R}^n)}^2.
		\end{equation}
		For any $\rho\in(0,1)$, using the H\"older inequality again we have
		\begin{equation*}
			\int_{B_\rho}W^{2_s^\ast-2}u^2
			\leq\bigg(\int_{B_\rho}W^{2_s^\ast}\bigg)^\frac{2s}{n}\bigg(\int_{\mathbb{R}^n}u^{2_s^\ast}\bigg)^\frac{n-2s}{n}
			\leq C\rho^{2s}\|u\|_{\dot{H}^s(\mathbb{R}^n)}^2
		\end{equation*}
		and
		\begin{equation*}
			\begin{split}
				\int_{\mathbb{R}^n\setminus B_\frac{1}{\rho}}W^{2_s^\ast-2}u^2
				&\leq C\int_{\mathbb{R}^n\setminus B_\frac{1}{\rho}}\frac{u^2}{(1+|x|^2)^{2s}}       \\
				&\leq C\rho^s\int_{\mathbb{R}^n\setminus B_\frac{1}{\rho}}\frac{u^2}{|x|^{3s}}
				\leq C\rho^s\bigg(\int_{\mathbb{R}^n\setminus B_\frac{1}{\rho}}|x|^{-\frac{3n}{2}}\bigg)^\frac{2s}{n}\|u\|_{\dot{H}^s(\mathbb{R}^n)}^2
				\leq C\rho^{2s}\|u\|_{\dot{H}^s(\mathbb{R}^n)}^2.
			\end{split}
		\end{equation*}
		Next, we follow the arguments as those in \cite[Proposition 3.2]{F-Z-22}. Let $\{u_m\}$ be a sequence of functions in $\dot{H}^s(\mathbb{R}^n)$ with uniformly bounded norm. It follows from \eqref{shanshanjia} that $\{u_m\}$ is uniformly bounded in $L^2(\mathbb{R}^n,W^{2_s^\ast-2}dx)$ as well. By the Rellich-Kondrachov theorem and a digonal argument we deduce that, up to a subsequence, $\{u_m\}$ converges to some function $u$ both weakly in $\dot{H}^s(\mathbb{R}^n)\cap L^2(\mathbb{R}^n,W^{2_s^\ast-2}dx)$ and strongly in $L_{\text{loc}}^2(\mathbb{R}^n,W^{2_s^\ast-2}dx)$. Defining the compact set $E_\rho:=\overline{B_\frac{1}{\rho}}\setminus B_\rho$ and applying the strong convergency of $\{u_m\}$ inside $E_\rho$, we conclude the claim by choosing the arbitrarily small $\rho$.
		
		\par Now we assume that $u_m\rightharpoonup0$ in $X$. Thus $u_m\rightharpoonup0$ in $\dot{H}^s(\mathbb{R}^n)$ and $u_m\rightarrow0$ in $L^2(\mathbb{R}^n,W^{2_s^\ast-2}dx)$ respectively. By Lemma \ref{p1-00}, we get
		\begin{equation*}
			\int(|x|^{-\mu} \ast W^{p_s})W^{p_s-2}u_m^2
			=C\int W^{2_s^\ast-2}u_m^2
			\rightarrow0.
		\end{equation*}
		Moreover, by the Hardy-Littlewood-Sobolev inequality we obtain
		\begin{equation*}
			\int(|x|^{-\mu}\ast (W^{p_s-1}u_m))W^{p_s-1}u_m
			\leq C\bigg(\int W^\frac{2n(n+2s-\mu)}{(n-2s)(2n-\mu)}u_m^\frac{2n}{2n-\mu}\bigg)^\frac{2n-\mu}{n}
			\leq C\bigg(\int W^{2_s^\ast}\bigg)^\frac{n-\mu}{n}\int W^{2_s^\ast-2}u_m^2
			\rightarrow0.
		\end{equation*}
		As a consequence, we achieve $u_m\rightarrow0$ in $Y$, which completes the proof.
	\end{proof}
	
	Lemma \ref{niunai} indicates that the eigenvalues of problem (\ref{Pwhlep}) are discrete, which can be defined as follows:
	\begin{Def}\label{defevp}
		The Rayleigh quotient characterization of first eigenvalue implies
		\begin{equation}\label{deffev1}
			\tilde{\lambda}_1:=\inf_{v\in\dot{H}^{s}(\mathbb{R}^n)\backslash\{0\}}
			\frac{\int|(-\Delta)^\frac{s}{2}v|^2+\int\big(|x|^{-\mu} \ast W^{p_s}\big)W^{p_s-2}v^2}
			{\int\big(|x|^{-\mu} \ast W^{p_s-1}v\big)W^{p_s-1}v+ \int(|x|^{-\mu} \ast W^{p_s})W^{p_s-2}v^2}.
		\end{equation}
		In addition, for any $k\in\mathbb{N}$ the eigenvalues can be characterized as follows:
		\begin{equation}\label{deffevk}
			\tilde{\lambda}_{k+1}:=\inf_{v\in \mathcal{H}_{k+1}\backslash\{0\}}
			\frac{\int|(-\Delta)^\frac{s}{2}v|^2+\int\big(|x|^{-\mu} \ast W^{p_s}\big)W^{p_s-2}v^2}
			{\int\big(|x|^{-\mu}\ast W^{p_s-1}v\big)W^{p_s-1}v+ \int\big(|x|^{-\mu} \ast W^{p_s}\big)W^{p_s-2}v^2},
		\end{equation}
		where
		\begin{equation}\label{defczs}
			\mathcal{H}_{k+1}:=\left\{v\in \dot{H}^{s}(\mathbb{R}^n): \int_{\mathbb{R}^n}(-\Delta)^{\frac{s}{2}} v \cdot (-\Delta)^{\frac{s}{2}} v_j=0\quad \mbox{for all}\quad j=1,\ldots,k\right\},
		\end{equation}
		and $v_j$ is the corresponding eigenfunction to $\tilde{\lambda}_j$.
	\end{Def}
	
	We have the following discrete spectral information of operator $L[u]$ (see \eqref{defanndg}).
	\begin{lem}\label{propep}
		Let $\tilde{\lambda}_j$, $j=1,2,\ldots,$ denote the eigenvalues of (\ref{Pwhlep}) in increasing order as in Definition \ref{defevp}. Then the operator $L[u]$ defined in \eqref{defanndg} has a discrete spectrum such that
		\begin{itemize}
			\item[$(i)$] The first eigenvalue
			$\tilde{\lambda}_1=1$ with the eigenfunction space $\mathcal{H}_1=span\left\{W\right\}$.
			\item[$(ii)$] The second eigenvalue
			$\tilde{\lambda}_2=p_s$ with the eigenfunction space
			$$ \mathcal{H}_2=span
			\left\{\partial_{\mathcal{\xi}_1}W,\ldots,\partial_{ \mathcal{\xi}_n}W,\quad x\cdot\nabla W+\frac{n-2}{2}W\right\}.$$
		\end{itemize}
		Furthermore, $\tilde{\lambda}_{n+3}>\tilde{\lambda}_2=p_s$.
	\end{lem}
	\begin{proof}
		The result follows from the nondegeneracy of bubbles for linearized  equation of (\ref{ele-1.1}) at $W$. The proof can be solved in the same manner similar to the argument as in \cite{DSB213,DLYZ24} except minor modifications. So we omit the details.
	\end{proof}
	
	\subsection{Spectral inequality}
	For any $1\leq i\leq\kappa$, let $\varrho\in\dot{H}^s(\mathbb{R}^n)$ satisfy the following orthogonal conditions:
	\begin{equation*}
		\langle\varrho, W_i\rangle_{\dot{H}^s(\mathbb{R}^n)}=\langle\varrho, \partial_{\lambda} W_i\rangle_{\dot{H}^s(\mathbb{R}^n)}=\langle\varrho, \partial_{\xi_j}W_i\rangle_{\dot{H}^s(\mathbb{R}^n)}=0\hspace{6mm}\text{for any}\hspace{2mm}1\leq j\leq n.
	\end{equation*}
	In view of Lemma \ref{propep}, we know that the functions $W_i$, $\partial_{\lambda}W_i$ and $\partial_{\xi_j}W_i$ are eigenfunctions for the eigenvalue problem \eqref{Pwhlep}.
	Then the above orthogonal conditions are equivalent to
	\begin{equation}\label{EP1}
		\int\Big(|x|^{-\mu} \ast \big(W_i^{p_s-1}\varrho\big)\Big)W_i^{p_s}=
		\int I_{n,\mu,s}[W_i,\varrho]\partial_{\lambda} W_i=
		\int I_{n,\mu,s}[W_i,\varrho]\partial_{\xi_j}W_i=0,
	\end{equation}
	for any $1\leq i\leq\kappa$ and $1\leq j\leq n$.
	
	We need the following spectral inequality which plays a crucial role in the proof of Lemma \ref{Ni-1-3}.
	\begin{lem}\label{estim}
		Let $n>2s$, $\mu\in(0,n)$, $0<\mu\leq4$ and $\kappa\in\mathbb{N}$. There exists a positive constant $\delta=\delta(n,s,\kappa,\mu)>0$ such that if $\sigma=\sum\limits_{i=1}^{\kappa}W_i$ is a linear combination of $\delta$-interacting bubbles and $\varrho\in\dot{H}^s(\mathbb{R}^n)$ satisfies \eqref{EP1}. Then we have that
		\begin{equation}\label{ps1ps}
			(p_{s}-1)\int\Big(|x|^{-\mu}\ast \sigma^{p_{s}}\Big)\sigma^{p_{s}-2}\varrho^2
			+p_{s}\int\Big(|x|^{-\mu}\ast\big(\sigma^{p_{s}-1}\varrho\big)\Big)\sigma^{p_s-1}\varrho
			\leq\tau_0\big\|\varrho\big\|_{\dot{H}^s}^2,
		\end{equation}
		where $\tau_0$ is a constant strictly less than $1$ which depends only on $n$, $s$, $\kappa$ and $\mu$.
	\end{lem}
	\begin{proof}
		The case $\kappa=1$ is clear. For $\kappa\geq2$, the proof is similar to \cite{DHP,C-K-L-24} except minor modifications due to the parameters $s$. Hence, we omit the details.
	\end{proof}

	\section{Remainder terms of the fractional nonlocal Sobolev-type inequality}\label{section3}
	This section aims to establish both the gradient type remainder term and the remainder term in the weak $L^\frac{n}{n-2s}$-norm by utilizing the non-degeneracy property of the extremal functions.
	\subsection{Proof of Theorem \ref{remainder terms}}
	The main ingredient of the proof of Theorem \ref{remainder terms} is contained in Lemma \ref{baowenbei}, where the behavior of the sequences near $\mathcal{M}=\{cW[\xi,\lambda]: c\in\mathbb{R}, \lambda>0, \xi\in\mathbb{R}^n\}$ is investigated.
	
	\begin{lem}\label{baowenbei}
		Let $s\in(0,\frac{n}{2})$, $n>2s$, $\mu(0,n)$ with $0<\mu\leq4s$. For any sequence $\{u_m\}\subset\dot{H}^s(\mathbb{R}^n)\setminus\mathcal{M}$ satisfying
		\begin{equation*}
			\inf_m\|u_m\|_{\dot{H}^s(\mathbb{R}^n)}>0\hspace{4mm}\text{and}\hspace{4mm}\dist(u_m,\mathcal{M})\rightarrow0,
		\end{equation*}
		then we have
		\begin{equation}\label{caomei1}
			\liminf_{m\rightarrow\infty}\frac{\int|(-\Delta)^{\frac{s}{2}} u_m|^2 dx
				-C_{HLS}\left(\int(|x|^{-\mu} \ast|u_m|^{p_s})|u_m|^{p_s}dx\right)^{\frac{1}{p_s}}}{\dist(u_m,\mathcal{M})^2}
			\geq A_1
		\end{equation}
		and
		\begin{equation}\label{caomei2}
			\limsup_{m\rightarrow\infty}\frac{\int|(-\Delta)^{\frac{s}{2}} u_m|^2 dx
				-C_{HLS}\left(\int(|x|^{-\mu} \ast|u_m|^{p_s})|u_m|^{p_s}dx\right)^{\frac{1}{p_s}}}{\dist(u_m,\mathcal{M})^2}\leq1.
		\end{equation}
	\end{lem}
	\begin{proof}
		Let $\dist(u_m,\mathcal{M}):=\inf_{ c\in\mathbb{R}, \lambda>0, \xi\in\mathbb{R}^n}\|u_m-cW[\xi,\lambda]\|_{\dot{H}^s(\mathbb{R}^n)}\rightarrow0$ as $m\rightarrow\infty$. Then for any $u_m\in\dot{H}^s(\mathbb{R}^n)$, $d_m$ can be attained by $(c_m,\lambda_m,\xi_m)\in\mathbb{R}\times\mathbb{R}^+\times\mathbb{R}^n$, that is
		\begin{equation*}
			d_m=\|u_m-c_mW[\xi_m,\lambda_m]\|_{\dot{H}^s(\mathbb{R}^n)}.
		\end{equation*}
		Form Lemma \ref{propep}, the tangential space at $(c_m,\lambda_m,\xi_m)$ is given by
		\begin{equation*}
			T_{c_m,\lambda_m,\xi_m}
			=\text{span}
			\left\{W[\xi_m,\lambda_m], \frac{\partial W[\xi_m,\lambda]}{\partial \lambda}\Big|_{\lambda=\lambda_m},\frac{\partial W[\xi,\lambda_m]}{\partial x_1}\Big|_{x=\xi_m},\ldots,\frac{\partial W[\xi,\lambda_m]}{\partial x_n}\Big|_{x=\xi_m}\right\}
		\end{equation*}
		and $u_m-c_mW[\xi_m,\lambda_m]$ is perpendicular to $T_{c_m,\lambda_m,\xi_m}$. In particular,
		\begin{equation*}		\int(-\Delta)^\frac{s}{2}W_m\cdot(-\Delta)^\frac{s}{2}(u_m-c_mW_m)=0.
		\end{equation*}
		In what follows, the notation $W_m$ will represent $W[\xi_m,\lambda_m]$ for simplicity. Therefore we may assume that $u_m=c_mW_m+d_mv_m$ and $v_m$ is perpendicular to $T_{c_m,\lambda_m,\xi_m}$ with $\|v_m\|_{\dot{H}^s(\mathbb{R}^n)}=1$. This imples that
		\begin{equation*}
			\|u_m\|_{\dot{H}^s(\mathbb{R}^n)}^2
			=c_m^2\|W_m\|_{\dot{H}^s(\mathbb{R}^n)}^2+d_m^2.
		\end{equation*}
		Since $p_s\geq2$ and the orthogonality from above, we get
		\begin{equation}\label{kuangquanshui}
			\begin{split}
				\int\big(|x|^{-\mu} \ast |u_m|^{p_s}\big)|u_m|^{p_s}
				= & c_m^{2\cdot p_s}\int\big(|x|^{-\mu} \ast W_m^{p_s}\big)W_m^{p_s}
				+ o(d_m^2)
				\\&
				+ p_s(p_s-1)c_m^{2(p_s-1)}d_m^2 \int\big(|x|^{-\mu} \ast W_m^{p_s}\big)W_m^{p_s-2}v_m^2    \\
				& + p_s^2c_m^{2(p_s-1)}d_m^2 \int\big(|x|^{-\mu} \ast W_m^{p_s-1}v_m\big)W_m^{p_s-1}v_m.
			\end{split}
		\end{equation}
		
		\par Denote
		\begin{equation*}
			\mathcal{P}_{m,1}
			:=\int\big(|x|^{-\mu} \ast W_m^{p_s-1}v_m\big)W_m^{p_s-1}v_m\hspace{4mm}\text{and}\hspace{4mm}\mathcal{P}_{m,2}
			:=\int\big(|x|^{-\mu} \ast W_m^{p_s}\big)W_m^{p_s-2}v_m^2.
		\end{equation*}
		In the following, the proof is divided into four cases.
		\par \textbf{Case 1 :} $\mathcal{P}_{m,1}=\mathcal{P}_{m,2}=o(1)$. By \eqref{kuangquanshui}, in this case we have
		\begin{equation*}
			\int\big(|x|^{-\mu} \ast |u_m|^{p_s}\big)|u_m|^{p_s}
			\leq c_m^{2\cdot p_s}\int_{\mathbb{R}^n}\big(|x|^{-\mu} \ast W_m^{p_s}\big)W_m^{p_s}
			+ o(d_m^2).
		\end{equation*}
		Thus it follows that
		\begin{equation*}
			\Big(\int\big(|x|^{-\mu} \ast |u_m|^{p_s}\big)|u_m|^{p_s}\Big)^\frac{1}{p_s}
			\leq \Big(c_m^{2\cdot p_s}\|W_m\|_{\dot{H}^s(\mathbb{R}^n)}^2+o(d_m^2)\Big)^\frac{1}{p_s}
			\leq c_m^2\|W_m\|_{\dot{H}^s(\mathbb{R}^n)}^\frac{2}{p_s}+o(d_m^2)
		\end{equation*}
		and
		\begin{equation*}
			\begin{split}
				\|u_m\|_{\dot{H}^s(\mathbb{R}^n)}^2
				-C_{HLS}\Big(\int\big(|x|^{-\mu} \ast |u_m|^{p_s}\big)|u_m|^{p_s}\Big)^\frac{1}{p_s}
				&\geq c_m^2\|W_m\|_{\dot{H}^s(\mathbb{R}^n)}^2+d_m^2-c_m^2C_{HLS}\|W_m\|_{\dot{H}^s(\mathbb{R}^n)}^\frac{2}{p_s}+o(d_m^2)     \\
				&\geq d_m^2+o(d_m^2),
			\end{split}
		\end{equation*}
		here we use the fact $C_{HLS}=\|W_m\|_{\dot{H}^s(\mathbb{R}^n)}^{2-\frac{2}{p_s}}$. Choosing $d_m$ small enough we have $1>C>0$ such that
		\begin{equation*}
			\|u_m\|_{\dot{H}^s(\mathbb{R}^n)}^2
			-C_{HLS}\Big(\int\big(|x|^{-\mu} \ast |u_m|^{p_s}\big)|u_m|^{p_s}\Big)^\frac{1}{p_s}
			\geq Cd_m^2.
		\end{equation*}
		This proves \eqref{caomei1}.
		
		\par \textbf{Case 2 :} $\mathcal{P}_{m,1}\geq \tilde{C}>0$, $\mathcal{P}_{m,2}\geq \tilde{C}>0$. By the HLS inequality, H\"older inequality and Sobolev inequality, we can obtain that
		\begin{equation*}
			\begin{split}
				\int\big(|x|^{-\mu} \ast W_m^{p_s}\big)W_m^{p_s-2}v_m^2
				&\leq C(n,\mu)\|W_m^{p_s-1}v_m\|_{L^\frac{2n}{2n-\mu}(\mathbb{R}^n)}^2  \\
				&\leq C(n,\mu)\|W_m\|_{L^{2_s^\ast}(\mathbb{R}^n)}^{2(p_s-1)}\|v_m\|_{\dot{H}^s(\mathbb{R}^n)}^2
				\leq\|v_m\|_{\dot{H}^s(\mathbb{R}^n)}^2
				\leq1
			\end{split}
		\end{equation*}
		and
		\begin{equation*}
			\int\big(|x|^{-\mu} \ast W_m^{p_s-1}v_m\big)W_m^{p_s-1}v_m
			\leq\|v_m\|_{\dot{H}^s(\mathbb{R}^n)}^2
			\leq1.
		\end{equation*}
		Furthermore, the definition of $\tilde{\lambda}_{n+3}$ implies that
		\begin{equation*}
			1
			\geq\tilde{\lambda}_{n+3}\int\big(|x|^{-\mu} \ast W_m^{p_s-1}v_m\big)W_m^{p_s-1}v_m
			+(\tilde{\lambda}_{n+3}-1)\int\big(|x|^{-\mu} \ast W_m^{p_s}\big)W_m^{p_s-2}v_m^2.
		\end{equation*}
		Then from \eqref{kuangquanshui} and $\tilde{\lambda}_{n+3}>p_s$ we can derive
		\begin{equation*}
			\begin{split}
				\int\big(|x|^{-\mu} \ast |u_m|^{p_s}\big)|u_m|^{p_s}
				\leq& c_m^{2\cdot p_s}\int\big(|x|^{-\mu} \ast W_m^{p_s}\big)W_m^{p_s}
				+ o(d_m^2)      \\
				&+p_sc_m^{2(p_s-1)}d_m^2(p_s-\tilde{\lambda}_{n+3})\Big[\int\big(|x|^{-\mu} \ast W_m^{p_s}\big)W_m^{p_s-2}v_m^2        \\
				&+\int\big(|x|^{-\mu} \ast W_m^{p_s-1}v_m\big)W_m^{p_s-1}v_m \Big]          \\
				&+p_sc_m^{2(p_s-1)}d_m^2\Big[\tilde{\lambda}_{n+3}\int\big(|x|^{-\mu} \ast W_m^{p_s-1}v_m\big)W_m^{p_s-1}v_m        \\
				&+(\tilde{\lambda}_{n+3}-1)\int\big(|x|^{-\mu} \ast W_m^{p_s}\big)W_m^{p_s-2}v_m^2\Big]     \\
				\leq&c_m^{2\cdot p_s}\|W_m\|_{\dot{H}^s(\mathbb{R}^n)}^2				+p_sc_m^{2(p_s-1)}d_m^2\big[2\tilde{C}(p_s-\tilde{\lambda}_{n+3})+1\big]+ o(d_m^2).
			\end{split}
		\end{equation*}
		Thus we have
		\begin{equation*}
			\begin{split}
				\Big(\int\big(|x|^{-\mu} \ast |u_m|^{p_s}\big)|u_m|^{p_s}\Big)^\frac{1}{p_s}
				&\leq c_m^2\Big(\|W_m\|_{\dot{H}^s(\mathbb{R}^n)}^2+p_sc_m^{-2}d_m^2\big[2\tilde{C}(p_s-\tilde{\lambda}_{n+3})+1\big]\Big)^\frac{1}{p_s}
				+o(d_m^2)      \\
				&\leq c_m^2\|W_m\|_{\dot{H}^s(\mathbb{R}^n)}^\frac{2}{p_s}
				+d_m^2[2\tilde{C}(p_s-\tilde{\lambda}_{n+3})+1\big]\|W_m\|_{\dot{H}^s(\mathbb{R}^n)}^{\frac{2}{p_s}-2}
				+o(d_m^2),
			\end{split}
		\end{equation*}
		which leads to
		\begin{equation*}
			\begin{split}
				&\|u_m\|_{\dot{H}^s(\mathbb{R}^n)}^2
				-C_{HLS}\Big(\int\big(|x|^{-\mu} \ast |u_m|^{p_s}\big)|u_m|^{p_s}\Big)^\frac{1}{p_s}        \\
				\geq&c_m^2\|W_m\|_{\dot{H}^s(\mathbb{R}^n)}^2+d_m^2
				-C_{HLS}\Bigl\{c_m^2\|W_m\|_{\dot{H}^s(\mathbb{R}^n)}^\frac{2}{p_s}
				+d_m^2[2\tilde{C}(p_s-\tilde{\lambda}_{n+3})+1\big]\|W_m\|_{\dot{H}^s(\mathbb{R}^n)}^{\frac{2}{p_s}-2}
				+o(d_m^2)\Bigr\}     \\
				=&d_m^2\Big(1-[2\tilde{C}(p_s-\tilde{\lambda}_{n+3})+1\big]C_{HLS}\|W_m\|_{\dot{H}^s(\mathbb{R}^n)}^{\frac{2}{p_s}-2}\Big)
				+c_m^2\|W_m\|_{\dot{H}^s(\mathbb{R}^n)}^2\Big(1-C_{HLS}\|W_m\|_{\dot{H}^s(\mathbb{R}^n)}^{\frac{2}{p_s}-2}\Big)    \\
				=&2\tilde{C}(\tilde{\lambda}_{n+3}-p_s)d_m^2+o(d_m^2).
			\end{split}
		\end{equation*}
		By choosing $\tilde{C}$ small enough such that $2\tilde{C}(\tilde{\lambda}_{n+3}-p_s)<1$, we can conclude \eqref{caomei1} for sufficiently small $d_m$.
		
		\par \textbf{Case 3 :} $\mathcal{P}_{m,1}=o(1)$, $\mathcal{P}_{m,2}\geq C>0$. By applying the same reasoning as in Case 2, we obtain \eqref{caomei1}.
		
		\par \textbf{Case 4 :} $\mathcal{P}_{m,1}\geq C>0$, $\mathcal{P}_{m,2}=o(1)$. By applying the same reasoning as in Case 2, we obtain \eqref{caomei1}.\\
		Combining all cases, we have proved \eqref{caomei1}.
		
		\par Next we show that \eqref{caomei2} holds true. Thanks to \eqref{kuangquanshui}, we know
		\begin{equation*}
			\bigg(\int\Big(|x|^{-\mu} \ast |u_m|^{p_s}\Big)|u_m|^{p_s} dx\bigg)^\frac{1}{p_s}
			\geq c_m^2\|W_m\|_{\dot{H}^s(\mathbb{R}^n)}^\frac{2}{p_s}
			+ o(d_m^2).
		\end{equation*}
		Therefore,
		\begin{equation*}
			\begin{split}
				&\|u_m\|_{\dot{H}^s(\mathbb{R}^n)}^2
				-C_{HLS}\bigg(\int\Big(|x|^{-\mu} \ast |u_m|^{p_s}\Big)|u_m|^{p_s} dx\bigg)^\frac{1}{p_s}     \\
				\leq& d_m^2+o(d_m^2)				+c_m^2\|W_m\|_{\dot{H}^s(\mathbb{R}^n)}^2\bigg(1-C_{HLS}\|W_m\|_{\dot{H}^s(\mathbb{R}^n)}^{\frac{2}{p_s}-2}\bigg)
				=\big(1+o(1)\big)d_m^2,
			\end{split}
		\end{equation*}
		which proves \eqref{caomei2}.
	\end{proof}
	
	\begin{proof}[Proof of Theorem \ref{remainder terms}]
		By contradiction, we may assume that there exists a sequence $\{u_m\}\subset\dot{H}^s(\mathbb{R}^n)\setminus\mathcal{M}$ such that
		\begin{equation}\label{chelizi1}
			\frac{\int|(-\Delta)^{\frac{s}{2}} u_m|^2 dx
				-C_{HLS}\left(\int(|x|^{-\mu} \ast|u_m|^{p_s})|u_m|^{p_s}dx\right)^{\frac{1}{p_s}}}{\dist(u_m,\mathcal{M})^2}\rightarrow+\infty
		\end{equation}
		or
		\begin{equation}\label{chelizi2}
			\frac{\int|(-\Delta)^{\frac{s}{2}} u_m|^2 dx
				-C_{HLS}\left(\int(|x|^{-\mu} \ast|u_m|^{p_s})|u_m|^{p_s}dx\right)^{\frac{1}{p_s}}}{\dist(u_m,\mathcal{M})^2}\rightarrow0.
		\end{equation}
		By homogeneity, we can assume that $\|u_m\|_{\dot{H}^s(\mathbb{R}^n)}=1$. Since
		\begin{equation*}
			\dist(u_m,\mathcal{M})
			=\inf_{ c\in\mathbb{R}, \lambda>0, \xi\in\mathbb{R}^n}\|u_m-cW[\xi,\lambda]\|_{\dot{H}^s(\mathbb{R}^n)}
			\leq\|u_m\|_{\dot{H}^s(\mathbb{R}^n)}=1,
		\end{equation*}
		there exists a subsequence such that $\dist(u_m,\mathcal{M})\rightarrow\varpi\in[0,1]$.
		
		\par If \eqref{chelizi1} holds, it necessarily follows that $\varpi=0$ which contradicts Lemma \ref{baowenbei}. Moreover, if \eqref{chelizi2} holds, it also leads to a contradiction with Lemma \ref{baowenbei} when $\varpi=0$. Consequently, the only remaining possibility is that \eqref{chelizi2} holds and $0<\varpi\leq1$, that is
		\begin{equation*}
			\dist(u_m,\mathcal{M})\rightarrow\varpi>0\hspace{4mm}\text{and}\hspace{4mm}\|u_m\|_{\dot{H}^s(\mathbb{R}^n)}^2
			-C_{HLS}\bigg(\int\Big(|x|^{-\mu} \ast |u_m|^{p_s}\Big)|u_m|^{p_s} dx\bigg)^\frac{1}{p_s}\rightarrow0.
		\end{equation*}
		Then we must have
		\begin{equation*}
			\bigg(\int_{\mathbb{R}^n}\Big(|x|^{-\mu} \ast |u_m|^{p_s}\Big)|u_m|^{p_s} dx\bigg)^\frac{1}{p_s}\rightarrow\frac{1}{C_{HLS}}\hspace{4mm}\text{and}\hspace{4mm}\|u_m\|_{\dot{H}^s(\mathbb{R}^n)}=1.
		\end{equation*}
		By Lions's concentration-compactness principle \cite{Lion85-1,Lion85-2}, there exist $\xi_m\in\mathbb{R}^n$ and $\lambda_m\in\mathbb{R}^+$ such that
		\begin{equation*}
			\lambda_m^\frac{n-2s}{2}u_m\big(\lambda_m(x-\xi_m)\big)\rightarrow W_0\hspace{4mm}\text{in}\hspace{2mm}\dot{H}^s(\mathbb{R}^n),\hspace{4mm}\text{as}\hspace{2mm}m\rightarrow\infty,
		\end{equation*}
		for some $W_0\in\mathcal{M}$, and
		\begin{equation*}
			\dist(u_m,\mathcal{M})
			=\dist\bigg(	\lambda_m^\frac{n-2s}{2}u_m\big(\lambda_m(x-\xi_m)\big),\mathcal{M}\bigg)
			\rightarrow0\hspace{4mm}\text{as}\hspace{2mm}m\rightarrow\infty.
		\end{equation*}
		This is impossible.
	\end{proof}
	
	\subsection{Proof of Theorem \ref{thmprtb}}
	Theorem \ref{thmprtb} follows immediately from the Proposition 4.1 in \cite{CFW13}.
	\begin{Prop}\label{shanzhashu}
		There exists a constant $C_0$ depending only on $n$ and $s\in(0,\frac{n}{2})$ such that
		\begin{equation*}
			\|u\|_{L_w^\frac{n}{n-2s}(\Omega)}
			\leq C_0|\Omega|^\frac{1}{2_s^\ast}d(u,\mathcal{M})
		\end{equation*}
		for all subdomains $\Omega\subset\mathbb{R}^n$ with $|\Omega|<\infty$ and all $u\in\dot{H}^s(\Omega)$.
	\end{Prop}
	\begin{proof}[Proof of Theorem \ref{thmprtb}]
		The conclusion now directly follows  from the combination of Theorem \ref{remainder terms} and Proposition \ref{shanzhashu}.
	\end{proof}
	
	\section{A nonlocal version of profile decompositions}\label{section4}
	Building upon the renowned profile decomposition results of Struwe in \cite{Struwe-1984} for the equation (\ref{bec}) with $s=1$, we draw inspiration and leverage these findings to establish a profile decomposition for nonnegative solutions of the nonlocal Hartree-type equation \eqref{ele-1.1}.
	\begin{proof}[Proof of Theorem \ref{emm}]
		Since
		\begin{equation}\label{UM}
			(\kappa-\frac{1}{2})C_{HLS}^{\frac{2n-\mu}{n+2s-\mu}}\leq\big\|u_m\big\|_{\dot{H}^{s}(\mathbb{R}^n)}^2\leq(\kappa+\frac{1}{2})C_{HLS}^{\frac{2n-\mu}{n+2s-\mu}},
		\end{equation}
		the sequence $\{u_m\}$ is bounded in $\dot{H}^{s}(\mathbb{R}^n)$.
		Then there exist functions $\{\psi_j\}_{j\in\mathbb{N}}\subset\dot{H}^{s}(\mathbb{R}^n)$, sequences $\{h_m^{(j)}\}_{j\in\mathbb{N}}\subset(0,\infty)$ and points $\{x_m^{(j)}\}_{j\in\mathbb{N}}\subset\mathbb{R}^n$, such that for a renumbered subsequence $\{u_m\}$, there holds
		\begin{equation}\label{Q1}
			\Big|\log\big(\frac{h_m^{(i)}}{h_m^{(j)}}\big)\Big|+\Big|\frac{x_{m}^{(i)}-x_{m}^{(j)}}{h_m^{(i)}}\Big|\rightarrow\infty\hspace{6mm}\mbox{as}\hspace{2mm}m\rightarrow\infty,\hspace{2mm}\mbox{for all}\hspace{2mm}i\neq j,
		\end{equation}
		\begin{equation}\label{Q2}
			u_{m}=\sum_{j=1}^l(h_{m}^{(j)})^{-\frac{n-2s}{2}}\psi_j\big(\frac{x-x_{m}^{(j)}}{h_{m}^{(j)}}\big)+r_{m}^{(l)},
		\end{equation}
		where
		\begin{equation}\label{Q3}
			\begin{split}
				&\limsup \limits_{m\rightarrow\infty}\big\|r_{m}^{(l)}\big\|_{L^{2_{s}^{\ast}}(\mathbb{R}^n)}\rightarrow0\hspace{6mm}\mbox{as}\hspace{2mm}l\rightarrow\infty,\\&
				\big\|u_{m}\big\|_{\dot{H}^s(\mathbb{R}^n)}^{2}=\sum_{j=1}^l\big\|\psi_{j}\big\|_{\dot{H}^s(\mathbb{R}^n)}^{2}+\big\|r_{m}^{(l)}\big\|_{\dot{H}^s(\mathbb{R}^n)}^{2} \hspace{6mm}\mbox{as}\hspace{2mm}l\rightarrow\infty,
			\end{split}
		\end{equation}
		(see e.g. \cite[Theorem 1.1]{GGG}).
		Due to $\|u_m\|_{\dot{H}^{s}(\mathbb{R}^n)}$ is bounded from below by \eqref{UM}, then we may assume that, up to suitable rescaling, $\psi_1\not\equiv0$. By the definition of profile, for fixed $j\in J$ (where set of indices $J\subset\mathbb{N}$) there exist sequences $\{h_m^{(j)}\}_{j\in\mathbb{N}}\subset(0,\infty)$ and $\{x_m^{(j)}\}_{j\in\mathbb{N}}\subset\mathbb{R}^n$ such that $u_m(h_m^{(j)}\cdot+x_m^{(j)})\rightharpoonup \psi_j(\cdot)$ in $\dot{H}^{s}(\mathbb{R}^n)$ as $m\rightarrow\infty$. Thus, $\psi_1$ can be regarded as the profile corresponding to $h_m^{(1)}=1$ and $x_m^{(1)}=0$.
		
		Next, we are going to prove that if the convergence is not strong then $u_m$ contains further profiles (except possibly $\psi_1$). We prove this result by contradiction. Assuming that the statement is false and following the same arguments as those in \cite[Proposition 1]{PP14}, we obtain $u_m$ converges strongly to $\psi_1$ in $L^{2_s^{\ast}}(\mathbb{R}^n)$. Moreover, in view of weak limit and the fact
		\begin{equation}\label{Wm}
			\Big\|(-\Delta)^{s}u_m-\left(|x|^{-\mu}\ast |u_m|^{p_{s}}\right)|u_m|^{p_{s}-2}u_m\Big\|_{\dot{H}^{-s}(\mathbb{R}^n)}\rightarrow0\quad \mbox{as}\quad m\rightarrow\infty,
		\end{equation}
		some elementary estimates give us that
		\begin{equation*}
			\begin{split}
				\big\|u_m-\psi_1\big\|_{\dot{H}^s(\mathbb{R}^n)}^2&=\big\langle(-\Delta)^s u_m,~u_m-\psi_1\big\rangle-\big\langle(-\Delta)^s\psi_1,~u_m-\psi_1\big\rangle\\&
				=\big\langle(-\Delta)^{s}u_m-\left(|x|^{-\mu}\ast |u_m|^{p_{s}}\right)|u_m|^{p_{s}-2}u_m,~u_m-\psi_1\big\rangle\\&-\big\langle(-\Delta)^{s}\psi_1-\left(|x|^{-\mu}\ast |\psi_1|^{p_{s}}\right)|\psi_1|^{p_{s}-2}\psi_1,~u_m-\psi_1\big\rangle\\&
				+\int\big[\left(|x|^{-\mu}\ast |u_m|^{p_{s}}\right)|u_m|^{p_{s}-2}u_m-\left(|x|^{-\mu}\ast |\psi_1|^{p_{s}}\right)|\psi_1|^{p_{s}-2}\psi_1\big]\big(u_m-\psi_1\big)=o(1),
			\end{split}
		\end{equation*}
		which is a contradiction. Here $\langle u,v\rangle$ denotes the dual bracket between $\dot{H}^s(\mathbb{R}^n)$ and $\dot{H}^{-s}(\mathbb{R}^n)$ for any $v,u\in\dot{H}^s(\mathbb{R}^n)$.
		
		Next, we claim that there holds
		\begin{equation}\label{FJ}
			(-\Delta)^s \psi_j-(|x|^{-\mu}\ast |\psi_j|^{p_s})|\psi_j|^{p_s-2}\psi_j=0,\quad \mbox{for every}~~\psi_j\in \dot{H}^s(\mathbb{R}^n),
		\end{equation}
		in the weak sense. Indeed, for $\phi\in \dot{H}^{s}(\mathbb{R}^n)$,
		an easy change of variables with the test function $\phi_m^{(j)}:=(h_{m}^{(j)})^{-\frac{n-2s}{2}}\phi\big((x-x_m^{(j)})/h_{m}^{(j)}\big)$, together with the asymptotic orthogonality of the scaled
		profiles \eqref{Q1} and \eqref{Wm}, the scaling invariance of $\dot{H}^s(\mathbb{R}^n)$ norm and a direct computation, imply that
		\begin{equation*}
			\begin{split}
				o(1)=&\big\langle(-\Delta)^s u_m-(|x|^{-\mu}\ast |u_m|^{p_s})|u_m|^{p_s-2}u_m, \phi_m^{(j)}\big\rangle\\=&
				\big\langle(-\Delta)^s \Big((h_{m}^{(j)})^{-\frac{n-2s}{2}}\psi_j(\frac{x-x_k^{(j)}}{h_m^{(j)}})\Big)\\&-(|x|^{-\mu}\ast \big|(h_{m}^{(j)})^{-\frac{n-2s}{2}}\psi_j(\frac{x-x_k^{(j)}}{h_m^{(j)}})\big|^{p_s})\big|(h_{m}^{(j)})^{-\frac{n-2s}{2}}\psi_j(\frac{x-x_k^{(j)}}{h_m^{(j)}})\big|^{p_s-2}(h_{m}^{(j)})^{-\frac{n-2s}{2}}\psi_j(\frac{x-x_k^{(j)}}{h_m^{(j)}}),~\phi_m^{(j)}\big\rangle+o(1)\\
				=&\big\langle(-\Delta)^s \psi_j-(|x|^{-\mu}\ast |\psi_j|^{p_s})|\psi_j|^{p_s-2}\psi_j,~\phi\big\rangle+o(1).
			\end{split}
		\end{equation*}
		Thus, one has $\big\langle(-\Delta)^s \psi_j-(|x|^{-\mu}\ast |\psi_j|^{p_s})|\psi_j|^{p_s-2}\psi_j,~\phi\big\rangle=0$ for every $\phi\in \dot{H}^{s}(\mathbb{R}^n)$ and the claim follows.
		Furthermore, by the nonlocal Sobolev inequality \eqref{Prm} and the fact $\|\psi_j\|_{\dot{H}^s(\mathbb{R}^n)}^2=\int(|x|^{-\mu} \ast|\psi_j|^{p_s})|\psi_j|^{p_s}dx$, we obtain
		\begin{equation*}
			\begin{split}		
				C_{HLS}\leq\frac{\|\psi_j\|_{\dot{H}^s(\mathbb{R}^n)}^2}{\left(\int(|x|^{-\mu} \ast|\psi_j|^{p_s})|\psi_j|^{p_s}dx\right)^{\frac{1}{p_s}}}=\big\|\psi_j\big\|_{\dot{H}^s(\mathbb{R}^n)}^{2(1-\frac{1}{p_s})}, \quad \mbox{i.e.}~~\big\|\psi_j\big\|_{\dot{H}^s(\mathbb{R}^n)}^{2}\geq C_{HLS}^{\frac{2n-\mu}{n+2s-\mu}}.
			\end{split}
		\end{equation*}
		As a consequence, the profiles $\psi_j$ are in finite number $l=\kappa$ by \eqref{UM} and \eqref{Q3}. It follows from \eqref{Q2}-\eqref{Q3} that
		\begin{equation}\label{Q22}
			\big\|u_{m}-\sum_{j=1}^{\kappa}(h_{m}^{(j)})^{-\frac{n-2s}{2}}\psi_j\big(\frac{x-x_{m}^{(j)}}{h_{m}^{(j)}}\big)\big\|_{L^{2_{s}^{\ast}}(\mathbb{R}^n)}\rightarrow0\hspace{6mm}\mbox{as}\hspace{2mm}m\rightarrow\infty.
		\end{equation}
		By the nonnegativity of functions $u_m$, one can easily check that $\psi_j\geq0$. Since $W$ is the unique nonnegative solution of \eqref{ele-1.1} in $\dot{H}^s(\mathbb{R}^n)$ for $s\in(0,\frac{n}{2})$, then we have $\psi_j=W[\xi_j^{(m)},\lambda_j^{(m)}]$ for $\lambda_j^{(m)}\in\mathbb{R}^{+}$ and $\xi_j^{(m)}\in\mathbb{R}^n$.
		
		Finally, we are going to show that the sequence $\{r_m^{(\kappa)}\}$ converges strongly to $0$ in $\dot{H}^s(\mathbb{R}^n)$. Introducing $v_m=\sum_{j=1}^{\kappa}(h_{m}^{(j)})^{-\frac{n-2s}{2}}\psi_j\big((x-x_{m}^{(j)})/h_{m}^{(j)}\big)$, then it follows from \eqref{Wm} and \eqref{Q22} that
		\begin{equation*}
			\begin{split}
				\big\|u_{m}-v_m\big\|_{\dot{H}^{s}(\mathbb{R}^n)}^2
				=&\big\langle(-\Delta)^su_{m},~u_{m}-v_m\big\rangle-\big\langle(-\Delta)^sv_m,~u_{m}-v_m\big\rangle\\
				=&\big\langle(|x|^{-\mu}\ast u_{m}^{p_s})u_{m}^{p_s-1},~u_{m}-v_m\big\rangle-\big\langle(|x|^{-\mu}\ast v_{m}^{p_s})v_{m}^{p_s-1},~u_{m}-v_m\big\rangle  \\&+O\Big(\big\|(-\Delta)^{s}u_m-\left(|x|^{-\mu}\ast u_m^{p_{s}}\right)u_m^{p_{s}-1}\big\|_{\dot{H}^{-s}(\mathbb{R}^n)}\big\|u_{m}-v_m\big\|_{\dot{H}^{s}(\mathbb{R}^n)}\Big)
				\\&+O\Big(\big\|(-\Delta)^{s}v_m-\left(|x|^{-\mu}\ast v_m^{p_{s}}\right)v_m^{p_{s}-1}\big\|_{\dot{H}^{-s}(\mathbb{R}^n)}\big\|u_{m}-v_m\big\|_{\dot{H}^{s}(\mathbb{R}^n)}\Big)\\
				=&\big\langle(|x|^{-\mu}\ast u_{m}^{p_s})u_{m}^{p_s-1},~u_{m}-v_m\big\rangle-\big\langle(|x|^{-\mu}\ast v_{m}^{p_s})v_{m}^{p_s-1},~u_{m}-v_m\big\rangle+o(1)=o(1).
			\end{split}
		\end{equation*}
		Here we also used the fact that the boundedness of sequences $u_m$ and $v_m$ in $L^{2_{s}^{\ast}}(\mathbb{R}^n)$ and some elementary inequalities. This concludes the proof.
	\end{proof}

	\section{The existence of first approximation for $n=6s,~\mu\in(0,4s)$ or $n\neq6s,~\mu\in(0,4s]$}	\label{sangshen}
	The proof of Theorem \ref{Figalli} follows by adapting the strategy outlined in \cite{DSW21,C-K-L-24,DHP}. In order to establish the existence of first approximation, we divide the proof into two cases. In this section, we consider the case when $n=6s,\mu\in(0,4s)$ or $n\neq6s,\mu\in(0,4s]$. The remaining case, $n=6s$ and $\mu=4s$, will be addressed in the next section. Firstly, let the error between $u$ and the best approximation $\sigma=\sum_{i=1}^{\kappa}W_i$ denoted by $\varrho$, i.e. $u=\sigma+\varrho.$ We begin from the following decomposition:
	\begin{equation}\label{u-0}
		(-\Delta)^s \phi-I_{n,\mu,s}[\sigma,\phi]-g-N(\phi)-(-\Delta)^s u+\big(|x|^{-\mu}\ast |u|^{p_s}\big)|u|^{p_s-2}u=0,
	\end{equation}
	where
	\begin{equation}\label{I-FAI-1}
		I_{n,\mu,s}[\sigma,\phi]:=p_s\Big(|x|^{-\mu}\ast \sigma^{p_s-1}\phi\Big)
		\sigma^{p_s-1}+(p_s-1)\Big(|x|^{-\mu}\ast\sigma^{p_s}\Big)
		\sigma^{p_s-2}\phi,
	\end{equation}
	\begin{equation}\label{u-1}
		\begin{split}
			g:=\Big(|x|^{-\mu}\ast\sigma^{p_s}\Big)\sigma^{p_s-1}-\sum_{i=1}^{\kappa}\Big(|x|^{-\mu}\ast W_{i}^{p_s}\Big)W_{i}^{p_s-1},
		\end{split}
	\end{equation}
	\begin{equation}\label{u-2}
		\begin{split}
			N(\phi)
			&:=\Big(|x|^{-\mu}\ast(\sigma+\phi)^{p_s}\Big)(\sigma+\phi)^{p_s-1}
			-\Big(|x|^{-\mu}\ast\sigma^{p_s}\Big)\sigma^{p_s-1}\\&
			~~~-p_s
			\Big(|x|^{-\mu}\ast\sigma^{p_s-1}\phi\Big)
			\sigma^{p_s-1}-(p_s-1)
			\Big(|x|^{-\mu}\ast\sigma^{p_s}\Big)
			\sigma^{p_s-2}\phi.
		\end{split}
	\end{equation}
	It is noticing that \eqref{w-tittle} implies $\|\varrho\|_{\dot{H}^{s}(\mathbb{R}^n)}\leq\delta$. We further decompose $\varrho=\varrho_0+\varrho_1$ and then prove the existence of first approximation $\varrho_0$, which solves the following system
	\begin{equation*}
		\left\{\begin{array}{l}
			\displaystyle (-\Delta)^s \varrho_0-\Big[\Big(|x|^{-\mu}\ast(\sigma+\varrho_0)^{p_s}\Big)(\sigma+\varrho_0)^{p_s-1}
			-\Big(|x|^{-\mu}\ast\sigma^{p_s}\Big)\sigma^{p_s-1}\Big]\\ =\Big(|x|^{-\mu}\ast\sigma^{p_s}\Big)\sigma^{p_s-1}-\sum_{i=1}^{\kappa}\Big(|x|^{-\mu}\ast W_{i}^{p_s}\Big)W_{i}^{p_s-1}+	\sum_{i=1}^{\kappa}\sum_{a=1}^{n+1}c_{a}^{i}I_{n,\mu,s}[W_{i},\mathcal{Z}^{a}_i]\hspace{4mm}\mbox{in}\hspace{2mm} \mathbb{R}^n,\\
			\varrho_0\in \dot{H}^{s}(\mathbb{R}^n), \hspace{2mm}c_{a}^{1},\cdots,c_{a}^{n+1}\in\mathbb{R},\\
			\displaystyle 	\int I_{n,\mu,s}[W_{i},\mathcal{Z}^{a}_i]\varrho_0=0,\hspace{4mm}i=1,\cdots, \kappa; ~a=1,\cdots,n+1,
		\end{array}
		\right.
	\end{equation*}
	where $\{c_a^i\}$ is a family of scalars and $\mathcal{Z}_i^a$ are the rescaled derivative of $W[\xi_i,\lambda_i]$ defined as follows:
	\begin{equation}\label{qta}
		\begin{split}
			&\mathcal{Z}_i^a=\frac{1}{\lambda_i}\frac{\partial W[\xi,\lambda_i]}{\partial \xi^{a}}\Big|_{\xi=\xi_i}=(2s-n)W[\xi_i,\lambda_i] \frac{\lambda_i(\cdot^{a}-\xi^a)}{1+\lambda^2|\cdot-x|^2}, \hspace{3mm}\mbox{for}\hspace{2mm}i=1,\cdots,\kappa,\\&
			\mathcal{Z}_{i}^{n+1}=\lambda_{i}\frac{\partial W[\xi_{i},\lambda]}{\partial \lambda}\Big|_{\lambda=\lambda_i}=\frac{n-2s}{2}W[\xi_i,\lambda_i]
			\frac{1-\lambda_i^2|\cdot-x|^2}{1+\lambda_i^2|\cdot-x|^2},\hspace{2mm}\mbox{for}\hspace{2mm}i=1,\cdots,\kappa.
		\end{split}
	\end{equation}
	Here $\xi^{a}$ is the $a$-th component of $\xi$ for $a=1,\cdots,n$.
	
	\begin{Def}\label{del-1}
		Let $W_i$ and $W_j$ be two bubbles, if $\mathscr{R}_{ij}=\sqrt{\lambda_i\lambda_j}|\xi_i-\xi_j|$, then we call them a bubble cluster, otherwise call them a bubble tower.
		We also set
		\begin{equation}\label{R1}
			\mathscr{R}_{ij}=\max\Big\{\sqrt{\lambda_i/\lambda_j}, \sqrt{\lambda_j/\lambda_i}, \sqrt{\lambda_i\lambda_j}|\xi_i-\xi_j|\Big\}\quad\text{if}\quad i\neq j\in I,
		\end{equation}
		and
		$$\mathscr{R}:=\frac{1}{2}\min\limits_{i\neq j}\big\{\mathscr{R}_{ij}:~i,j=1,\cdots,\kappa,~i\neq j\big\}.$$
		Furthermore, we denote
		\begin{equation*}
			\mathscr{Q}:=\max\big\{Q_{ij}(\xi_i,\xi_j,\lambda_i,\lambda_j):~~ i,~j=1,\cdots,\kappa\big\}\leq\delta,
		\end{equation*}
		where $Q_{ij}$ can be found in \eqref{mianbao}.
	\end{Def}

	For all $i=1,\cdots, \kappa; ~a=1,\cdots,n+1$, consider the elliptic problem with the linear operator
	\begin{equation}\label{coefficients1}
		\left\{\begin{array}{l}
			\displaystyle (-\Delta)^s \phi-I_{n,\mu,s}[\sigma,\phi]=g+	\sum_{i=1}^{\kappa}\sum_{a=1}^{n+1}c_{a}^{i}I_{n,\mu,s}[W_{i},\mathcal{Z}^{a}_i]\hspace{4mm}\mbox{in}\hspace{2mm} \mathbb{R}^n,\\
			\phi\in \dot{H}^{s}(\mathbb{R}^n), \hspace{2mm}c_{a}^{1},\cdots,\hspace{2mm}c_{a}^{n+1}\in\mathbb{R},\\
			\displaystyle 	\int I_{n,\mu,s}[W_{i},\mathcal{Z}^{a}_i]\phi=0,\hspace{4mm}i=1,\cdots, \kappa; ~a=1,\cdots,n+1.
		\end{array}
		\right.
	\end{equation}
	In order to conclude the proof of Theorem \ref{Figalli}, we establish the following key lemma.
	
	\begin{lem}\label{qiegao}
		Let $n=6s$ and $0<\mu<4s$ or $n\neq6s$ and $0<\mu\leq4s$. There exists a large constant $C=C(n,\kappa,\mu,s)$ such that
		\begin{equation}
			\Big\|\Big(|x|^{-\mu}\ast\sigma^{p_s}\Big)\sigma^{p_s-1}-\sum_{i=1}^{\kappa}\Big(|x|^{-\mu}\ast W_{i}^{p_s}\Big)W_{i}^{p_s-1}\Big\|_{L^{(2_s^{\ast})^{\prime}}(\mathbb{R}^n)}\leq C\mathscr{Q}^{\min\{\frac{\mu}{n-2s},1\}},
		\end{equation}
		where we denote by $(2_s^{\ast})^{\prime}=\frac{2n}{n+2s}$ the H\"{o}lder conjugate of $2_s^{\ast}$ and $C$ depends only on $n$, $s$, $\kappa$ and $\mu$.
	\end{lem}
	\begin{proof}
		The only difference is that the exponents have been modified by the parameter $s$. Thus one can follow the same proof as in \cite[Lemma 4.2]{DHP} and we omit it.
	\end{proof}
	
	\begin{lem}\label{wanfan}
		Assume that $n=6s$ and $0<\mu<4s$ or $n\neq6s$ and $0<\mu\leq4s$. Let $\phi$, $g$ and $c_b^j$ satisfy the system \eqref{coefficients1} and  $\sigma=\sum_{i=1}^{\kappa}W_i$ is a family of $\delta$-interacting bubbles, then there holds\\
		$$|c_b^j|\lesssim\|g\|_{L^{(2_s^{\ast})^{\prime}}(\mathbb{R}^n)}+\mathscr{Q}^{\min\{\frac{\mu}{n-2s},1\}}\|\phi\|_{\dot{H}^s(\mathbb{R}^n)},\hspace{3mm}j=1,\cdots, \kappa~\text{and}\hspace{2mm}b=1,\cdots,n+1.
		$$
	\end{lem}
	\begin{proof}
		Multiplying \eqref{coefficients1} by $\mathcal{Z}_j^b$ and integrating by parts we get
		\begin{equation}\label{pingjiehuoguo}
			\int I_{n,\mu,s}[\sigma,\phi]\mathcal{Z}_j^b+\int g\mathcal{Z}_j^b+	\sum_{i=1}^{\kappa}\sum_{a=1}^{n+1}c_{a}^{i}\int I_{n,\mu,s}[W_{i},\mathcal{Z}^{a}_i]\mathcal{Z}_j^b=0,
		\end{equation}
		for any $1\leq j\leq n$, $1\leq b\leq n+1$. Here we use the orthogonal condition in \eqref{coefficients1}.
		
		\par By Lemma \ref{armidale}, for $1\leq a,b\leq n+1$, there exist some constants $\gamma^b>0$ such that
		\begin{equation*}
			\sum_{i=1}^{\kappa}\sum_{a=1}^{n+1}c_{a}^{i}\int I_{n,\mu,s}[W_{i},\mathcal{Z}^{a}_i]\mathcal{Z}_j^b+c_b^j\gamma^b+\sum_{i\neq j}\sum_{a=1}^{n+1}c_{a}^{i}O(Q_{ij})=0.
		\end{equation*}
		Plugging in the above estimates to \eqref{pingjiehuoguo}, we see that $\{c_b^j\}$ satisfies the linear system
		\begin{equation}\label{kaihui}
			c_b^j\gamma^b+\sum_{i\neq j}\sum_{a=1}^{n+1}c_{a}^{i}O(Q_{ij})
			=\int I_{n,\mu,s}[\sigma,\phi]\mathcal{Z}_j^b+\int g\mathcal{Z}_j^b.
		\end{equation}
		Denote $\vec{c}^j := (c_1^j, \cdots, c_{n+1}^j) \in \mathbb{R}^{n+1}$ for $j = 1, \cdots, \kappa$. We concatenate these vectors to $\vec{c} = (\vec{c}^1, \cdots, \vec{c}^\kappa) \in \mathbb{R}^{\kappa(n+1)}$ and think of the above equations as a linear system on $\vec{c}$. Since $Q_{ij} \leq \mathscr{Q}\leq \delta$, the coefficient matrix is diagonally dominant and hence solvable. It remains to estimate the terms on the right-hand side.
		
		\par For each $j$ and $b$, by the orthogonal condition in \eqref{coefficients1} again, we have
		\begin{equation*}
			\int I_{n,\mu,s}[\sigma,\phi]\mathcal{Z}_j^b
			=\int \bigg(p_s\big(|x|^{-\mu}\ast\sigma^{p_s-1}\phi\big)\sigma^{p_s-1}+(p_s-1)\big(|x|^{-\mu}\ast\sigma^{p_s}\big)\sigma^{p_s-2}\phi\bigg)\mathcal{Z}_j^b=:J_1+J_2,
		\end{equation*}
		where
		\begin{equation*}
			J_1=p_s\int \big(|x|^{-\mu}\ast\sigma^{p_s-1} \mathcal{Z}_j^b\big)\sigma^{p_s-1}\phi-\big(|x|^{-\mu}\ast W_j^{p_s-1}\mathcal{Z}_j^b\big)W_j^{p_s-1}\phi,
		\end{equation*}
		\begin{equation*}
			J_2=(p_s-1)\int\big(|x|^{-\mu}\ast\sigma^{p_s}\big)\sigma^{p_s-2}\mathcal{Z}_j^b\phi-\big(|x|^{-\mu}\ast W_j^{p_s}\big)W_j^{p_s-2}\mathcal{Z}_j^b\phi.
		\end{equation*}
		
		Thanks to the fact that $(\sigma^{p_s-1}-W_i^{p_s-1})W_i\geq0$ for each $i$, we have
		\begin{equation*}
			(\sigma^{p_s-1}-W_j^{p_s-1})W_j\leq\sum_{i=1}^\kappa(\sigma^{p_s-1}-W_i^{p_s-1})W_i=\sigma^{p_s}-\sum_{i=1}^\kappa W_i^{p_s}
		\end{equation*}
		and
		\begin{equation*}
			(\sigma^{p_s-2}-W_j^{p_s-2})W_j\leq\sum_{i=1}^\kappa(\sigma^{p_s-2}-W_i^{p_s-2})W_i=\sigma^{p_s-1}-\sum_{i=1}^\kappa W_i^{p_s-1}.
		\end{equation*}
		Using Lemma \ref{qiegao} and $|\mathcal{Z}_j^b| \lesssim W_j$, we deduce that
		\begin{equation}\label{qk1}
			\begin{split}
				|J_1|				&\lesssim\Big|\int\big(|x|^{-\mu}\ast\sigma^{p_s-1}\mathcal{Z}_j^b\big)\big(\sigma^{p_s-1}-W_j^{p_s-1}\big)\phi\Big|
				+\Big|\int\big(|x|^{-\mu}\ast(\sigma^{p_s-1}-W_j^{p_s-1})\mathcal{Z}_j^b\big)W_j^{p_s-1}\phi\Big|        \\
				&\lesssim\Big\|\Big(|x|^{-\mu}\ast\sigma^{p_s}\Big)\sigma^{p_s-1}-\sum_{i=1}^{\kappa}\Big(|x|^{-\mu}\ast W_{i}^{p_s}\Big)W_{i}^{p_s-1}\Big\|_{L^{(2_s^{\ast})^{\prime}}(\mathbb{R}^n)}\|\phi\|_{L^{2_s^\ast}(\mathbb{R}^n)}
				\lesssim\mathscr{Q}^{\min\{\frac{\mu}{n-2s},1\}}\|\phi\|_{\dot{H}^s(\mathbb{R}^n)},
			\end{split}
		\end{equation}
		\begin{equation}\label{qk2}
			|J_2|
			\lesssim\Big|\int\big(|x|^{-\mu}\ast\sigma^{p_s}\big)\big(\sigma^{p_s-2}-W_j^{p_s-2}\big)W_j\phi\Big|
			+\Big|\int\big(|x|^{-\mu}\ast(\sigma^{p_s}-W_j^{p_s})\big)W_j^{p_s-1}\phi\Big|
			\lesssim\mathscr{Q}^{\min\{\frac{\mu}{n-2s},1\}}\|\phi\|_{\dot{H}^s(\mathbb{R}^n)}.
		\end{equation}
		By the H\"older inequality, we also have
		\begin{equation}\label{qk4}
			\Big|\int g\mathcal{Z}_j^b\Big|
			\leq\int|g|W_j
			\leq\|g\|_{L^{(2_s^{\ast})^{\prime}}(\mathbb{R}^n)}\|W_j\|_{L^{2_s^\ast}(\mathbb{R}^n)}
			\lesssim\|g\|_{L^{(2_s^{\ast})^{\prime}}(\mathbb{R}^n)}.
		\end{equation}
		Applying the estimates \eqref{qk1}-\eqref{qk4}, it follows by \eqref{kaihui} that
		\begin{equation*}		|c_b^j|\lesssim\|g\|_{L^{(2_s^{\ast})^{\prime}}(\mathbb{R}^n)}+\mathscr{Q}^{\min\{\frac{\mu}{n-2s},1\}}\|\phi\|_{\dot{H}^s(\mathbb{R}^n)},\hspace{3mm}j=1,\cdots, \kappa~\text{and}~b=1,\cdots,n+1.
		\end{equation*}
		The desired estimate is obtained.
	\end{proof}
	
	\begin{lem}\label{2estimate2}
		Assume that $n=6s$ and $0<\mu<4s$ or $n\neq6s$ and $0<\mu\leq4s$. Let $\phi$ be the solution to problem \eqref{coefficients1}. Then it holds that
		\begin{equation}\label{beibingyang}
			\|\phi\|_{\dot{H}^s(\mathbb{R}^n)}
			\leq C \|g\|_{L^{(2_s^{\ast})^{\prime}}(\mathbb{R}^n)}.
		\end{equation}
	\end{lem}
	\begin{proof}
		It follows directly from Lemma \ref{wanfan} that
		\begin{equation}\label{caI}
			\sum_{i=1}^{\kappa}\sum_{a=1}^{n+1}|c_{a}^{i}|
			\leq c\bigg(\mathscr{Q}^{\min\{\frac{\mu}{n-2s},1\}}\|\phi\|_{\dot{H}^s(\mathbb{R}^n)}+\|g\|_{L^{(2_s^{\ast})^{\prime}}(\mathbb{R}^n)}\bigg).
		\end{equation}
		By contradiction, if \eqref{beibingyang} does not hold true, there exists a sequence of functions
		$g=g_m$ with $\Vert g_m\Vert_{L^{(2_s^{\ast})^{\prime}}(\mathbb{R}^n)}\rightarrow0$ as $m\rightarrow\infty$, and $\phi=\phi_m$ with $\Vert\phi_m\Vert_{\dot{H}^s}=1$ solving the equation
		\begin{equation}\label{ccients}
			\left\{\begin{array}{l}
				\displaystyle (-\Delta)^s \phi_m-I_{n,\mu,s}[\sigma_m,\phi_m]=g_m+	\sum_{i=1}^{\kappa}\sum_{a=1}^{n+1}c_{a,m}^{i}I_{n,\mu,s}[W_{i}^{(m)},\mathcal{Z}^{a}_{i,m}]\hspace{4mm}\mbox{in}\hspace{2mm} \mathbb{R}^n,\\
				\displaystyle 	\int I_{n,\mu,s}[W_{i}^{(m)},\mathcal{Z}^{a}_{i,m}]\phi_m=0,\hspace{4mm}i=1,\cdots, \kappa; ~a=1,\cdots,n+1,
			\end{array}
			\right.
		\end{equation}
		with $\frac{1}{m}$-interacting bubbles $\big\{W_i^{\left(m\right)}=W[\xi_i^{\left(m\right)},\lambda_i^{\left(m\right)}]:~i=1,\cdots,\kappa\big\}_{m=1}^{\infty}$, and scalars $\{c_{a,m}^{i}\}_{m=1}^{\infty}$. By \eqref{caI}, we have
		\begin{equation}\label{AIC}
			\sum_{i=1}^{\kappa}\sum_{a=1}^{n+1}|c_{a,m}^{i}|\rightarrow0\hspace{3mm}\mbox{as}\hspace{2mm}m\rightarrow\infty.
		\end{equation}
		Then, using \eqref{ccients}-\eqref{AIC}, the Hardy-Littlewood-Sobolev inequality and H\"{o}lder inequality, we get
		$$
		\int I_{n,\mu,s}[\sigma_m,\phi_m]\phi_m=\|\phi_m\|_{\dot{H}^s(\mathbb{R}^n)}^2+O\big(\|g_m\|_{L^{(2_s^{\ast})^{\prime}}(\mathbb{R}^n)}+	\sum_{i=1}^{\kappa}\sum_{a=1}^{n+1}|c_{a,m}^{i}|\big)\rightarrow1\hspace{4mm}\mbox{as}\hspace{2mm}m\rightarrow\infty.
		$$
		On the other hand, noticing that the right-hand side of \eqref{ccients} vanishes in $\dot{H}^{-s}(\mathbb{R}^n)$, and since $\phi_m$ is perpendicular to the kernel of linearized equation in $\dot{H}^s(\mathbb{R}^n)$ with $\|\phi_m\|_{\dot{H}^{s}}=1$, by the nondegeneracy result in Theorem \ref{prondgr}, we arrive at
		$$\int I_{n,\mu,s}[\sigma_m,\phi_m]\phi_m\rightarrow0\hspace{3mm}\mbox{as}\hspace{2mm}m\rightarrow\infty.$$
		This leads to a contradiction, and the conclusion follows.
	\end{proof}
	
	From Lemma \ref{wanfan} and Lemma \ref{2estimate2}, using a standard argument as in the proof of Proposition 4.1 in \cite{delPino-1}, we can establish the following result.
	\begin{lem}\label{chenwending}
		Assume that $n=6s$ and $0<\mu<4s$ or $n\neq6s$ and $0<\mu\leq4s$. There exist positive constants $\delta_0$ and $C$, independent of $\delta$, such that for all $\delta\leq\delta_0$ and all $g$ with $\|g\|_{L^{(2_s^{\ast})^{\prime}}(\mathbb{R}^n)}<\infty$, the system \eqref{coefficients1} has a unique solution $\phi \equiv \mathcal{L}_\delta(g)$. Besides,
		\begin{equation*}
			\big\|\mathcal{L}_\delta(g)\big\|_{\dot{H}^s(\mathbb{R}^n)}
			\leq C\big\|g\big\|_{L^{(2_s^{\ast})^{\prime}}(\mathbb{R}^n)},\quad \big|c_a^i\big|\leq C\delta\big\|g\big\|_{L^{(2_s^{\ast})^{\prime}}(\mathbb{R}^n)}.
		\end{equation*}
	\end{lem}
	
	With the aid of the above linear theory, we can solve the following nonlinear equation:
	\begin{equation}\label{coefficients11}
		\left\{\begin{array}{l}
			\displaystyle (-\Delta)^s \phi
			-\Big(|x|^{-\mu}\ast \big(\sigma+\phi\big)^{p_s}\Big)
			\big(\sigma+\phi\big)^{p_s-1}
			+\sum_{i=1}^\kappa\Big(|x|^{-\mu}\ast W_i^{p_s}\Big)W_i^{p_s-1}
			\displaystyle =
			\sum_{i=1}^{\kappa}\sum_{a=1}^{n+1}c_{a}^{i}I_{n,\mu,s}[W_{i},\mathcal{Z}^{a}_i]\hspace{4.14mm}\mbox{in}\hspace{1.14mm} \mathbb{R}^n,\\
			\displaystyle \int I_{n,\mu,s}[W_{i},\mathcal{Z}^{a}_i]\phi=0,\hspace{4mm}i=1,\cdots, \kappa; ~a=1,\cdots,n+1.
		\end{array}
		\right.
	\end{equation}
	Recalling $g$ and $N(\phi)$ from \eqref{u-1}-\eqref{u-2}, then \eqref{coefficients11} can be written as
	\begin{equation}\label{AA}
		\begin{split}
			(-\Delta)^s\phi-I_{n,\mu,s}[\sigma,\phi]-g-N(\phi)
			=\sum_{i=1}^{\kappa}\sum_{a=1}^{n+1}c_{a}^{i}I_{n,\mu,s}[W_{i},\mathcal{Z}^{a}_i].
		\end{split}
	\end{equation}
	
	\begin{lem}\label{zhaji}
		Assume that $n=6s$ and $0<\mu<4s$ or $n\neq6s$ and $0<\mu\leq4s$. Then there exist $\varrho_0 \in\dot{H}^s(\mathbb{R}^n)$ and a family of scalars $\{c_{a}^{i}\}$ which solve \eqref{coefficients11} such that for $\delta$ is small enough, there holds
		\begin{equation*}
			\|\varrho_{0}\|_{\dot{H}^s(\mathbb{R}^n)}\leq c\mathscr{Q}^{\min\{\frac{\mu}{n-2s},1\}}.
		\end{equation*}
	\end{lem}
	\begin{proof}
		Observe that \eqref{coefficients11} is equivalent to
		\begin{equation*}
			\phi=\mathcal{A}(\phi):=\mathcal{L}_{\delta}(N(\phi))+\mathcal{L}_{\delta}(g),
		\end{equation*}
		where $\mathcal{L}_\delta$ is defined in Lemma \ref{chenwending}. In the following, we are going to show that $\mathcal{A}$ is a contraction mapping. First we claim that
		\begin{itemize}
			\item if $\mu=4s$, then
			\begin{equation}\label{kele1}
				\|N(\phi)\|_{L^{(2_s^{\ast})^{\prime}}(\mathbb{R}^n)}\leq L_0\|\phi\|_{\dot{H}^s(\mathbb{R}^n)}^2;
			\end{equation}
			\item if $0<\mu<4s$, then
			\begin{equation}\label{kele2}
				\|N(\phi)\|_{L^{(2_s^{\ast})^{\prime}}(\mathbb{R}^n)}\leq L_0\|\phi\|_{\dot{H}^s(\mathbb{R}^n)}^{\min\{p_s-1,2\}}.
			\end{equation}
		\end{itemize}
		Indeed, for $\mu=4s$ we have
		\begin{equation*}
			N(\phi)
			=2\big(|x|^{-4s}\ast\sigma\phi\big)\phi
			+\big(|x|^{-4s}\ast\phi^2\big)\sigma
			+\big(|x|^{-4s}\ast\phi^2\big)\phi
		\end{equation*}
		By the H\"older inequality, Sobolev inequlity and Hardy-Littlewood-Sobolev inequality, we can derive that \eqref{kele1} holds true.
		For $0<\mu<4s$, we can apply Lemma \ref{gaojiexiang} to achieve \eqref{kele2} as $\delta$ is small enough. Then we may choose $L_0>0$ sufficiently large in \eqref{kele1} and \eqref{kele2} to guarantee that
		\begin{equation*}
			\big\|\mathcal{L}_\delta(g)\big\|_{\dot{H}^s(\mathbb{R}^n)}\leq L_0\big\|g\big\|_{L^{(2_s^{\ast})^{\prime}}(\mathbb{R}^n)}.
		\end{equation*}
		Moreover it follows from Lemma \ref{qiegao} that there exists a positive constant $\zeta_0$ such that
		\begin{equation*}
			\big\|g\big\|_{L^{(2_s^{\ast})^{\prime}}(\mathbb{R}^n)}
			\leq\zeta_0\mathscr{Q}^{\min\{\frac{\mu}{n-2s},1\}}.
		\end{equation*}
		
		\par Set
		\begin{equation*}
			\mathcal{E}=\Big\{w:w\in C(\mathbb{R}^n)\cap \dot{H}^s(\mathbb{R}^n),\|w\|_{\dot{H}^s(\mathbb{R}^n)}\leq (\zeta_0L_0+1)\mathscr{Q}^{\min\{\frac{\mu}{n-2s},1\}}\Big\}.
		\end{equation*}
		We will show that $\mathcal{A}$ is a contraction map from $\mathcal{E}$ to $\mathcal{E}$. When $\mu=4s$, choose $\delta$ sufficiently small such that $L_0^2(\zeta_0L_0+1)^2\mathscr{Q}^{\min\{\frac{\mu}{n-2s},1\}}\leq1$, then we have
		\begin{equation*}
			\begin{split}
				\|\mathcal{A}(\phi)\|_{\dot{H}^s(\mathbb{R}^n)}
				&\leq L_0\|N(\phi)\|_{L^{(2_s^{\ast})^{\prime}}(\mathbb{R}^n)}
				+L_0\|g\|_{L^{(2_s^{\ast})^{\prime}}(\mathbb{R}^n)}            \\
				&\leq L_0^2(\zeta_0L_0+1)^2\mathscr{Q}^{2\min\{\frac{\mu}{n-2s},1\}}+\zeta_0L_0\mathscr{Q}^{\min\{\frac{\mu}{n-2s},1\}}
				\leq(\zeta_0L_0+1)\mathscr{Q}^{\min\{\frac{\mu}{n-2s},1\}}.
			\end{split}
		\end{equation*}
		For $0<\mu<4s$, choosing $\delta>0$ sufficiently small such that $L_0^2(\zeta_0L_0+1)^{\min\{p_s-1,2\}}\mathscr{Q}^{\min\{\frac{\mu}{n-2s},1\}}\leq1$.
		Then we get
		\begin{equation*}
			\begin{split}
				\|\mathcal{A}(\phi)\|_{\dot{H}^s(\mathbb{R}^n)}
				&\leq L_0^2(\zeta_0L_0+1)^{\min\{p_s-1,2\}}\mathscr{Q}^{\min\{\frac{\mu}{n-2s},1\}}+\zeta_0L_0\mathscr{Q}^{\min\{\frac{\mu}{n-2s},1\}}    \\
				&\leq(\zeta_0L_0+1)\mathscr{Q}^{\min\{\frac{\mu}{n-2s},1\}}.
			\end{split}
		\end{equation*}
		Hence, $\mathcal{A}$ maps $\mathcal{E}$ to $\mathcal{E}$.
		On the other hand, taking $\phi_1$ and $\phi_2$ in $\mathcal{E}$, for $\mu=4s$ we see that
		\begin{equation*}
			\|\mathcal{A}(\phi_1)-\mathcal{A}(\phi_2)\|_{\dot{H}^s(\mathbb{R}^n)}
			\leq L_0\|N(\phi_1)-N(\phi_2)\|_{L^{(2_s^\ast)^\prime}(\mathbb{R}^n)}.
		\end{equation*}
		The main goal, therefore, is to estimate each term on the right-hand side.
		\begin{equation*}
			\begin{split}
				|N(\phi_1)-N(\phi_2)|
				\lesssim&\big(|x|^{-4s}\ast\sigma|\phi_1|\big)|\phi_1-\phi_2|
				+\big(|x|^{-4s}\ast\sigma|\phi_1-\phi_2|\big)|\phi_2|
				+\big(|x|^{-4s}\ast|\phi_1-\phi_2||\phi_1+\phi_2|\big)\sigma    \\
				&+\big(|x|^{-4s}\ast\phi_1^2\big)|\phi_1-\phi_2|
				+\big(|x|^{-4s}\ast|\phi_1-\phi_2||\phi_1+\phi_2|\big)|\phi_2|,
			\end{split}
		\end{equation*}
		which implies
		\begin{equation*}
			\|\mathcal{A}(\phi_1)-\mathcal{A}(\phi_2)\|_{\dot{H}^s(\mathbb{R}^n)}
			\leq C(\zeta_0L_0+1)\mathscr{Q}^{\min\{\frac{\mu}{n-2s},1\}}\|\phi_1-\phi_2\|_{\dot{H}^s(\mathbb{R}^n)}.
		\end{equation*}
		For $0<\mu<4s$ we exploit the estimates from \cite[Lemma 6.3]{YZ25}, together with Lemma \ref{gaojiexiang}, to obtain
		\begin{equation*}
			\|\mathcal{A}(\phi_1)-\mathcal{A}(\phi_2)\|_{\dot{H}^s(\mathbb{R}^n)}
			\leq C(\zeta_0L_0+1)\mathscr{Q}^{\min\{\frac{\mu}{n-2s},1\}(p_s-2)}\|\phi_1-\phi_2\|_{\dot{H}^s(\mathbb{R}^n)}.
		\end{equation*}
		Therefore, we can deduce that
		\begin{equation*}
			\|\mathcal{A}(\phi_{1})-\mathcal{A}(\phi_{2})\|_{\dot{H}^s(\mathbb{R}^n)}
			\leq \frac{1}{2}\|\phi_{1}-\phi_{2}\|_{\dot{H}^s(\mathbb{R}^n)},
		\end{equation*}
		provide that $\delta>0$ is small enough which means that $\mathcal{A}$ is a contraction mapping from $\mathcal{E}$ into itself. Consequently, there exists a unique $\varrho_0\in\mathcal{E}$ such that $\varrho_0=\mathcal{A}(\varrho_0)$. Moreover, it follows from Lemma \ref{qiegao} and Lemma \ref{chenwending} that
		\begin{equation*}
			\big\|\varrho_{0}\big\|_{\dot{H}^s(\mathbb{R}^n)}
			\leq c\mathscr{Q}^{\min\{\frac{\mu}{n-2s},1\}},
		\end{equation*}
		which concludes the proof.
	\end{proof}

	\section{The existence of first approximation for $n=6s$ and $\mu=4s$}\label{stabilitysection6}
	To show the existence of first approximation when $n=6s$ and $\mu=4s$, we carry out the Lyapunov-Schmidt reduction argument in a weighted space.
	\subsection{Constructing the weighted space and norm}
	For simplicity of notation, we define the following weight functions:
	\begin{equation*}
		\begin{split}
			&s_{i,1}(x,\mathscr{R}):=\frac{\lambda_i^{2s}}{\tau(z_i)^{2s}\mathscr{R}^{4s}}\chi_{\{|z_i|\leq\mathscr{R}^2\}},\hspace{6mm}
			s_{i,2}(x,\mathscr{R})=\frac{\lambda_{i}^{2s}}{\tau(z_i)^{3s}\mathscr{R}^{2s}}\chi_{\{|z_i|>\mathscr{R}^2\}},\\&
			t_{i,1}(x,\mathscr{R}):=\frac{\lambda_i^{4s}}{\tau(z_i)^{4s}\mathscr{R}^{4s}}\chi_{\{|z_i|\leq\mathscr{R}^2\}},\hspace{6mm}
			t_{i,2}(x,\mathscr{R}):=\frac{\lambda_{i}^{4s}}{\tau(z_i)^{5s}\mathscr{R}^{2s}}\chi_{\{|z_i|>\mathscr{R}^2\}}.
		\end{split}
	\end{equation*}
	Here $\tau(z_i)=(1+|z_i|^2)^{1/2}$ with $z_i=\lambda_i(x-\xi_i)$.
	For the functions $\phi$ and $g$,
	we define the following weighted $\|\cdot\|_{\ast}$ and $\|\cdot\|_{\ast\ast}$ norms that will help us to capture the behavior of the interaction term $g$.
	\begin{Def}\label{st-11}
		For $n=6s$ and $\mu=4s$, define the norm $\|\cdot\|_{\ast}$ as
		\begin{equation}\label{f-ai}
			\|\phi\|_{\ast}=\sup_{x\in\mathbb{R}^{6s}}\big|\phi(x)\big|S^{-1}(x)
		\end{equation}
		and the norm $\|\cdot\|_{\ast\ast}$ as
		\begin{equation}\label{h}
			\|g\|_{\ast\ast}=\sup_{x\in\mathbb{R}^{6s}}\big|g(x)\big|T^{-1}(x)
		\end{equation}
		with the weights
		\begin{equation*}
			S(x)=\sum_{i=1}^{\kappa}\big[s_{i,1}(x,\mathscr{R})
			+s_{i,2}(x,\mathscr{R})\big]\hspace{4mm}\text{and}\hspace{4mm}T(x)=\sum_{i=1}^{\kappa}\big[t_{i,1}(x,\mathscr{R})
			+t_{i,2}(x,\mathscr{R})\big].
		\end{equation*}
	\end{Def}
	Due to the above norms, we will find a function $\phi$ which solvs the following system
	\begin{equation}\label{c1}
		\left\{\begin{array}{l}
			\displaystyle (-\Delta)^s \phi-I_{6s,4s,s}[\sigma,\phi]=g+	\sum_{i=1}^{\kappa}\sum_{a=1}^{n+1}c_{a}^{i}I_{6s,4s,s}[W_{i},\mathcal{Z}_i^a]\hspace{4mm}\mbox{in}\hspace{2mm} \mathbb{R}^{6s},\\
			\displaystyle 	\int I_{6s,4s,s}[W_{i},\mathcal{Z}_i^a]\phi=0,\hspace{4mm}i=1,\cdots, \kappa; ~a=1,\cdots,n+1,
		\end{array}
		\right.
	\end{equation}
	where $\{c_{a}^{i}\}$ is a family of scalars and $\mathcal{Z}_i^a$ are defined in \eqref{qta}. To this end, our purpose in what follows is to prove the following result.
	\begin{lem}\label{ww10}
		Assume that $n=6s$ and $\mu=4s$. There exist positive constants $\delta_0$ and $C$, independent of $\delta$, such that for all $\delta\leq\delta_0$ and all $g$ with $\|g\|_{\ast\ast}<\infty$, the system \eqref{c1} has a unique solution $\phi \equiv \mathcal{L}_\delta(g)$. Besides,
		\begin{equation*}
			\|\mathcal{L}_\delta(g)\|_{\ast}\leq C\|g\|_{\ast\ast},\quad |c_a^i|\leq C\delta\|g\|_{\ast\ast}.
		\end{equation*}
	\end{lem}
	
	The key point here is to prove the following priori estimate for $\|\phi\|_{\ast}$ in Lemma \ref{estimate2} and estimate of coefficients $c_j^b$ in Lemma \ref{cll}. Once this is done, a well-known standard argument (cf. \cite{delPino-1}) shows that Lemma \ref{ww10} holds true. However, for clarity and coherence, the proof of Lemma \ref{estimate2} will be deferred in Subsection \ref{wutang}.
	\begin{lem}\label{estimate2}
		Assume that $n=6s$ and $\mu=4s$. Let $\phi$ be the solution to problem \eqref{c1}. Then it holds that
		$$\|\phi\|_{\ast}\leq C \|g\|_{\ast\ast}, $$
		where the norms $\|\cdot\|_{\ast}$ and $\|\cdot\|_{\ast\ast}$ are the same as in Definition \ref{st-11}.
	\end{lem}
	\begin{lem}\label{cll}
		Assume that $n=6s$ and $\mu=4s$. Let $\phi$, $g$ and $c_b^j$ satisfy the system \eqref{c1} and  $\sigma=\sum_{i=1}^{\kappa}W_i$ is a family of $\delta$-interacting bubbles, then there holds:\\
		$$|c_b^j|\lesssim\mathscr{Q}\|g\|_{\ast\ast}+\mathscr{Q}^2|\log\mathscr{Q}|\|\phi\|_{\ast},\hspace{3mm}j=1,\cdots, \kappa~\text{and}~b=1,\cdots,n+1.
		$$
	\end{lem}
	Note that, the proof of Lemma \ref{cll} heavily relies on the following Lemma \ref{estimate1} and Lemma \ref{cll-0-0}.
	\begin{lem}\label{estimate1}
		Let $n=6s$, $\mu=4s$ and $\kappa\in\mathbb{N}$. There exist a positive constant $\delta=\delta(\kappa,s)>0$ and large constant $C=C(\kappa,s)$ such that
		\begin{equation}\label{eq3.7}
			\Big\|\Big(|x|^{-4s}\ast\sigma^2\Big)\sigma-\sum_{i=1}^{\kappa}\Big(|x|^{-4s}\ast W_{i}^2\Big)W_{i}\Big\|_{\ast\ast}\leq C,
		\end{equation}
		where $C$ depends only on $\kappa$ and $s$.
	\end{lem}
	\begin{proof}
		The only difference is that the exponents have been modified by the parameter $s$. Thus one can follow the same proof as in \cite[Lemma 4.1]{DHP} and we omit it.
	\end{proof}
	\begin{lem}\label{cll-0-0}
		Assume that $n=6s$ and $\mu=4s$,  we have that
		\begin{equation}\label{selqa}
			\big\|S\big\|_{L^3(\mathbb{R}^{6s})}\lesssim\mathscr{R}^{-4s}(\log\mathscr{R})^\frac{1}{3}  \hspace{4mm}\text{and}\hspace{4mm}
			\big\|T\big\|_{L^\frac{3}{2}(\mathbb{R}^{6s})}\lesssim\mathscr{R}^{-4s}(\log\mathscr{R})^\frac{2}{3} .
		\end{equation}
	\end{lem}
	\begin{proof}
		The only difference is that the exponents have been modified by the parameter $s$. Thus one can follow the same proof as in \cite[Lemma 3.7]{DSW21} and we omit it.
	\end{proof}
	Now we can prove Lemma \ref{cll}.
	\begin{proof}[Proof of Lemma \ref{cll}]
		Multiplying \eqref{c1} by $\mathcal{Z}_j^b$ and integrating we get
		\begin{equation}\label{xiayu}
			\int_{\mathbb{R}^{6s}}I_{6s,4s,s}[\sigma,\phi]\mathcal{Z}_j^b+\int_{\mathbb{R}^{6s}}g\mathcal{Z}_j^b+	\sum_{i=1}^{\kappa}\sum_{a=1}^{n+1}c_{a}^{i}\int_{\mathbb{R}^{6s}}I_{6s,4s,s}[W_{i},\mathcal{Z}^{a}_i]\mathcal{Z}_j^b=0,
		\end{equation}
		for any $1\leq j\leq n$, $1\leq b\leq n+1$. Here we used the orthogonal condition in \eqref{c1}.
		
		\par By Lemma \ref{armidale}, for $1\leq a,b\leq n+1$, there exist some constants $\gamma^b>0$ such that
		\begin{equation*}
			\sum_{i=1}^{\kappa}\sum_{a=1}^{n+1}c_{a}^{i}\int_{\mathbb{R}^{6s}}I_{6s,4s,s}[W_{i},\mathcal{Z}^{a}_i]\mathcal{Z}_j^b+c_b^j\gamma^b+\sum_{i\neq j}\sum_{a=1}^{n+1}c_{a}^{i}O(Q_{ij})=0.
		\end{equation*}
		Plugging in the above estimates to \eqref{xiayu}, we see that $\{c_b^j\}$ satisfies the linear system
		\begin{equation}\label{laoniangjiu}
			c_b^j\gamma^b+\sum_{i\neq j}\sum_{a=1}^{n+1}c_{a}^{i}O(Q_{ij})
			=\int_{\mathbb{R}^{6s}}I_{6s,4s,s}[\sigma,\phi]\mathcal{Z}_j^b+\int_{\mathbb{R}^{6s}}g\mathcal{Z}_j^b.
		\end{equation}
		Denote $\vec{c}^j := (c_1^j, \cdots, c_{n+1}^j) \in \mathbb{R}^{n+1}$ for $j = 1, \cdots, \kappa$. We concatenate these vectors to $\vec{c} = (\vec{c}^1, \cdots, \vec{c}^\kappa) \in \mathbb{R}^{\kappa(n+1)}$ and think of the above equations as a linear system on $\vec{c}$. Since $Q_{ij} \leq \mathscr{Q} \leq \delta$, the coefficient matrix is diagonally dominant and hence solvable. It remains to estimate the terms on the right-hand side.
		
		\par For each $j$ and $b$, by the orthogonal condition in \eqref{c1} we have
		\begin{equation*}
			\begin{split}
				\int_{\mathbb{R}^{6s}}I_{6s,4s,s}[\sigma,\phi]\mathcal{Z}_j^b
				&=\int_{\mathbb{R}^{6s}}\bigg(2\big(|x|^{-4s}\ast\sigma\phi\big)\sigma+\big(|x|^{-4s}\ast\sigma^2\big)\phi\bigg)\mathcal{Z}_j^b     \\
				&=2\int_{\mathbb{R}^{6s}}\big(|x|^{-4s}\ast\sigma \mathcal{Z}_j^b\big)\sigma\phi
				-\big(|x|^{-4s}\ast W_j\mathcal{Z}_j^b\big)W_j\phi
				+\int_{\mathbb{R}^{6s}}\big(|x|^{-4s}\ast(\sigma^2-W_j^2)\big)\mathcal{Z}_j^b\phi       \\
				&=:J_1+J_2.
			\end{split}
		\end{equation*}
		Thanks to the fact that $(\sigma-W_i)W_i\geq0$ for each $i$, we have
		\begin{equation*}
			(\sigma-W_j)W_j\leq\sum_{i=1}^\kappa(\sigma-W_i)W_i=\sigma^2-\sum_{i=1}^\kappa W_i^2.
		\end{equation*}
		Using Lemma \ref{estimate1}, Lemma \ref{cll-0-0} and $|\mathcal{Z}_j^b| \lesssim W_j$, we can obtain
		\begin{equation}\label{bz1}
			\begin{split}
				|J_1|
				&\lesssim\Big|\int_{\mathbb{R}^{6s}}\big(|x|^{-4s}\ast\sigma \mathcal{Z}_j^b\big)\big(\sigma-W_j\big)\phi\Big|
				+\Big|\int_{\mathbb{R}^{6s}}\big(|x|^{-4s}\ast(\sigma^2-\sum_{i=1}^\kappa W_i)\big)W_j\phi\Big|        \\
				&\lesssim\|\phi\|_\ast\int_{\mathbb{R}^{6s}}ST
				\lesssim\|\phi\|_\ast\|S\|_{L^3}\|T\|_{L^\frac{3}{2}}
				\lesssim\mathscr{Q}^2|\log\mathscr{Q}|\|\phi\|_{\ast}
			\end{split}
		\end{equation}
		and
		\begin{equation}\label{bz2}
			|J_2|
			\lesssim\|\phi\|_\ast\int_{\mathbb{R}^{6s}}ST
			\lesssim\|\phi\|_\ast\|S\|_{L^3}\|T\|_{L^\frac{3}{2}}
			\lesssim\mathscr{Q}^2|\log\mathscr{Q}|\|\phi\|_{\ast}.
		\end{equation}
		Taking \cite[Lemma 3.18]{Aryan} into account, we also have
		\begin{equation}\label{bz4}
			\bigg|\int_{\mathbb{R}^{6s}}g\mathcal{Z}_j^b\bigg|
			\leq\int_{\mathbb{R}^{6s}}|g|W_j
			\leq\|g\|_{\ast\ast}\int_{\mathbb{R}^{6s}}TW_j
			\lesssim\mathscr{Q}\|g\|_{\ast\ast}.
		\end{equation}
		Using the estimates \eqref{bz1}-\eqref{bz4}, it follows from \eqref{laoniangjiu} that
		\begin{equation*}			|c_b^j|\lesssim\mathscr{Q}\|g\|_{\ast\ast}+\mathscr{Q}^2|\log\mathscr{Q}|\|\phi\|_{\ast},\hspace{3mm}j=1,\cdots, \kappa~\text{and}~b=1,\cdots,n+1,
		\end{equation*}
		which concludes the proof.
	\end{proof}
	
	\subsection{Existence}
	With Lemma \ref{estimate2} in hand, we can solve the following nonlinear equation by applying the contraction mapping principle.
	\begin{equation}\label{suannai}
		\left\{\begin{array}{l}
			\displaystyle (-\Delta)^s \phi
			-\Big(|x|^{-\mu}\ast \big(\sigma+\phi\big)^2\Big)
			\big(\sigma+\phi\big)
			+\sum_{i=1}^\kappa\Big(|x|^{-\mu}\ast W_i^2\Big)W_i
			\displaystyle =
			\sum_{i=1}^{\kappa}\sum_{a=1}^{n+1}c_{a}^{i}I_{6s,4s,s}[W_{i},\mathcal{Z}^{a}_i]\hspace{4.14mm}\mbox{in}\hspace{1.14mm} \mathbb{R}^{6s},\\
			\displaystyle \int I_{6s,4s,s}[W_{i},\mathcal{Z}^{a}_i]\phi=0,\hspace{4mm}i=1,\cdots, \kappa; ~a=1,\cdots,n+1.
		\end{array}
		\right.
	\end{equation}
	\begin{lem}\label{xianyuxian}
		Assume that $n=6s$ and $\mu=4s$. Then there exist $\varrho_0 \in\dot{H}^s(\mathbb{R}^{6s})$ and a family of scalars $\{c_{a}^{i}\}$ which solve \eqref{suannai} such that for $\delta$ is small enough, there holds
		\begin{equation*}
			\|\varrho_{0}\|_\ast\leq c.
		\end{equation*}
	\end{lem}
	\begin{proof}
		It is easy to see that \eqref{suannai} is equivalent to
		\begin{equation*}
			\phi=\mathcal{B}(\phi):=\mathcal{L}_{\delta}(N(\phi))+\mathcal{L}_{\delta}(g),
		\end{equation*}
		where $\mathcal{L}_\delta$ is defined in Lemma \ref{ww10}. In the following, we will prove that $\mathcal{B}$ is a contraction mapping. First we claim that $\|N(\phi)\|_{\ast\ast}\leq L_1\|\phi\|_\ast S^2T^{-1}$. Then for the weight functions, we have
		\begin{equation*}
			\frac{s_{i,1}^2}{t_{i,1}}
			=\mathscr{R}^{-4s}
			\approx\mathscr{Q}\hspace{6mm}if\hspace{2mm}|z_i|\leq\mathscr{R}\hspace{4mm}\text{and}\hspace{4mm}
			\frac{s_{i,2}^2}{t_{i,2}}
			\leq\mathscr{R}^{-2s}
			\approx\mathscr{Q}\hspace{6mm}if\hspace{2mm} |z_i|>\mathscr{R}.
		\end{equation*}
		Thanks to the fact $\frac{\sum_{i}a_i}{\sum_{i}b_i}\leq\max_i\{\frac{a_i}{b_i}\}$, we know that $S^2T^{-1}\leq\mathscr{Q}$. Also we may choose $L_1>0$ sufficiently large to guarantee that
		\begin{equation*}
			\|\mathcal{L}_\delta(g)\|_\ast\leq L_1\|g\|_{\ast\ast}.
		\end{equation*}
		Moreover, Lemma \ref{estimate1} tells us that there exists a positive constant $\zeta_1$ such that $\|g\|_{\ast\ast}\leq\zeta_1$.
		
		\par Set
		\begin{equation*}
			\mathcal{E}=\Big\{w:w\in C(\mathbb{R}^n)\cap \dot{H}^s(\mathbb{R}^n),\|w\|_{\dot{H}^s(\mathbb{R}^{6s})}\leq \zeta_1L_1+1\Big\}.
		\end{equation*}
		We will prove that $\mathcal{B}$ is a contraction map from $\mathcal{E}$ to $\mathcal{E}$. Choosing $\delta$ sufficiently small such that $L_1^2(\zeta_1L_1+1)\mathscr{Q}\leq1$, then we have
		\begin{equation*}
			\|\mathcal{B}(\phi)\|_{\ast}
			\leq L_1\|N(\phi)\|_{\ast\ast}
			+L_1\|g\|_{\ast\ast}
			\leq L_1^2(\zeta_1L_1+1)\mathscr{Q}+\zeta_1L_1
			\leq\zeta_1L_1+1.
		\end{equation*}
		Hence, $\mathcal{B}$ maps $\mathcal{E}$ to $\mathcal{E}$.
		On the other hand, taking $\phi_1$ and $\phi_2$ in $\mathcal{E}$ we see that
		\begin{equation*}
			\|\mathcal{B}(\phi_1)-\mathcal{B}(\phi_2)\|_\ast
			\leq L_1\|N(\phi_1)-N(\phi_2)\|_\ast
			\leq\zeta_1L_1\mathscr{Q}\|\phi_1-\phi_2\|_\ast.
		\end{equation*}
		Therefore if $\delta$ small, we can deduce that
		\begin{equation*}
			\|\mathcal{B}(\phi_{1})-\mathcal{B}(\phi_{2})\|_\ast
			\leq \frac{1}{2}\|\phi_{1}-\phi_{2}\|_\ast,
		\end{equation*}
		provide that $\delta>0$ is small enough which means that $\mathcal{B}$ is a contraction mapping from $\mathcal{E}$ into itself. Consequently, there exists a unique $\varrho_0\in\mathcal{E}$ such that $\varrho_0=\mathcal{B}(\varrho_0)$. Furthermore, by applying Lemma \ref{ww10} and Lemma \ref{eq3.7}, we obtain that $\|\varrho_{0}\|_{\ast} \leq c$. Thus, we complete the proof.
	\end{proof}	
	
	\subsection{Proof of Lemma \ref{estimate2}}\label{wutang}
	This subsection is devoted to derive a rough $C^0$ estimate of solution to problem \eqref{c1} by exploiting the Green's function representation. In order to prove Lemma \ref{estimate2}, we are going to introduce the tree structure of $\delta$-interacting bubbles as $\delta\rightarrow0$. The concept of bubble-trees was investigated in \cite{DSW21,Druet,Parker,G-Tian}.
	
	\textbf{Tree structure of weak interacting bubbles.}
	To reformulate the tree structure, we introduce a sequence of $\kappa$ bubbles  $\big\{W_i^{\left(m\right)}:=W[\xi_i^{\left(m\right)},\lambda_i^{\left(m\right)}],~i=1,\cdots,\kappa\big\}_{m=1}^{\infty}$ satisfying
	\begin{equation}\label{TTT}
		\left\lbrace
		\begin{aligned}
			&\lambda_1^{(m)}\leq\cdots\leq\lambda_{\kappa}^{(m)}\hspace{2mm}\mbox{for all}\hspace{2mm}m\in\mathbb{N},\\&
			\mathscr{R}_m:=\frac{1}{2}\min\limits_{i\neq j}\big\{\mathscr{R}_{ij}^{(m)}:~i,j=1,\cdots,\kappa,~i\neq j\big\},\\&
			\mbox{either}\hspace{2mm}\lim\limits_{m\rightarrow\infty}\xi_{ij}^{(m)}=\xi_{ij}^{(\infty)}\in\mathbb{R}^n\hspace{2mm}\mbox{or}\hspace{2mm}\lim\limits_{m\rightarrow\infty}\xi_{ij}^{(m)}\rightarrow\infty,
		\end{aligned}
		\right.
	\end{equation}
	where $\xi_{ij}^{(m)}=\lambda_{i}^{(m)}(\xi_{j}^{(m)}-\xi_{i}^{(m)})$.
	\begin{Def}
		Let  $\preceq$ be a strict partial
		order on a set $\widetilde{T}$, and $\prec$ the corresponding strict partial order on a set $\widetilde{T}$.
		\begin{itemize}
			\item[$\bullet$]
			We say that a partially ordered set $(\widetilde{T},\prec)$ is a tree if for any $\tilde{t}\in\widetilde{T}$ such that the set $\{\tilde{s}\in\widetilde{T}:~\tilde{s}\prec\tilde{t}\}$ is well-ordered by the relation $\prec$.
			\item[$\bullet$]
			A descendant of $\tilde{s}\in\widetilde{T}$ is any element $\tilde{t}\in\widetilde{T}$ such that $\tilde{s}\prec\tilde{t}$.
		\end{itemize}
	\end{Def}
	The next lemma can be found in \cite{DSW21}.
	\begin{lem}
		For any sequence of $\{W_{i}^{(m)}\}$ satisfying \eqref{TTT}, we set a relation
		$\prec$ on $I=\{1,\cdots,\kappa\}$ as
		$$i\prec j~\Leftrightarrow~i<j\hspace{2mm}\mbox{and}\hspace{2mm}\hspace{2mm}\lim\limits_{m\rightarrow\infty}\xi_{ij}^{(m)}=\xi_{ij}^{(\infty)}\in\mathbb{R}^n.$$
		Then $\prec$ is a strict partial order and there exists $\kappa^{\ast}\in\{1,\cdots,\kappa\}$ such that $I$ can be be expressed as a tree.
	\end{lem}
	For each $i\in I$, we set $\mathscr{H}(i)$ be the set of descendants
	of $i\in I$, that is, $\mathscr{H}(i)=\{j\in I:~i\prec j\}$, or equivalently $$\mathscr{H}(i):=\{j\in I:\hspace{2mm} i<j,\hspace{2mm}\lim\limits_{m\rightarrow\infty}\xi_{ij}^{(m)}=\xi_{ij}^{(\infty)}\in\mathbb{R}^n\}.$$
	
	We need the following interaction estimates.
	\begin{lem}
		Suppose $\lambda_{i}\leq\lambda_{j}$. When the dimension $n=6s$, we have
		\begin{equation}\label{1}
		s_{i,1}	W_j\lesssim\mathscr{R}^{-2s}[t_{j,1}+t_{j,2}+t_{i,1}],
		\end{equation}
		\begin{equation}\label{2}
			s_{i,2}W_j\lesssim\mathscr{R}^{-2s}[t_{j,1}+t_{j,2}+t_{i,2}],
		\end{equation}
		\begin{equation}\label{3}
			s_{j,1}W_i\lesssim\tau(\xi_{ij})^{-2s}[t_{i,1}+t_{j,1}]+\mathscr{R}^{-2s}\tau(\xi_{ij})^{-s}t_{i,2},
		\end{equation}
		\begin{equation}\label{4}
			s_{j,2}W_i\lesssim\tau(\xi_{ij})^{-s}t_{i,1}+\mathscr{R}^{-2s}t_{i,2}+\tau(\xi_{ij})^{-2s}t_{j,2}.
		\end{equation}
		For any $0<\varepsilon<1$ and $M>1$, we also have the following
		\begin{equation}\label{5}
			\big(s_{j,1}+s_{j,2}\big)W_i\lesssim\bigg[\bigg(\frac{\lambda_{i}}{\lambda_{j}}\bigg)^2+\varepsilon^2\bigg]^s(t_{j,1}+t_{j,2})\hspace{6mm}if\hspace{2mm}|z_i-\xi_{ij}|\leq\varepsilon,
		\end{equation}
		\begin{equation}\label{6}
			s_{j,1}+s_{j,2}
			\lesssim\tau(\xi_{ij})^{5s}\varepsilon^{-3s}[\varepsilon^{-\frac{3s}{2}}t_{i,1}+\varepsilon^st_{j,2}]\hspace{6mm}if\hspace{2mm}|z_i-\xi_{ij}|\geq\varepsilon.
		\end{equation}
	\end{lem}
	\begin{proof}
		The only difference is that the exponents have been modified by the parameter $s$. Thus one can follow the same proof as in \cite[Lemma 4.2]{DSW21} and we omit it.
	\end{proof}
	
	For each $i\in I=\left\{1,\cdots,\kappa\right\}$, recall that
	\begin{equation*}		z_i^{\left(m\right)}=\lambda_i^{\left(m\right)}(x-\xi_i^{\left(m\right)})\hspace{4mm} \mbox{and}\hspace{4mm} \xi_{ij}^{\left(m\right)}
		=\lambda_i^{\left(m\right)}\left(\xi_j^{\left(m\right)}-\xi_i^{\left(m\right)}\right), ~~~\mbox{for all}~ j\neq i.
	\end{equation*}
	Furthermore, we can assume that
	\begin{equation}\label{AA-00}
		C^{\star}:=1+\max\limits_{i,j\in I, m\geq0}\{|\xi_{ij}^{(m)}|:\hspace{2mm}i<j\hspace{4mm}\mbox{and}\hspace{4mm}\lim\limits_{m\rightarrow\infty}\xi_{ij}^{(m)}\in\mathbb{R}^n\}<\infty.
	\end{equation}
	
	\textbf{Decomposition of $\mathbb{R}^{6s}$.}
	Let us define the sets
	\begin{equation*}
		\begin{split}
			\Omega^{\left(m\right)}
			:=\mathop{\bigcup}\limits_{i\in I}\left\{x:\vert z_i^{\left(m\right)}\vert\leq M\right\}\hspace{2mm}\mbox{and}\hspace{2mm}
			\Omega_i^{\left(m\right)}
			:=\left\{x:\vert z_i^{\left(m\right)}\vert\leq M,
			\vert z_i^{\left(m\right)}-\xi_{ij}^{\left(m\right)}\vert\geq\varepsilon_0,\forall j\in \mathscr{H}(i)\right\},
		\end{split}
	\end{equation*}
	where $M=M(s,\kappa)>0$ is a large constant and $\varepsilon_0=\varepsilon_0(s,\kappa)>0$ is a small constant to be determined later. Then the decomposition of $\mathbb{R}^{6s}$ is given by
	$$\mathbb{R}^{6s}=\mathcal{C}_{ext,M,1}+\mathcal{C}_{core,M,2}+\mathcal{C}_{neck,M,3}$$
	with
	$$\mathcal{C}_{ext,M,1}:=\bigcap\limits_{i\in I}\left\{x:\vert z_i^{\left(m\right)}\vert>M\right\},\hspace{2mm}\mathcal{C}_{core,M,2}:=\bigcup\limits_{i\in I}\Omega_i^{\left(m\right)}\hspace{2mm}\mbox{and}\hspace{2mm} \mathcal{C}_{neck,M,3}:=\bigcup\limits_{i\in I}D_i^{(m)},$$
	where $D_i^{(m)}$ is given by
	\begin{equation*}
		D_i^{(m)}:=\mathop{\bigcup}\limits_{j\in \mathscr{H}(i)}
		\left\{x:
		\vert z_i^{(m)}-\xi_{ij}^{(m)}\vert\leq\varepsilon_0
		\right\}\setminus\Big(\mathop{\bigcup}\limits_{j\in \mathscr{H}(i)}\left\{x:\vert z_j^{(m)}\vert< M\right\}\Big).
	\end{equation*}
	
	\textbf{Conclusion.}
	Assume that Lemma \ref{estimate2} does not hold true, in other words, up to a subsequence, there exists a sequence of functions
	$g=g_m$ with $\Vert g_m\Vert_{\ast\ast}\rightarrow0$ and $\phi=\phi_m$ with $\Vert\phi_m\Vert_\ast=1$ and $\frac{1}{m}$-interacting bubbles $\big\{W_i^{\left(m\right)}: i\in I\big\}_{m=1}^\infty$ solving the equation
	\begin{equation}\label{00}
		\left\{\begin{array}{l}
			\displaystyle (-\Delta)^s \phi_m- I_{6s,4s,s}[\sigma_m,\phi_m]
			=g_m+\sum_{i=1}^{\kappa}\sum_{a=1}^{n+1}c_{a,m}^{i}I_{6s,4s,s}[W_{i}^{(m)},\mathcal{Z}_{i,m}^a]\hspace{4.14mm}\mbox{in}\hspace{1.14mm} \mathbb{R}^{6s},\\
			\displaystyle \int_{\mathbb{R}^{6s}} I_{6s,4s,s}[W_{i}^{(m)},\mathcal{Z}^{a}_{i,m}]\phi_m=0,\hspace{4mm}i=1,\cdots, \kappa; ~a=1,\cdots,n+1.
		\end{array}
		\right.
	\end{equation}
	In order to achieve the desired contradiction with $\|\phi_{m}\|_{\ast}=1$ for any $m\in\mathbb{N}$, we divide the proof of Lemma \ref{estimate2} into several steps.
	\begin{step}\label{step5.1}
		When $\left\{x_m\right\}\subset\mathcal{C}_{ext,M,1}$, we have that
		\begin{equation}\label{n-m-1}
			|\phi_{m}|(x)<\frac{1}{6}S(x)\hspace{4mm}\mbox{as}\hspace{2mm}m\rightarrow+\infty.
		\end{equation}
	\end{step}
	\begin{proof}[Proof of Step \ref{step5.1}]
			We claim that there exists $C=C(\kappa,s)>0$ such that the following inequality holds for any $K>1$,
			\begin{equation*}
				\frac{|\phi_m|(x)}{S(x)}\leq C\bigg(\mathscr{R}_m^{-2s}+\|g_m\|_{\ast\ast}+K^{8s}\frac{\overline{S}(x)}{S(x)}+K^{-s}+K^{4s}\mathscr{R}_m^{-2s}\bigg),
			\end{equation*}
			where $\overline{S}(x)=\sum_{i\neq j}\bar{s}_{i,1}(x)+\bar{s}_{i,2}(x)$ is defined by
			\begin{equation*}
				\bar{s}_{i,1}(x,\mathscr{R})
				=\frac{\lambda_{i}^{2s}}{\tau(z_i)^{3s}\mathscr{R}^{4s}}(1+\log\tau(z_i))\chi_{\{|z_i|\leq\mathscr{R}^2\}}\hspace{4mm}\text{and}\hspace{4mm}\bar{s}_{i,2}(x,\mathscr{R})
				=\frac{\lambda_{i}^{2s}}{\tau(z_i)^{4s}\mathscr{R}^{2s}}\log\tau(z_i)\chi_{\{|z_i|>\mathscr{R}^2\}}.
			\end{equation*}
			
			In fact, by the Green's representation, we have
			\begin{equation}\label{gelinhanshu}
				\phi_m(\tilde{x})
				=C\int_{\mathbb{R}^{6s}}\frac{1}{|\tilde{x}-x|^{4s}}\bigg(I_{6s,4s,s}[\sigma_m,\phi_m]+g_m+\sum_{i=1}^{\kappa}\sum_{a=1}^{n+1}c_{a}^{i}I_{6s,4s,s}[W_{i}^{(m)},\mathcal{Z}_{i,m}^a]\bigg)dx.
			\end{equation}
			Applying $\int|\tilde{x}-x|^{2s-n}T(x)dx\lesssim S(\tilde{x})$ from \cite[Lemma 4.5]{C-K-L-24}, we get
			\begin{equation}\label{diyige}
				\Big|\int_{\mathbb{R}^{6s}}\frac{1}{|\tilde{x}-x|^{4s}}\bigg(g_m+\sum_{i=1}^{\kappa}\sum_{a=1}^{n+1}c_{a}^{i}I_{6s,4s,s}[W_{i}^{(m)},\mathcal{Z}_{i,m}^a]\bigg)dx\Big|
				\lesssim\big(\|g_m\|_{\ast\ast}+\mathscr{R}_m^{-4s}\big)S(\tilde{x}).
			\end{equation}
			Next we consider
			\begin{equation}\label{hongzao1}
				\begin{split}
					\int_{\mathbb{R}^{6s}}\frac{1}{|\tilde{x}-x|^{4s}}I_{6s,4s,s}[\sigma_m,\phi_m](x)dx
					&=\int_{\mathbb{R}^{6s}}\frac{1}{|\tilde{x}-x|^{4s}}\bigg(\int_{\mathbb{R}^{6s}}\frac{\sigma_m^2(y)}{|x-y|^{4s}}dy\phi_m(x)+\int_{\mathbb{R}^{6s}}\frac{2\sigma_m(y)\phi_m(y)}{|x-y|^{4s}}dy\sigma_m(x)\bigg)dx     \\
					&=:J_1+J_2.
				\end{split}
			\end{equation}
			For $J_1$. Using Lemma \ref{p1-00}, we achieve
			\begin{equation*}
				|J_1|
				\lesssim\sum_{i,j=1}^\kappa\int_{\mathbb{R}^{6s}}\frac{1}{|\tilde{x}-x|^{4s}}W_j^{(m)}(x)\big[s_{i,1}(x,\mathscr{R}_m)+s_{i,2}(x,\mathscr{R}_m)\big]dx
				=:\sum_{i,j=1}^\kappa \big(\mathcal{P}_{ij}^{(1)}+\mathcal{P}_{ij}^{(2)}\big)(\tilde{x}).
			\end{equation*}
			
			Consider $\mathcal{P}_{ii}^{(1)}$ (that is $i=j$). Applying Lemma \ref{B3}, we can obtain
			\begin{equation*}
				\begin{split}
					\mathcal{P}_{ii}^{(1)}(\tilde{x})
					&\approx\lambda_{i}^{2s}\mathscr{R}_m^{-4s}\int_{|z_i|\leq\mathscr{R}_m^2}\frac{1}{|\tilde{z_i}-z_i|^{4s}}\frac{1}{\tau(z_i)^{6s}}dz_i    \\
					&\lesssim\lambda_{i}^{2s}\mathscr{R}_m^{-4s}\bigg[\tau(\tilde{z_i})^{-4s}(1+\log\sqrt{1+|\tilde{z_i}|^2})\chi_{\{|\tilde{z_i}|\leq\mathscr{R}_m^2\}}+(\log\mathscr{R}_m)\tau(\tilde{z_i})^{-4s}\chi_{\{|\tilde{z_i}|>\mathscr{R}_m^2\}}\bigg]   \\
					&\lesssim\bar{s}_{i,1}(\tilde{x})+\bar{s}_{i,2}(\tilde{x}).
				\end{split}
			\end{equation*}
			Similarly, applying Lemma \ref{B3}, we have
			\begin{equation*}
				\begin{split}
					\mathcal{P}_{ii}^{(2)}(\tilde{x})
					&\approx\lambda_{i}^{2s}\mathscr{R}_m^{-2s}\int_{|z_i|>\mathscr{R}_m^2}\frac{1}{|\tilde{z_i}-z_i|^{4s}}\frac{1}{\tau(z_i)^{7s}}dz_i   \\
					&\lesssim\lambda_{i}^{2s}\mathscr{R}_m^{-2s}\bigg[\mathscr{R}_m^{-8s}\chi_{\{|\tilde{z_i}|\leq\mathscr{R}_m^2\}}+\tau(\tilde{z_i})^{-4s}(1+\log\sqrt{1+|\tilde{z_i}|^2})\chi_{\{|\tilde{z_i}|>\mathscr{R}_m^2\}}\bigg]     \\
					&\lesssim\bar{s}_{i,1}(\tilde{x})+\bar{s}_{i,2}(\tilde{x}).
				\end{split}
			\end{equation*}
			
			Consider the case $i\neq j$. If $\lambda_{i}<\lambda_{j}$, then using \eqref{1}, \eqref{2} and $\int|\tilde{x}-x|^{2s-n}T(x)dx\lesssim S(\tilde{x})$ we have
			\begin{equation*}
				\big(\mathcal{P}_{ij}^{(1)}+\mathcal{P}_{ij}^{(2)}\big)(\tilde{x})
				\lesssim\mathscr{R}_m^{-2s}S(\tilde{x}).
			\end{equation*}
			If $\lambda_{i}\geq\lambda_{j}$ and $\tau(\xi_{ij})>M$, then using \eqref{3} and \eqref{4} we have
			\begin{equation*}
				\big(\mathcal{P}_{ij}^{(1)}+\mathcal{P}_{ij}^{(2)}\big)(\tilde{x})
				\lesssim(\mathscr{R}_m^{-2s}+K^{-s})S(\tilde{x}).
			\end{equation*}
			If $\lambda_{i}\geq\lambda_{j}$ and $\tau(\xi_{ij})\leq M$, we apply \eqref{5} and \eqref{6} taking $\varepsilon=K^{-1}$ to get
			\begin{equation*}
				\big(\mathcal{P}_{ij}^{(1)}+\mathcal{P}_{ij}^{(2)}\big)(\tilde{x})
				\lesssim(K^{4s}\mathscr{R}_m^{-4s}+K^{-2s})S(\tilde{x})+K^{8s}\overline{S}(\tilde{x}).
			\end{equation*}
			Consolidating the estimates of $\mathcal{P}_{ij}^{(1)}$ and $\mathcal{P}_{ij}^{(2)}$, we can deduce that
			\begin{equation}\label{hongzao2}
				|J_1|
				\lesssim\big(\mathscr{R}_m^{-2s}+K^{-s}+K^{4s}\mathscr{R}_m^{-2s}\big)S(\tilde{x})+K^{8s}\overline{S}(\tilde{x}).
			\end{equation}
			
			Next we compute $|J_2|$. First we have
			\begin{equation}\label{jiaozi1}
				\begin{split}
					\Big|\int_{\mathbb{R}^{6s}}\frac{2\sigma_m(x)\phi_m(x)}{|\tilde{x}-x|^{4s}}dx\sigma_m(\tilde{x})\Big|
					&\lesssim\sum_{l=1}^{\kappa}\sum_{j=1}^{\kappa}\sum_{i=1}^{\kappa}\int_{\mathbb{R}^{6s}}\frac{W_l^{(m)}(x)[s_{i,1}(x,\mathscr{R}_m)+s_{i,2}(x,\mathscr{R}_m)]}{|\tilde{x}-x|^{4s}}dxW_j^{(m)}(\tilde{x})     \\
					&=:\sum_{l=1}^{\kappa}\sum_{j=1}^{\kappa}\sum_{i=1}^{\kappa}\bigg(\mathcal{Q}_{il}^{(1)}+\mathcal{Q}_{il}^{(2)}\bigg)(\tilde{x})W_j^{(m)}(\tilde{x}).
				\end{split}
			\end{equation}
			
			Consider $\mathcal{Q}_{il}^{(1)}$. In this case, we use the fact $(1+|z|^2)^{1/2}\approx1+|z|$ to get
			\begin{equation*}
				\begin{split}
					\mathcal{Q}_{il}^{(1)}(\tilde{x})
					&\approx\lambda_{l}^{2s}\mathscr{R}_m^{-4s}\int_{\{|z_i|\leq\mathscr{R}_m^2\}}\frac{1}{|\tilde{z_i}-z_i|^{4s}}\frac{1}{\tau(z_i)^{2s}}\frac{1}{\tau(z_l)^{4s}}dz_i     \\
					&\approx\lambda_{l}^{2s}\mathscr{R}_m^{-4s}\int_{\{|z_i|\leq\mathscr{R}_m^2\}}\frac{1}{|\tilde{z_i}-z_i|^{4s}}\frac{1}{\tau(z_i)^{2s}}\frac{1}{\big(1+\frac{\lambda_{l}}{\lambda_{i}}|z_i-\xi_{il}|\big)^{4s}}dz_i.
				\end{split}
			\end{equation*}
			Following the reasoning presented in \cite{DHP} we obtain
			\begin{equation}\label{jiaozi2}
				\mathcal{Q}_{il}^{(1)}(\tilde{x})W_j^{(m)}(\tilde{x})
				\lesssim \mathscr{R}_m^{-4s}T(\tilde{x})+\sum_{j=1}^{\kappa}\frac{\lambda_{j}^{4s}\mathscr{R}_m^{-4s}\log\sqrt{1+|z_j|^2}}{\tau(z_j)^{8s}}.
			\end{equation}
			
			Considering
			\begin{equation*}
				\mathcal{Q}_{il}^{(2)}
				\approx\lambda_{l}^{2s}\mathscr{R}_m^{-2s}\int_{\{|z_i|>\mathscr{R}_m^2\}}\frac{1}{|\tilde{z_i}-z_i|^{4s}}\frac{1}{\tau(z_i)^{3s}}\frac{1}{\big(1+\frac{\lambda_{l}}{\lambda_{i}}|z_i-\xi_{il}|\big)^{4s}}dz_i,
			\end{equation*}
			similarly, we have
			\begin{equation}\label{jiaozi3}
				\mathcal{Q}_{il}^{(2)}(\tilde{x})W_j^{(m)}(\tilde{x})
				\lesssim \mathscr{R}_m^{-4s}T(\tilde{x})+\sum_{j=1}^{\kappa}\frac{\lambda_{j}^{4s}\mathscr{R}_m^{-4s}\log\sqrt{1+|z_j|^2}}{\tau(z_j)^{8s}}.
			\end{equation}
			By combining \eqref{jiaozi1}-\eqref{jiaozi3}, we arrive at
			\begin{equation}\label{hongzao3}
				|J_2|
				\lesssim\int_{\mathbb{R}^{6s}}\frac{1}{|\tilde{x}-x|^{4s}}\bigg(\mathscr{R}_m^{-4s}T(\tilde{x})+\sum_{j=1}^{\kappa}\frac{\lambda_{j}^{4s}\mathscr{R}_m^{-4s}\log\sqrt{1+|z_j|^2}}{\tau(z_j)^{8s}}\bigg)dx
				\lesssim\big(\mathscr{R}_m^{-4s}+K^{-s}\big)S(\tilde{x})+\overline{S}(\tilde{x}).
			\end{equation}
			As a consequence of \eqref{hongzao1}, \eqref{hongzao2} and \eqref{hongzao3}, it follows that
			\begin{equation*}
				\Big|\int_{\mathbb{R}^{6s}}\frac{1}{|\tilde{x}-x|^{4s}}I_{6s,4s,s}[\sigma_m,\phi_m](x)dx\Big|
				\lesssim\big(\mathscr{R}_m^{-2s}+K^{-s}+K^{4s}\mathscr{R}_m^{-2s}\big)S(\tilde{x})+K^{8s}\overline{S}(\tilde{x}),
			\end{equation*}
			which together with \eqref{diyige}, our claim follows immediately.
		
		Now let us choose $K=K(\kappa,s)$ sufficiently large such that $CK^{-s}<(100\kappa)^{-1}$. Observe that he functions $\bar{s}_{i,1}$ and $\bar{s}_{i,2}$ decay faster than $s_{i,1}$ and $s_{i,2}$, respectively. More precisely, in $\mathcal{C}_{ext,M,1}$ we have
		\begin{equation*}
			\bar{s}_{i,1}\leq2M^{-s}(\log M)s_{i,1},\hspace{4mm}\bar{s}_{i,2}\leq2M^{-s}(\log M)s_{i,2}.
		\end{equation*}
		Then taking $M=M(\kappa,s,C^\star)$ sufficiently large satisfying
		\begin{equation*}
			2CK^{8s}M^{-s}(\log M)\leq(100\kappa)^{-1},
		\end{equation*}
		it follows that
		\begin{equation*}
			\frac{|\phi_m|(x)}{S(x)}\leq\frac{1}{50\kappa}+o(1)<\frac{1}{6}\hspace{6mm}\text{on}\hspace{2mm}\mathcal{C}_{ext,M,1},\hspace{2mm}\text{as}\hspace{2mm}m\rightarrow\infty,
		\end{equation*}
		because we assume that $\Vert g_m\Vert_{\ast\ast}\rightarrow0$.
	\end{proof}
	
	\par Let the blow-up sequences $\bar{\phi}_{m}$, $\bar{g}_{m}$ and $\bar{\sigma}_m$ be given by
	\begin{equation*}
		\left\lbrace
		\begin{aligned}
			&\bar{\phi}_m(z):=\phi_m(z/\lambda_{i_0}^{(m)}+\xi_{i_0j}^{(m)})S^{-1}(x_m),     \\  &\bar{g}_m(z):=(\lambda_{i_0}^{(m)})^{-2s}\bigg[g_m+\sum_{i=1}^{\kappa}\sum_{a=1}^{n+1}c_{a,m}^{i}I_{6s,4s,s}[W_i^{(m)},\mathcal{Z}^{a}_{i,m}]\bigg](z/\lambda_{i_0}^{(m)}+\xi_{i_0j}^{(m)})S^{-1}(x_m),    \\
			&\bar{\sigma}_m(z):=W[0,1](z)+\sum_{j\in I\setminus\{i_0\}}W[\xi_{i_0j}^{(m)},\lambda_{j}^{(m)}/\lambda_{i_0}^{(m)}](z),
		\end{aligned}
		\right.
	\end{equation*}
	with $z=\lambda_{i_0}^{(m)}(x-\xi_{i_0j}^{(m)})$, so that from \eqref{c1} $\bar{\phi}_m$ satisfies
	\begin{equation}\label{fai}
		\left\{\begin{array}{l}
			\displaystyle (-\Delta)^s \bar{\phi}_m(z)-I_{6s,4s,s}[\bar{\sigma}_m(z),\bar{\phi}_m(z)]
			=\bar{g}_m(z)\hspace{4.14mm}\mbox{in}\hspace{1.14mm} \mathbb{R}^{6s},\\
			\displaystyle \int_{\mathbb{R}^{6s}} I_{6s,4s,s}[W,\mathcal{Z}^{a}]\bar{\phi}_m=0,\hspace{4mm}i=1,\cdots, \kappa; ~a=1,\cdots,n+1.
		\end{array}
		\right.
	\end{equation}
	Here $W=W[0,1](z)$ and $\mathcal{Z}^a=\mathcal{Z}^a(z)=\mathcal{Z}_i^a(z)$. Let $\xi_{i_0j}^{(\infty)}:=\lim_{m\rightarrow\infty}\xi_{i_0j}^{\left(m\right)}$ and define
	\begin{equation}\label{nongfushanquan}
		\begin{split}
			&\mathcal{X}_L:=\mathop{\bigcap}\limits_{j\in\mathscr{H}(i)}
			\left\{z:\vert z-\xi_{i_0j}^{(\infty)}\vert\geq1/L\right\}.
		\end{split}
	\end{equation}
	Let $\mathcal{Y}_L:=\left\{z:\vert z\vert\leq L\right\}\cap \mathcal{X}_L$. If we choose $L\geq2\max\left\{M,\varepsilon^{-1}\right\}$, it is easy to see that $\Omega_i^{\left(m\right)}\subset\subset \mathcal{Y}_L$ for $m$ large enough.	
	
	In light of \cite[Lemma 4.7]{DSW21}, we can establish the following estimates.
	\begin{lem}\label{cll-11-2}
		Assume that $n=6s$, $\mu=4s$ and $j\notin\mathscr{H}(i)$. For any $\widetilde{A}>0$, we have that
		$W_j^{(m)}(x)=o(1)W_i^{(m)}(x)$, and
		\begin{equation*}
			s_{j,1}+s_{j,2}=o(1)s_{i,1};\hspace{4mm}t_{j,1}+t_{j,2}=o(1)t_{i,1}
		\end{equation*}
		uniformly for $x\in B(0,\widetilde{A})$ where  $B(0,\widetilde{A}):=\{x:|\xi_i^{(m)}|\leq\widetilde{A}\}$.
	\end{lem}
	
	As a consequence of Lemma \ref{cll-11-2}, we can directly deduce that
	\begin{lem}\label{cll-11-3}
		Assume that $n=6s$, $\mu=4s$ and $j\notin\mathscr{H}(i)$. Then we have that
		\begin{equation}\label{zzz-2}
			\sum\limits_{j=1}^{\kappa}W_j^{(m)}=\sum\limits_{j\in\mathscr{H}(i)}W_j^{(m)}+(1+o(1))W_i^{(m)},
		\end{equation}
		\begin{equation}\label{617}
			S(x)
			=\sum\limits_{j\in\mathscr{H}(i)}(s_{j,1}(x)+s_{j,2}(x))+(1+o(1))s_{i,1}(x),
		\end{equation}
		\begin{equation}\label{618}
			T(x)
			=\sum\limits_{j\in\mathscr{H}(i)}(t_{j,1}(x)+t_{j,2}(x))+(1+o(1))t_{i,1}(x),
		\end{equation}
		uniformly for $x\in B(0,\widetilde{A})$.
	\end{lem}
	
	\begin{step}\label{step5.2}
		When $\left\{x_m\right\}\subset\mathcal{C}_{core,M,2}$, we have that	
		\begin{equation*}
			|\phi_{m}|(x)=o(1)S(x)\hspace{4mm}\text{as}\hspace{2mm}m\rightarrow+\infty
		\end{equation*}
	\end{step}
	To prove estimate \eqref{n-m-2} by contradiction, up to a subsequence, we choose $\varepsilon_1>0$ and $\{x_m\}\subset\mathcal{C}_{core,M,2}$ such that
	\begin{equation*}
		|\phi_{m}|(x_m)>\varepsilon_1S(x_m),
	\end{equation*}
	At this stage, it remains to prove the following results.
	\begin{Prop}\label{diyigemingti}
		Assume that $n=6s$ and $\mu=4s$. In each compact subset $\mathcal K_L$, we have the following estimate, uniformly for $z\in \mathcal{Y}_L$
		\begin{equation}\label{KM-0}
			\bar{\sigma}_m(z)\rightarrow W[0,1](z)
			\hspace{2mm}\mbox{as}\hspace{2mm} m\rightarrow\infty.
		\end{equation}
		Moreover we have that
		\begin{equation}\label{km-1}
			\vert\bar{\phi}_m\vert(z)\lesssim\sum_{j\in\mathscr{H}(i_0) }\big(\frac{M}{|z-\xi_{i_0j}^{(\infty)}|}\big)^{3s}+M^{2s},\hspace{2mm}\mbox{for} \hspace{2mm}z\in \mathcal{Y}_L,
		\end{equation}	
	\end{Prop}
	\begin{proof}
		Thanks to Lemma \ref{cll-11-3}, to conclude the proof of \eqref{KM-0}, it suffices to check that
		$$\sum_{j\in I\setminus\{i_0\}}W[\xi_{i_0j}^{(m)},\lambda_{j}^{(m)}/\lambda_{i_0}^{(m)}](z)=o(1)W[0,1](z),\hspace{2mm}\mbox{for}\hspace{2mm}z\in \mathcal{Y}_{L},\hspace{2mm}j\in\mathscr{H}(i_0).$$
		In the case $j\in\mathscr{H}(i_0)$, one must have $\lambda_{j}^{(m)}/\lambda_{i_0}^{(m)}\rightarrow\infty$ as $m\rightarrow\infty$.
		Then together with $z\in \mathcal{Y}_L$, we get that
		$$W[\xi_{i_0j}^{(m)},\lambda_{j}^{(m)}/\lambda_{i_0}^{(m)}]W^{-1}[0,1]\leq(\lambda_{j}^{(m)}/\lambda_{i_0}^{(m)})^{-2s}(2L)^{8s}\rightarrow0\hspace{6mm}\mbox{as}\hspace{2mm}m\rightarrow\infty,$$
		as desired.
		
		Using the elementary inequality
		\begin{equation*}
			(\sum_{i=1}^{\kappa}b_i)^{-1}\sum_{i=1}^{\kappa}a_i\leq\max\Big\{\frac{a_i}{b_i},\cdots,\frac{a_\kappa}{b_\kappa}\Big\}\leq\sum_{i=1}^{\kappa}b_i^{-1}a_i
			\hspace{2mm}\mbox{for all}\hspace{2mm}a_1,\cdots,a_\kappa>0,\hspace{2mm} b_1,\cdots,b_\kappa>0,
		\end{equation*}
		it follows from \eqref{617}-\eqref{618} and the definition of weight functions that
		\begin{equation*}
			\begin{split}
				&\frac{|\phi_{m}(x)|}{S(x_m)}\leq\frac{\sum_{j\in\mathscr{H}(i_0)}^{\kappa}(s_{j,1}(x)+s_{j,2}(x))+s_{i_0,1}(x)}{\sum_{j\in\mathscr{H}(i_0)}^{\kappa}(s_{j,1}(x_{m})+s_{j,2}(x_{m}))+s_{i_0,1}(x_{m})}
				\\&\lesssim\frac{s_{i_0,1}(x)}{s_{i_0,1}(x_{m})}+\sum_{j\in\mathscr{H}(i_0)}^{\kappa}\bigg[\frac{s_{j,1}(x)}{s_{j,1}(x_{m})}_{\big\{\big|\lambda_{j}^{(m)}/\lambda_{i_0}^{(m)}(z-\xi_{i_0j}^{(\infty)})\big|\leq\mathscr{R}_m^2\big\}}
				+{\frac{s_{j,2}(x)}{s_{j,2}(x_{m})}}_{\big\{\big|\lambda_{j}^{(m)}/\lambda_{i_0}^{(m)}(z-\xi_{i_0j}^{(\infty)})\big|>\mathscr{R}_m^2\big\}}\bigg],
			\end{split}
		\end{equation*}
		\begin{equation*}
			\begin{split}
				&\frac{|g_m(x)|}{(\lambda_{i_0}^{(m)})^{2s}S_1(x_m)}\lesssim\frac{(\lambda_{i_0}^{(m)})^{-2s}\sum_{j\in\mathscr{H}(i_0)}^{\kappa}(t_{j,1}(x)+t_{j,2}(x))+t_{i_0,1}(x)
				}{\sum_{j\in\mathscr{H}(i_0)}^{\kappa}(s_{j,1}(x_{m})+s_{j,2}(x_{m}))+s_{i_0,1}(x_{m})}
				\lesssim\frac{(\lambda_{i_0}^{(m)})^{-2s}t_{i_0,1}(x)}{s_{i_0,1}(x_{m})}\\&+\sum_{j\in\mathscr{H}(i_0)}^{\kappa}\bigg[\frac{(\lambda_{i_0}^{(m)})^{-2s}t_{j,1}(x)}{s_{j,1}(x_{m})}_{\big\{\big|\lambda_{j}^{(m)}/\lambda_{i_0}^{(m)}(z-\xi_{i_0j}^{(\infty)})\big|\leq\mathscr{R}_m^2\big\}}
				+\frac{(\lambda_{i_0}^{(m)})^{-2s}t_{j,2}(x)}{s_{j,2}(x_{m})}_{\big\{\big|\lambda_{j}^{(m)}/\lambda_{i_0}^{(m)}(z-\xi_{i_0j}^{(\infty)})\big|>\mathscr{R}_m^2\big\}}\bigg].     \end{split}
		\end{equation*}
		Note that, there exists a large number $m_L\in\mathbb{N}$ such that
		\begin{equation*}
			z\in \mathcal{Y}_L,~j\in\mathscr{H}(i_0),~m\geq m_L \Rightarrow
			\big|\lambda_{j}^{(m)}/\lambda_{i_0}^{(m)}(z-\xi_{i_0j}^{(\infty)})\big|\gg\mathscr{R}_m^2.
		\end{equation*}
		Therefore, we obtain
		\begin{equation*}
			\frac{s_{i_0,1}(x)}{s_{i_0,1}(x_{m})}= \leq1+M^{2s},\hspace{2mm}\frac{(\lambda_{i_0}^{(m)})^{-2s}v_{i_0,1}(x)}{z_{i_0,1}(x_{m})}\leq1+M^{2s},\hspace{2mm}
		\end{equation*}
		for $z=\lambda_{i_0}(x-\xi_{i_0})$ and $\zeta_m=\lambda_{i_0}(x_m-\xi_{i_0})$.
		In the case $i_0\prec j$, $\lim\limits_{m\rightarrow\infty}\xi_{ij}^{(m)}$ exists, one has $\lambda_{j}^{(m)}/\lambda_{i_0}^{(m)}\rightarrow\infty$ as $m\rightarrow\infty$. It follows from \eqref{AA-00} and $|\zeta_{m}|\leq M$ that
		\begin{equation*}
			\frac{s_{j,2}(x)}{s_{j,2}(x_{m})}= \frac{\tau(\lambda_j(x_m-\xi_j))^{3s}}{\tau(\lambda_j(x-\xi_j))^{3s}}
			\leq \big(\frac{M}{|z-\xi_{i_0j}^{(\infty)}|}\big)^{3s},\hspace{2mm}
			\frac{(\lambda_{i_0}^{(m)})^{-2s}t_{j,2}(x)}{s_{j,2}(x_{m})}=
			\frac{\tau(\lambda_j(x_m-\xi_j))^{3s}}{\tau(\lambda_j(x-\xi_j))^{5s}}
			\leq L^{5s}M^{3s}.
		\end{equation*}
		Combining all these together, the result follows.
	\end{proof}
	Next, we conclude the proof of Step \ref{step5.2} by proving the following useful results.
	\begin{Prop}\label{converges-4}
		Up to a subsequence, we have
		\begin{equation*}
			\bar{\phi}_m(z)\rightarrow0 \hspace{2mm}\mbox{in}\hspace{2mm}C_{loc}^{0}\big(\mathbb{R}^{6s}\setminus\{\xi_{i_0j}^{(\infty)}:j\in\mathscr{H}(i_0)\}\big)\hspace{6mm}\mbox{as}\hspace{2mm} m\rightarrow\infty.
		\end{equation*}
	\end{Prop}
	\begin{proof}
		Note that, combining the standard elliptic regularity and diagonal argument, up to subsequence, we have
		\begin{equation*}	\bar{\phi}_m\rightarrow\bar{\phi}_{\infty}\hspace{2mm}\mbox{in}\hspace{2mm} C_{loc}^{0}\big(\mathbb{R}^{6s}\setminus\{\xi_{i_0j}^{(\infty)}:j\in\mathscr{H}(i_0)\}\big)
			\hspace{6mm}\mbox{as}\hspace{2mm} m\rightarrow\infty
		\end{equation*}
		for some function $\bar{\phi}_{\infty}$. Furthermore, by \eqref{km-1} we have
		$$\vert\bar{\phi}_{\infty}(z)\vert
		\lesssim\sum_{j\in\mathscr{H}(i_0)}^{\kappa}\big(\frac{M}{|z-\xi_{i_0j}^{(\infty)}|}\big)^{3s}+M^{2s},\hspace{4mm}         \mbox{in}\hspace{2mm}\mathbb{R}^{6s}\setminus\{\xi_{i_0j}^{(\infty)}:j\in\mathscr{H}(i_0)\},
		$$
		Then by the same argument as in \cite{YZ25}, the following result holds
		\begin{equation*}			    (-\Delta)^s\bar{\phi}_{\infty}-I_{6s,4s,s}[W[0,1],\bar{\phi}_{\infty}]
			=0,   \hspace{4mm} \text{in}\hspace{2mm}\mathbb{R}^{6s}\setminus\{\xi_{i_0j}^{(\infty)}:j\in\mathscr{H}(i_0)\}
		\end{equation*}
		and
		\begin{equation*}
			\displaystyle \int_{\mathbb{R}^{6s}} I_{6s,4s,s}[W[0,1],\mathcal{Z}^{a}[0,1]]\bar{\phi}_{\infty}=0,\hspace{4mm}\mbox{for all}\hspace{2mm}i=1,\cdots,\kappa\hspace{2mm}\mbox{and}\hspace{2mm}a=1,\cdots,n+1.
		\end{equation*}
		Furthermore, we can claim that there exists a function $h(z)$ satisfying
		\begin{equation*}
			\left\{
			\begin{array}{ll}
				(-\Delta)^s h-I_{6s,4s,s}[W[0,1],h]
				=0,        &\text{in}~\mathbb{R}^{6s}\setminus\{\xi_{i_0j}^{(\infty)}:j\in\mathscr{H}(i_0)\},            \\
				\vert h(z)\vert			\lesssim\sum_{j\in\mathscr{H}(i_0)}\big(\frac{M}{|z-\xi_{i_0j}^{(\infty)}|}\big)^{3s}+M^{2s},                  &\text{in}~\mathbb{R}^{6s}\setminus\{\xi_{i_0j}^{(\infty)}:j\in\mathscr{H}(i_0)\}.
			\end{array}
			\right.
		\end{equation*}
		Then $h\in L^{\infty}(\mathbb{R}^n).$
		Indeed, in order to conclude the proof of Proposition \ref{converges-4}, it suffices to obtain the boundness in the set $B(0,4\widetilde{C})\setminus\{\xi_{i_0j}^{(\infty)}:j\in\mathscr{H}(i_0)\}$, where $4\widetilde{C}:=1+\max_{j\in\mathscr{H}(i_0)}|\xi_{i_0j}^{(\infty)}|$
		For any $z\in B(0,4\widetilde{C})\setminus\{\xi_{i_0j}^{(\infty)}:j\in\mathscr{H}(i_0)\}$, we have
		\begin{equation*}
			\begin{split}			|h(z)|&\lesssim\int_{\mathbb{R}^{6s}}\frac{1}{|z-\omega|^{4s}}\Big[\Big(\int_{\mathbb{R}^{6s}}\frac{W[0,1]}{|\omega-y|^{4s}}dy\Big)
				W[0,1]+\Big(\int_{\mathbb{R}^{6s}}\frac{W[0,1]^2}{|\omega-y|^{4s}}dy\Big)
				\Big]d\omega\\&				+\sum_{j\in\mathscr{H}(i_0)}\int_{\mathbb{R}^{6s}}\frac{1}{|z-\omega|^{4s}}\Big(\int_{\mathbb{R}^{6s}}\frac{W[0,1]}{|\omega-y|^{4s}}\big(\frac{1}{|y-\xi_{i_0j}^{(\infty)}|}\big)^{3s}dy\Big)			W[0,1]d\omega\\&+\sum_{j\in\mathscr{H}(i_0)}\int_{\mathbb{R}^{6s}}\frac{1}{|z-\omega|^{4s}}\Big(\int_{\mathbb{R}^{6s}}\frac{W[0,1]^2}{|\omega-y|^{4s}}dy\Big)				\big(\frac{1}{|\omega-\xi_{i_0j}^{(\infty)}|}\big)^{3s}\Big]d\omega.
			\end{split}
		\end{equation*}
		Now, following the similar process to \cite{YZ25} with minor modifications, we eventually arrive at
		$$|g(z)|\leq C\Big(1+\sum_{j\in\mathscr{H}(i_0)}\frac{1}{|z-\xi_{i_0j}^{(\infty)}|^{\alpha_1}}+\sum_{j\in\mathscr{H}(i_0)}\frac{1}{|z-\xi_{i_0j}^{(\infty)}|^{\alpha_2}}\Big) \hspace{6mm}\text{in}~\mathbb{R}^{6s}\setminus\{\xi_{i_0j}^{(\infty)}:j\in\mathscr{H}(i_0)\},
		$$
		for some positive constants $\alpha_1$ and $\alpha_2$. Substituting this bound into the previous estimate of $|h(z)|$ and the claim then follows by iterating the above argument.
		Thus the conclusion follows from the orthogonal condition and non-degeneracy of $W_i$.
	\end{proof}

	\begin{proof}[Proof of Step \ref{step5.2}]
		Since $\vert \zeta_m\vert:=\vert\lambda_{i_0}^{(m)}(x_m-\xi_{i_0j}^{(m)})\vert\leq M$ and $\vert \zeta_m-\xi_{i_0j}^{\left(m\right)}\vert\geq\varepsilon_0$, up to a subsequence, we have $\lim_{m\rightarrow\infty}\zeta_m=\zeta_\infty\notin\{\xi_{i_0j}^{(\infty)}:j\in\mathscr{H}(i_0)\}$. However, this leads to a contradiction with $\bar{\phi}_{\infty}\equiv0$ by Proposition \ref{converges-4}.
	\end{proof}
	
	\begin{step}\label{step5.3}
		When $\left\{x_m\right\}\subset\mathcal{C}_{neck,M,3}:=\bigcup\limits_{i\in I}D_i^{(m)},$
		we have that
		\begin{equation*}
			|\phi_{m}|(x)<\frac{2}{3}S(x)\hspace{4mm}\text{as}\hspace{2mm}m\rightarrow+\infty.
		\end{equation*}
	\end{step}
	\begin{proof}[Proof of Step \ref{step5.3}]
		By \eqref{gelinhanshu} and \eqref{diyige}, we know
		\begin{equation*}
			|\phi_m|(\tilde{x})
			\lesssim\Big|\int_{\mathbb{R}^{6s}}\frac{1}{|\tilde{x}-x|^{4s}}I_{6s,4s,s}[\sigma_m,\phi_m]dx\Big|
			+o(1)S(\tilde{x}).
		\end{equation*}
		It follows from Step \ref{step5.1} and Step \ref{step5.2} that, provided $M$ is large enough,
		\begin{equation*}
			\begin{split}
				&\bigg|\int_{\mathbb{R}^{6s}}\frac{1}{|\tilde{x}-x|^{4s}}\bigg(\int_{\mathbb{R}^{6s}}\frac{\sigma_m^2(y)
				}{|x-y|^{4s}}dy\phi_m(x)+\int_{\mathbb{R}^{6s}}\frac{2\sigma_m(y)\phi_m(y)}{|x-y|^{4s}}dy\sigma_m(x)\bigg)dx\bigg|       \\
				\leq&
				\bigg|\int_{\bigcup\limits_{i\in I}D_i^{(m)}}\frac{1}{|\tilde{x}-x|^{4s}}\bigg(\int_{\mathbb{R}^{6s}}\frac{\sigma_m^2(y)
				}{|x-y|^{4s}}dy\phi_m(x)+\int_{\mathbb{R}^{6s}}\frac{2\sigma_m(y)\phi_m(y)}{|x-y|^{4s}}dy\sigma_m(x)\bigg)dx\bigg|
				+\bigg(\frac{1}{6}+o(1)\bigg)S(\tilde{x}).
			\end{split}
		\end{equation*}
		Moreover, from Step \ref{step5.1} we have
		\begin{equation*}
			\begin{split}
				&\bigg|\int_{D_i^{(m)}}\frac{1}{|\tilde{x}-x|^{4s}}\bigg(\int_{\mathbb{R}^{6s}}\frac{\sigma_m^2(y)
				}{|x-y|^{4s}}dy\phi_m(x)+\int_{\mathbb{R}^{6s}}\frac{2\sigma_m(y)\phi_m(y)}{|x-y|^{4s}}dy\sigma_m(x)\bigg)dx\bigg|     \\
				\leq&K^{8s}\sum_{j=1}^{\kappa}\sum_{l>j}\int_{D_i^{(m)}\cap\big\{\frac{\lambda_{l}}{K\lambda_{j}}\leq|z_l|\leq\frac{K\lambda_{l}}{\lambda_{j}},|z_j|\leq2K\big\}}\frac{1}{|\tilde{z_l}-z_l|^{4s}}\frac{\lambda_{l}^{2s}}{\tau(z_j)^{12s}\mathscr{R}_m^{4s}}dz_l    \\
				&+\mathscr{R}_m^{-4s}\sum_{j=1}^\kappa\int_{D_i^{(m)}}\frac{1}{|\tilde{z_j}-z_j|^{4s}}\frac{\lambda_{j}^{4s}\log\sqrt{1+|z_j|^2}}{\tau(z_j)^{8s}}dz_j.
			\end{split}
		\end{equation*}
		Let $M$ large enough, if $j>i$, then we have $D_i^{(m)}\cap\big\{\frac{\lambda_{l}}{K\lambda_{j}}\leq|z_l|\leq\frac{K\lambda_{l}}{\lambda_{j}},|z_j|\leq2K\big\}=\emptyset$. Consequently,
		\begin{equation*}
			\begin{split}
				&\bigg|\int_{D_i^{(m)}}\frac{1}{|\tilde{x}-x|^{4s}}\bigg(\int_{\mathbb{R}^{6s}}\frac{\sigma_m^2(y)
				}{|x-y|^{4s}}dy\phi_m(x)+\int_{\mathbb{R}^{6s}}\frac{2\sigma_m(y)\phi_m(y)}{|x-y|^{4s}}dy\sigma_m(x)\bigg)dx\bigg|       \\
				\leq&\sum_{j>i}\int_{D_i^{(m)}\cap\big\{\frac{\lambda_{l}}{K\lambda_{j}}\leq|z_l|\leq\frac{K\lambda_{l}}{\lambda_{j}},|z_j|\leq2K\big\}}\frac{1+o(1)}{|\tilde{z_j}-z_j|^{4s}}\frac{M^{8s}\lambda_{j}^{2s}}{\tau(z_i)^{12s}}dz_j
				+(\varepsilon^{2s}+o(1))s_{i,1}(\tilde{x})
				+M^{-2s}\sum_{j>i}\big(s_{j,1}(\tilde{x})+s_{j,2}(\tilde{x})\big)   \\
				\leq&\big(\varepsilon^{2s}+\varepsilon^{2s}M^{8s}+o(1)\big)s_{i,1}(\tilde{x})
				+M^{-2s}\sum_{j>i}\big(s_{j,1}(\tilde{x})+s_{j,2}(\tilde{x})\big).
			\end{split}
		\end{equation*}
		Thus for $x\in D_{i_0}^{(m)}$, let $\varepsilon$ small enough we can deduce
		\begin{equation*}
			\begin{split}
				&\bigg|\int_{\mathbb{R}^{6s}}\frac{1}{|\tilde{x}-x|^{4s}}\bigg(\int_{\mathbb{R}^{6s}}\frac{\sigma_m^2(y)
				}{|x-y|^{4s}}dy\phi_m(x)+\int_{\mathbb{R}^{6s}}\frac{2\sigma_m(y)\phi_m(y)}{|x-y|^{4s}}dy\sigma_m(x)\bigg)dx\bigg|       \\
				\leq&\frac{1}{6}S(\tilde{x})
				+\sum_{i=1}^\kappa\bigg(\big(\varepsilon^{2s}+\varepsilon^{2s}K^{8s}+o(1)\big)s_{i,1}(\tilde{x},\mathscr{R}_m)
				+M^{-2s}\sum_{j>i}\big(s_{j,1}(\tilde{x},\mathscr{R}_m)+s_{j,2}(\tilde{x},\mathscr{R}_m)\big)\bigg)
				\leq\frac{1}{3}S(\tilde{x}).
			\end{split}
		\end{equation*}
		In conclusion, we have that
		\begin{equation*}
			|\phi_m|(x)
			\leq\frac{2}{3}S(x)\hspace{6mm}for\hspace{2mm}x\in\cup_{i\in I}D_i^{(m)}.
		\end{equation*}
	\end{proof}
	
	\begin{proof}[Proof of Lemma \ref{estimate2}]
		It follows from Step \ref{step5.1}, Step \ref{step5.2} and Step \ref{step5.3} that $\|\phi_m\|_\ast<1$, which contradicts the assumption $\|\phi_m\|_\ast=1$. As a result, Lemma \ref{estimate2} is proven.
	\end{proof}
	
	\section{Proofs of Theorem \ref{Figalli} and Corollary \ref{Figalli2}}\label{section7}
	In this section, we will prove our main Theorem \ref{Figalli} and Corollary \ref{Figalli2} based on some crucial energy estimates.
	\subsection{Energy estimate of the first approximation}
	\begin{lem}\label{p00}
		Assume that $s\in(0,\frac{n}{2})$, $n>2s$, $0<\mu<n$ with $0<\mu\leq4s$. For $\delta$ is small enough, we have the following estimate:
		\begin{equation*}
			\big\|\varrho_0\big\|_{\dot{H}^s(\mathbb{R}^n)}\lesssim\mathcal{K}_{n,\mu,s}(\mathscr{Q}^{\min\{\frac{\mu}{n-2s},1\}}),
		\end{equation*}
		where  $\mathcal{K}_{n,\mu,s}$ is the piece-wise function defined in \eqref{hanshu}.
	\end{lem}
	\begin{proof}
		We begin from the equation \eqref{coefficients11}. Multiplying this equation by $\varrho_0$ and integrating by parts, we obtain
		\begin{equation}\label{0AA}	
			\|\varrho_0\|_{\dot{H}^s(\mathbb{R}^n)}^2
			=\int I_{n,\mu,s}[\sigma,\varrho_0]\varrho_0
			+\int N(\varrho_0)\varrho_0
			+\int g\varrho_0,
		\end{equation}
		where we use the orthogonal condition. When $n=6s$, $0<\mu<4s$ or $n\neq6s$, $0<\mu\leq4s$, it follows from Lemma \ref{qiegao} and Lemma \ref{zhaji} that $\|g\|_{L^{(2_s^\ast)^\prime}}\lesssim\mathscr{Q}^{\min\{\frac{\mu}{n-2s},1\}}$ and $\|\varrho_{0}\|_{\dot{H}^s(\mathbb{R}^n)}\lesssim\mathscr{Q}^{\min\{\frac{\mu}{n-2s},1\}}$. Therefore,
		\begin{equation*}
			\begin{split}
				\Big|\int I_{n,\mu,s}[\sigma,\varrho_0]\varrho_0\Big|
				&=\Big|p_s\int\big(|x|^{-\mu}\ast \sigma^{p_s-1}\varrho_0\big)
				\sigma^{p_s-1}\varrho_0+(p_s-1)\int\big(|x|^{-\mu}\ast\sigma^{p_s}\big)
				\sigma^{p_s-2}\varrho_0^2\Big|      \\
				&\lesssim\|\sigma\|_{L^{2_s^\ast}}^{2(p_s-1)}\|\varrho_0\|_{L^{2_s^\ast}}^2
				\lesssim\mathscr{Q}^{2\min\{\frac{\mu}{n-2s},1\}},
			\end{split}
		\end{equation*}
		\begin{equation*}
			\Big|\int N(\varrho_0)\varrho_0\Big|
			\lesssim o(1)\|\varrho_0\|_{L^{2_s^\ast}}^2
			\lesssim o(1)\mathscr{Q}^{2\min\{\frac{\mu}{n-2s},1\}},
		\end{equation*}
		\begin{equation*}
			\Big|\int g\varrho_0\Big|
			\lesssim\|g\|_{L^{(2_s^\ast)^\prime}}\|\varrho_{0}\|_{\dot{H}^s(\mathbb{R}^n)}
			\lesssim\mathscr{Q}^{2\min\{\frac{\mu}{n-2s},1\}}.
		\end{equation*}
		Hence, the proof is complete by putting the above estimates together.
		
		When $n=6s$ and $\mu=4s$, it follows from Lemma \ref{estimate1} and Lemma \ref{xianyuxian} that $|g(x)|\lesssim T(x)$ and $|\varrho_0(x)|\lesssim S(x)$. By Lemma \ref{cll-0-0}, we have $\|S\|_{L^{2_s^\ast}}\lesssim\mathscr{Q}|\log\mathscr{Q}|^\frac{1}{2}$ and $\|S\|_{L^{2_s^{\ast}}}\|T\|_{L^{(2_s^\ast)^\prime}}\lesssim\mathscr{Q}^2|\log\mathscr{Q}|$. Therefore,
		\begin{equation*}
			\Big|\int_{\mathbb{R}^{6s}}I_{6s,4s,s}[\sigma,\varrho_0]\varrho_0\Big|
			\lesssim\|\sigma\|_{L^{2_s^\ast}}^2\|\varrho_0\|_{L^{2_s^\ast}}^2
			\lesssim\mathscr{Q}^2|\log\mathscr{Q}|,
		\end{equation*}
		\begin{equation*}
			\Big|\int_{\mathbb{R}^{6s}} N(\varrho_0)\varrho_0\Big|
			\lesssim o(1)\|\varrho_0\|_{L^{2_s^\ast}}^2
			\lesssim o(1)\mathscr{Q}^2|\log\mathscr{Q}|,
		\end{equation*}
		\begin{equation*}
			\Big|\int_{\mathbb{R}^{6s}} g\varrho_0\Big|\lesssim\int ST\leq\big\|S\big\|_{L^{2_s^{\ast}}}\big\|T\big\|_{L^{(2_s^\ast)^\prime}}\lesssim\mathscr{Q}^2|\log\mathscr{Q}|.
		\end{equation*}
		Plugging in the above inequalities to \eqref{0AA}, the proof is completed.

	\end{proof}

	\subsection{Energy estimate of the second approximation}
	First we need to decompose the error function $\varrho$. Recall that $\varrho$ and $\varrho_{0}$ satisfy \eqref{u-0} and \eqref{coefficients11}, respectively.
	Thus, $\varrho_1=\varrho-\varrho_0$ solves
	\begin{equation}\label{c1-000}
		\left\{\begin{array}{l}
			\displaystyle (-\Delta)^s \varrho_1-\big(|x|^{-\mu}\ast(\sigma+\varrho_0+\varrho_1)^{p_s}\big)(\sigma+\varrho_0+\varrho_1)^{p_s-1}
			+\Big(|x^{-\mu}\ast(\sigma+\varrho_0)^{p_s}\Big)(\sigma+\varrho_0)^{p_s-1}
			\\
			\displaystyle+\sum_{i=1}^{\kappa}\sum_{a=1}^{n+1}c_{a}^{i}I_{n,\mu,s}[W_{i},\mathcal{Z}^{a}_i]-\hat{f}=0,\\
			\displaystyle \int I_{n,\mu,s}[W_{i},\mathcal{Z}^{a}_i]\varrho_1=0,\hspace{2mm}i=1,\cdots, \kappa; ~a=1,\cdots,n+1,
		\end{array}
		\right.
	\end{equation}
	where $\hat{f}:=(-\Delta)^su-\big(|x|^{-\mu}\ast |u|^{p_s}\big)|u|^{p_s-2}u$. Moreover, there exist appropriate constants $\gamma^i$ such that the following decomposition hold:
	\begin{equation}\label{c1-001}
		\varrho_1=\sum_{i=1}^{\kappa}\gamma^iW_i+\varrho_2\hspace{4mm}\text{with}\hspace{4mm}\gamma^i=\int(-\Delta)^\frac{s}{2}\varrho_1(-\Delta)^\frac{s}{2}W_i.
	\end{equation}
	Then for all $i=1,\cdots, \kappa; ~a=1,\cdots,n+1$, we achieve
	$$\int(-\Delta)^\frac{s}{2}\varrho_2\cdot(-\Delta)^\frac{s}{2}W_i=\int(-\Delta)^\frac{s}{2}\varrho_2\cdot(-\Delta)^\frac{s}{2}\mathcal{Z}_{i}^{a}=0.$$
	\begin{lem}\label{Ni-1-3}
		Assume that $s\in(0,\frac{n}{2})$, $n>2s$, $0<\mu<n$ with $0<\mu\leq4s$. We have that
		$$\big\|\varrho_2\big\|_{\dot{H}^s(\mathbb{R}^n)}\lesssim\sum_{i=1}^{\kappa}|\gamma^i|+\big\|\hat{f}\big\|_{\dot{H}^{-s}(\mathbb{R}^n)}.$$
	\end{lem}
	\begin{proof}
		Multiplying \eqref{c1-000} by $\varrho_2$ and integrating by parts, we have
		\begin{equation*}
			\begin{split}
				\|\varrho_2\|_{\dot{H}^s(\mathbb{R}^n)}^2
				=&\int\big(|x|^{-\mu}\ast(\sigma+\varrho_0+\varrho_1)^{p_s}\big)(\sigma+\varrho_0+\varrho_1)^{p_s-1}\varrho_2
				-\int\big(|x|^{-\mu}\ast(\sigma+\varrho_0)^{p_s}\big)(\sigma+\varrho_0)^{p_s-1}\varrho_2
				+\int\hat{f}\varrho_2         \\
				=&\int\big(|x|^{-\mu}\ast\big[(\sigma+\varrho_0+\varrho_1)^{p_s}-(\sigma+\varrho_0)^{p_s}-p_s(\sigma+\varrho_0)^{p_s-1}\varrho_1\big]\big)\big(\sigma+\varrho_0+\varrho_1\big)^{p_s-1}\varrho_2     \\
				&+\int\big(|x|^{-\mu}\ast(\sigma+\varrho_0)^{p_s}\big)\big[(\sigma+\varrho_0+\varrho_1)^{p_s-1}-(\sigma+\varrho_0)^{p_s-1}-(p_s-1)(\sigma+\varrho_0)^{p_s-2}\varrho_1\big]\varrho_2    \\
				&+p_s\int\big(|x|^{-\mu}\ast(\sigma+\varrho_0)^{p_s-1}\varrho_1\big)\big[(\sigma+\varrho_0+\varrho_1)^{p_s-1}-(\sigma+\varrho_0)^{p_s-1}\big]\varrho_2    \\
				&+\int I_{n,\mu,s}[\sigma+\varrho_0,\varrho_1]\varrho_2+\int\hat{f}\varrho_2     \\
				=&:J_1+J_2+J_3+\int I_{n,\mu,s}[\sigma+\varrho_0,\varrho_1]\varrho_2+\int\hat{f}\varrho_2.
			\end{split}
		\end{equation*}
		Since the elementary inequalities
		\begin{equation*}
			(\sigma+\varrho_0+\varrho_1)^{p_s}-(\sigma+\varrho_0)^{p_s}-p_s(\sigma+\varrho_0)^{p_s-1}\varrho_1
			\lesssim(\sigma+\varrho_0)^{p_s-2}\varrho_1^2
			+|\varrho_1|^{p_s},
		\end{equation*}
		\begin{equation*}
			(\sigma+\varrho_0+\varrho_1)^{p_s-1}-(\sigma+\varrho_0)^{p_s-1}-(p_s-1)(\sigma+\varrho_0)^{p_s-2}\varrho_1
			\lesssim\varrho_1^{p_s-1},
		\end{equation*}
		together with the H\"older inequality, Hardy-Littlewood-Sobolev inequality and Sobolev inequality, we get
		\begin{equation*}
			\Big|\int\hat{f}\varrho_2\Big|
			\leq\|\varrho_2\|_{\dot{H}^s(\mathbb{R}^n)}\|\hat{f}\|_{\dot{H}^{-s}(\mathbb{R}^n)},
		\end{equation*}
		\begin{equation*}
			|J_1+J_2+J_3|
			\lesssim
			\left\lbrace
			\begin{aligned}
				&\bigg(\mathcal{G}+\|\varrho_2\|_{\dot{H}^s(\mathbb{R}^n)}\bigg)^2\|\varrho_2\|_{\dot{H}^s(\mathbb{R}^n)}\hspace{12mm}if\hspace{2mm}\mu=4s,    \\
				&\bigg(\mathcal{G}+\|\varrho_2\|_{\dot{H}^s(\mathbb{R}^n)}\bigg)^{p_s-1}\|\varrho_2\|_{\dot{H}^s(\mathbb{R}^n)}\hspace{7mm}if\hspace{2mm}0<\mu<4s,
			\end{aligned}
			\right.
		\end{equation*}
		where $\mathcal{G}=\sum_{i=1}^\kappa|\gamma^i|$. Next we consider $\int I_{n,\mu,s}[\sigma+\varrho_0,\varrho_1]\varrho_2$. Noticing the decomposition of $\varrho_1$ in \eqref{c1-001}, we have
		$|\varrho_1|\leq C\mathcal{G}\sigma+|\varrho_2|$ and
		\begin{equation*}
			\begin{split}
				&\int I_{n,\mu,s}[\sigma+\varrho_0,\varrho_1]\varrho_2    \\
				\leq& C\mathcal{G}\int\big(|x|^{-\mu}\ast(\sigma+\varrho_0)^{p_s-1}\sigma\big)(\sigma+\varrho_0)^{p_s-1}|\varrho_2|+\big(|x|^{-\mu}\ast(\sigma+\varrho_0)^{p_s}\big)(\sigma+\varrho_0)^{p_s-2}\sigma|\varrho_2|     \\
				&+\int p_s\big(|x|^{-\mu}\ast(\sigma+\varrho_0)^{p_s-1}|\varrho_2|\big)(\sigma+\varrho_0)^{p_s-1}|\varrho_2|
				+(p_s-1)\big(|x|^{-\mu}\ast(\sigma+\varrho_0)^{p_s}\big)(\sigma+\varrho_0)^{p_s-2}\varrho_2^2   \\
				=:&~ C\mathcal{G}J_4+J_5
			\end{split}
		\end{equation*}
		For $J_4$, we apply the H\"older inequality, Hardy-Littlewood-Sobolev inequality and Sobolev inequality to get
		\begin{equation}\label{lajidai}
			|J_4|
			\lesssim\|\sigma+\varrho_0\|_{L^{2_s^\ast}}^{2p_s-2}\|\sigma\|_{L^{2_s^\ast}}\|\varrho_2\|_{L^{2_s^\ast}}
			\lesssim\|\varrho_2\|_{\dot{H}^s},
		\end{equation}
		here in the last inequality we have used Lemma \ref{p00} and the fact $\|W_i\|_{L^{2_s^\ast}}\lesssim1$ for all $1\leq i\leq\kappa$. For $|J_5|$, it follows from Lemma \ref{estim} that there exists a positive constant $\tau_0<1$ such that
		\begin{equation*}
			p_{s}\int\big(|x|^{-\mu}\ast\sigma^{p_{s}-1}\varrho_2\big)\sigma^{p_s-1}\varrho_2
			+(p_{s}-1)\int\big(|x|^{-\mu}\ast \sigma^{p_{s}}\big)\sigma^{p_{s}-2}\varrho_2^2
			\leq\tau_0\big\|\varrho_2\big\|_{\dot{H}^s}^2.
		\end{equation*}
		Therefore, we have
		\begin{equation}\label{lajifang}
			|J_5|
			\leq\tau_0\big\|\varrho_2\big\|_{\dot{H}^s}^2
			+|J_6|+|J_7|+|J_8|+|J_9|,
		\end{equation}
		where
		\begin{equation*}
			J_6:=p_s\int\big(|x|^{-\mu}\ast\big[(\sigma+\varrho_0)^{p_s-1}-\sigma^{p_s-1}\big]\varrho_2\big)\big(\sigma+\varrho_0\big)^{p_s-1}\varrho_2,
		\end{equation*}
		\begin{equation*}
			J_7:=p_s\int\big(|x|^{-\mu}\ast\sigma^{p_s-1}\varrho_2\big)\big[(\sigma+\varrho_0)^{p_s-1}-\sigma^{p_s-1}\big]\varrho_2,
		\end{equation*}
		\begin{equation*}
			J_8:=(p_s-1)\int\big(|x|^{-\mu}\ast\big[(\sigma+\varrho_0)^{p_s}-\sigma^{p_s}\big]\big)\big(\sigma+\varrho_0\big)^{p_s-2}\varrho_2^2,
		\end{equation*}
		\begin{equation*}
			J_9:=(p_s-1)\int\big(|x|^{-\mu}\ast\sigma^{p_s}\big)\big[(\sigma+\varrho_0)^{p_s-2}-\sigma^{p_s-2}\big]\varrho_2^2.
		\end{equation*}
		Then by the H\"older inequality, Hardy-Littlewood-Sobolev inequality and Sobolev inequality, we obtain
		\begin{equation}\label{lajiche}
			|J_6|+|J_7|+|J_8|+|J_9|
			\leq
			\left\lbrace
			\begin{aligned}
				&C\|\varrho_0\|_{\dot{H}^s(\mathbb{R}^n)}^2\|\varrho_2\|_{\dot{H}^s(\mathbb{R}^n)}^2\hspace{6mm}if\hspace{2mm}\mu=4s,\\
				&C\|\varrho_0\|_{\dot{H}^s(\mathbb{R}^n)}^{p_s-2}\|\varrho_2\|_{\dot{H}^s(\mathbb{R}^n)}^2\hspace{6mm}if\hspace{2mm}0<\mu<4s.
			\end{aligned}
			\right.
		\end{equation}
		Combining \eqref{lajidai}-\eqref{lajiche}, we can deduce that
		\begin{equation*}
			\Big|\int I_{n,\mu,s}[\sigma+\varrho_0,\varrho_1]\varrho_2\Big|
			\leq
			\left\lbrace
			\begin{aligned}
				&\bigg(\tau_0+C\|\varrho_0\|_{\dot{H}^s(\mathbb{R}^n)}^2\bigg)\|\varrho_2\|_{\dot{H}^s(\mathbb{R}^n)}^2+C\mathcal{G}\|\varrho_2\|_{\dot{H}^s(\mathbb{R}^n)}\hspace{6mm}if\hspace{2mm}\mu=4s,       \\
				&\bigg(\tau_0+C\|\varrho_0\|_{\dot{H}^s(\mathbb{R}^n)}^{p_s-2}\bigg)\|\varrho_2\|_{\dot{H}^s(\mathbb{R}^n)}^2+C\mathcal{G}\|\varrho_2\|_{\dot{H}^s(\mathbb{R}^n)}\hspace{6mm}if\hspace{2mm}0<\mu<4s.
			\end{aligned}
			\right.
		\end{equation*}
		By Lemma \ref{p00}, we can assume $\|\varrho_0\|_{\dot{H}^s(\mathbb{R}^n)}\ll1$. Thus, we have
		\begin{equation*}
			\|\varrho_2\|_{\dot{H}^s(\mathbb{R}^n)}^2
			\lesssim
			\left\lbrace
			\begin{aligned}
				&\bigg(\mathcal{G}+\|\varrho_2\|_{\dot{H}^s(\mathbb{R}^n)}\bigg)^2\|\varrho_2\|_{\dot{H}^s(\mathbb{R}^n)}+\|\varrho_2\|_{\dot{H}^s(\mathbb{R}^n)}\|\hat{f}\|_{\dot{H}^{-s}(\mathbb{R}^n)}+\mathcal{G}\|\varrho_2\|_{\dot{H}^s(\mathbb{R}^n)}\hspace{11mm}if\hspace{2mm}\mu=4s,    \\
				&\bigg(\mathcal{G}+\|\varrho_2\|_{\dot{H}^s(\mathbb{R}^n)}\bigg)^{p_s-1}\|\varrho_2\|_{\dot{H}^s(\mathbb{R}^n)}+\|\varrho_2\|_{\dot{H}^s(\mathbb{R}^n)}\|\hat{f}\|_{\dot{H}^{-s}(\mathbb{R}^n)}+\mathcal{G}\|\varrho_2\|_{\dot{H}^s(\mathbb{R}^n)}\hspace{6mm}if\hspace{2mm}0<\mu<4s.
			\end{aligned}
			\right.
		\end{equation*}
		Dividing $\|\varrho_2\|_{\dot{H}^s(\mathbb{R}^n)}$ on both sides (unless $\varrho_2\equiv0$, when there is nothing to pvove), we have
		\begin{equation*}
			\|\varrho_2\|_{\dot{H}^s(\mathbb{R}^n)}
			\lesssim\left\lbrace
			\begin{aligned}
				&\bigg(\mathcal{G}+\|\varrho_2\|_{\dot{H}^s(\mathbb{R}^n)}\bigg)^2+\|\hat{f}\|_{\dot{H}^{-s}(\mathbb{R}^n)}+\mathcal{G}\hspace{11mm}if\hspace{2mm}\mu=4s,    \\
				&\bigg(\mathcal{G}+\|\varrho_2\|_{\dot{H}^s(\mathbb{R}^n)}\bigg)^{p_s-1}+\|\hat{f}\|_{\dot{H}^{-s}(\mathbb{R}^n)}+\mathcal{G}\hspace{6mm}if\hspace{2mm}0<\mu<4s.
			\end{aligned}
			\right.
		\end{equation*}
		Then by choosing $\delta$ small enough, together with Lemma \ref{p00} we can obtain $\mathcal{G}\ll1$ and $\|\varrho_2\|_{\dot{H}^s(\mathbb{R}^n)}\ll1$. Thus we are able to conclude the desired estimate
		\begin{equation*}
			\|\varrho_2\|_{\dot{H}^s(\mathbb{R}^n)}
			\lesssim\mathcal{G}+\|\hat{f}\|_{\dot{H}^{-s}(\mathbb{R}^n)}.
		\end{equation*}
	\end{proof}
	
	\begin{lem}\label{rr-1}
		Given $s\in(0,\frac{n}{2})$, $n>2s$, $0<\mu<n$ with $0<\mu\leq4s$. If $\delta$ is small enough, then
		\begin{equation*}	
			|\gamma^i|
			\lesssim\big\|\hat{f}\big\|_{\dot{H}^{-s}(\mathbb{R}^n)}
			+\mathscr{Q}^{1+\min\{\frac{\mu}{n-2s},1\}},\hspace{2mm}i=1,\cdots, \kappa.
		\end{equation*}
	\end{lem}
	\begin{proof}
		Multiplying \eqref{c1-000} by $W_k$ and integrating, we have
		\begin{equation}\label{rr-2}
			\int(-\Delta)^\frac{s}{2}\varrho_1(-\Delta)^\frac{s}{2} W_k
			=\mathcal{F}_{n,\mu,s}(x)
			+\int\sum_{i=1}^{\kappa}\sum_{a=1}^{n+1}c_{a}^{i}I_{n,\mu,s}[W_{i},\mathcal{Z}^{a}_i]W_k+\int \hat{f}W_k,
		\end{equation}
		where we denote in what follows
		\begin{equation*}
			\begin{split}
				\mathcal{F}_{n,\mu,s}(x)
				&:=\int\big[\big(|x|^{-\mu}\ast(\sigma+\varrho_0+\varrho_1)^{p_s}\big)(\sigma+\varrho_0+\varrho_1)^{p_s-1}-\big(|x|^{-\mu}\ast(\sigma+\varrho_0)^{p_s}\big)(\sigma+\varrho_0)^{p_s-1}\big]W_k     \\
				&=p_s\int\big(|x|^{-\mu}\ast W_k^{p_s-1}\varrho_1\big)W_k^{p_s}
				+(p_s-1)\int\big(|x|^{-\mu}\ast W_k^{p_s}\big)W_k^{p_s-1}\varrho_1+J_1+J_2+J_3+J_4+J_5+J_6,
			\end{split}
		\end{equation*}
		where
		\begin{equation*}
			J_1:=\int\big(|x|^{-\mu}\ast\big[(\sigma+\varrho_0+\varrho_1)^{p_s}-(\sigma+\varrho_0)^{p_s}-p_s(\sigma+\varrho_0)^{p_s-1}\varrho_1\big]\big)\big(\sigma+\varrho_0+\varrho_1\big)^{p_s-1}W_k,
		\end{equation*}
		\begin{equation*}
			J_2:=\int\big(|x|^{-\mu}\ast(\sigma+\varrho_0)^{p_s}\big)\big[(\sigma+\varrho_0+\varrho_1)^{p_s-1}-(\sigma+\varrho_0)^{p_s-1}-(p_s-1)(\sigma+\varrho_0)^{p_s-2}\varrho_1\big]W_k,
		\end{equation*}
		\begin{equation*}
			J_3:=p_s\int\big(|x|^{-\mu}\ast\big[(\sigma+\varrho_0)^{p_s-1}-W_k^{p_s-1}\big]\varrho_1\big)\big(\sigma+\varrho_0+\varrho_1\big)^{p_s-1}W_k,
		\end{equation*}
		\begin{equation*}
			J_4:=p_s\int\big(|x|^{-\mu}\ast W_k^{p_s-1}\varrho_1\big)\big[(\sigma+\varrho_0+\varrho_1)^{p_s-1}-W_k^{p_s-1}\big]W_k,
		\end{equation*}
		\begin{equation*}
			J_5:=(p_s-1)\int\big(|x|^{-\mu}\ast(\sigma+\varrho_0)^{p_s}\big)\big[(\sigma+\varrho_0)^{p_s-2}-W_k^{p_s-2}\big]\varrho_1W_k,
		\end{equation*}
		\begin{equation*}
			J_6:=(p_s-1)\int\big(|x|^{-\mu}\ast\big[(\sigma+\varrho_0)^{p_s}-W_k^{p_s}\big]\big)W_k^{p_s-1}\varrho_1.
		\end{equation*}
		By the H\"older inequality, Hardy-Littlewood-Sobolev inequality and Sobolev inequality, together with Lemma \ref{Ni-1-3} and Lemma \ref{FPU1}, we deduce that
		\begin{equation*}
			\|\sigma+\varrho_0\|_{L^{2_s^\ast}(\mathbb{R}^n)}\lesssim1,\hspace{6mm}\|\varrho_0\|_{L^{2_s^\ast}(\mathbb{R}^n)}=o(1),\hspace{6mm}\|W_k\|_{L^{2_s^\ast}(\mathbb{R}^n)}\lesssim1,
		\end{equation*}
		\begin{equation*}
			\bigg\|\sum_{i\neq k}W_i^\frac{p_s-1}{p_s}W_k^\frac{1}{p_s}\bigg\|_{L^{2_s^\ast}(\mathbb{R}^n)}=o(1),\hspace{6mm}\bigg\|\sum_{i\neq k}W_i^\frac{p_s-2}{p_s}W_k^\frac{2}{p_s}\bigg\|_{L^{2_s^\ast}(\mathbb{R}^n)}=o(1),
		\end{equation*}
		\begin{equation*}
			\Big|(\sigma+\varrho_0+\varrho_1)^{p_s}-(\sigma+\varrho_0)^{p_s}-p_s(\sigma+\varrho_0)^{p_s-1}\varrho_1\Big|
			\lesssim(\sigma+\varrho_0)^{p_s-2}\varrho_1^2,
		\end{equation*}
		\begin{equation*}
			\Big|(\sigma+\varrho_0+\varrho_1)^{p_s-1}-(\sigma+\varrho_0)^{p_s-1}-(p_s-1)(\sigma+\varrho_0)^{p_s-2}\varrho_1\Big|
			\lesssim|\varrho_1|^{p_s-1},
		\end{equation*}
		\begin{equation*}
			\Big|(\sigma+\varrho_0)^{p_s-1}-W_k^{p_s-1}\Big|
			\lesssim\Big(\sum_{i\neq k}W_i^{p_s-2}+|\varrho_0|^{p_s-2}\Big)W_k,
		\end{equation*}
		\begin{equation*}
			\Big|(\sigma+\varrho_0+\varrho_1)^{p_s-1}-W_k^{p_s-1}\Big|
			\lesssim\Big(\sum_{i\neq k}W_i^{p_s-2}+|\varrho_0|^{p_s-2}+|\varrho_1|^{p_s-2}\Big)W_k,
		\end{equation*}
		\begin{equation*}
			\Big|(\sigma+\varrho_0)^{p_s-2}-W_k^{p_s-2}\Big|
			\lesssim\sum_{i\neq k}W_i^{p_s-2}+|\varrho_0|^{p_s-2},
		\end{equation*}
		\begin{equation*}
			\Big|(\sigma+\varrho_0)^{p_s}-W_k^{p_s}\Big|
			\lesssim\Big(\sum_{i\neq k}W_i^{p_s-1}+|\varrho_0|^{p_s-1}\Big)W_k.
		\end{equation*}
		Here $o(1)$ denotes a quantity that goes to zero as $\delta$ tends to zero. Then a direct computation yields
		\begin{equation*}
			|J_1|+|J_2|+|J_3|+|J_4|+|J_5|+|J_6|
			\lesssim
			\left\lbrace
			\begin{aligned}
				&\mathcal{G}^2+\|\hat{f}\|_{\dot{H}^{-s}(\mathbb{R}^n)}^2+o(1)\Big(\mathcal{G}+\|\hat{f}\|_{\dot{H}^{-s}(\mathbb{R}^n)}\Big)\hspace{11mm}if\hspace{2mm}\mu=4s,    \\
				&\mathcal{G}^{p_s-1}+\|\hat{f}\|_{\dot{H}^{-s}(\mathbb{R}^n)}^{p_s-1}+o(1)\Big(\mathcal{G}+\|\hat{f}\|_{\dot{H}^{-s}(\mathbb{R}^n)}\Big)\hspace{6mm}if\hspace{2mm}0<\mu<4s.
			\end{aligned}
			\right.
		\end{equation*}
		Now returning to \eqref{rr-2}, we proceed to derive
		\begin{equation}\label{luying}
			\begin{split}
				(2-2p_s)\int(-\Delta)^\frac{s}{2}\varrho_1(-\Delta)^\frac{s}{2} W_k
				\lesssim&\int\sum_{i=1}^{\kappa}\sum_{a=1}^{n+1}c_{a}^{i}I_{n,\mu,s}[W_{i},\mathcal{Z}^{a}_i]W_k
				+\|\hat{f}\|_{\dot{H}^{-s}(\mathbb{R}^n)}     \\
				&+\left\lbrace
				\begin{aligned}
					&\mathcal{G}^2+\|\hat{f}\|_{\dot{H}^{-s}(\mathbb{R}^n)}^2+o(1)\Big(\mathcal{G}+\|\hat{f}\|_{\dot{H}^{-s}(\mathbb{R}^n)}\Big)\hspace{11mm}if\hspace{2mm}\mu=4s,    \\
					&\mathcal{G}^{p_s-1}+\|\hat{f}\|_{\dot{H}^{-s}(\mathbb{R}^n)}^{p_s-1}+o(1)\Big(\mathcal{G}+\|\hat{f}\|_{\dot{H}^{-s}(\mathbb{R}^n)}\Big)\hspace{6mm}if\hspace{2mm}0<\mu<4s.
				\end{aligned}
				\right.
			\end{split}
		\end{equation}
		For the left-hand side of \eqref{luying}, we have
		\begin{equation*}
			(2-2p_s)\int(-\Delta)^\frac{s}{2}\varrho_1\cdot(-\Delta)^\frac{s}{2} W_k
			=(2-2p_s)\sum_{i=1}^\kappa\gamma^i\int(-\Delta)^\frac{s}{2}W_i\cdot(-\Delta)^\frac{s}{2}W_k
			=(2-2p_s)\gamma^k.
		\end{equation*}
		For the right-hand side of \eqref{luying}, observing that for $i=k$, it holds $\int I_{n,\mu,s}[W_{k},\mathcal{Z}^{a}_k]W_k=0$.
		By Lemma \ref{FPU1}, for $i\neq k$, it holds
		\begin{equation*}
			\int I_{n,\mu,s}[W_{i},\mathcal{Z}^{a}_i]W_k
			\lesssim\int W_i^{2_s^\ast-1}W_k
			\approx\mathscr{Q}.
		\end{equation*}
		Furthermore, the coefficients $c_{a}^i$ are controlled by Lemma \ref{wanfan} and Lemma \ref{cll} as follows:
		$$|c_a^i|\lesssim\mathscr{Q}^{\min\{\frac{\mu}{n-2s},1\}}.
		$$
		Hence, putting these estimates together we obtain the desired thesis as $\delta$ is small, which concludes the proof of this lemma.
	\end{proof}
	
	As a byproduct of Lemma \ref{Ni-1-3} and Lemma \ref{rr-1}, we easily get the following estimate holds.
	\begin{lem}\label{rr-1-00}
		Given $s\in(0,\frac{n}{2})$, $n>2s$, $0<\mu<n$ with $0<\mu\leq4s$, then as $\delta$ is small we have:
		\begin{equation}
			\|\varrho_1\|_{\dot{H}^s(\mathbb{R}^n)}
			\lesssim\|\hat{f}\|_{\dot{H}^{-s}(\mathbb{R}^n)}+\mathscr{Q}^{1+\min\{\frac{\mu}{n-2s},1\}}.
		\end{equation}
	\end{lem}

	\subsection{Conclusion}
	We are now in position to conclude the proof of Theorem \ref{Figalli}. Coming back to \eqref{u-0}, we note that
	\begin{equation}\label{dagonggaocheng}
		(-\Delta)^s \varrho-I_{n,\mu,s}[\sigma,\varrho]-g-N(\varrho)
		-\big[(-\Delta)^s u-\big(|x|^{-\mu}\ast u^{p_s}\big)u^{p_s-1}\big]
		=0.
	\end{equation}
	Multiplying $\mathcal{Z}_k^{n+1}$ to \eqref{dagonggaocheng} and integrating over $\mathbb{R}^n$, we get
	\begin{equation}\label{dache}
		\Big|\int g\mathcal{Z}_k^{n+1}\Big|
		\leq\Big|\int I_{n,\mu,s}[\sigma,\varrho]\mathcal{Z}_k^{n+1}\Big|
		+\Big|\int N(\varrho)\mathcal{Z}_k^{n+1}\Big|
		+\Big|\int\hat{f}\mathcal{Z}_k^{n+1}\Big|.
	\end{equation}
	It is easy to see that
	\begin{equation}\label{hangzhou}
		\Big|\int\hat{f}\mathcal{Z}_k^{n+1}\Big|
		\lesssim\|\hat{f}\|_{\dot{H}^{-s}(\mathbb{R}^n)}.
	\end{equation}
	Furthermore, we proceed to estimate each term appearing in \eqref{dache}.
	\begin{lem}\label{Ni-1}
		Assume that $s\in(0,\frac{n}{2})$, $n>2s$, $0<\mu<n$ with $0<\mu\leq4s$. If $\delta$ is small, we have
		\begin{equation*}
			\Big|\int I_{n,\mu\,s}[\sigma,\varrho]\mathcal{Z}_k^{n+1}\Big|
			=\|\hat{f}\|_{\dot{H}^{-s}(\mathbb{R}^n)}
			+o(\mathscr{Q}^{\min\{\frac{\mu}{n-2s},1\}}).
		\end{equation*}
	\end{lem}
	\begin{proof}
		An easy computation shows that
		\begin{equation}\label{pp-1}
			\begin{split}
				\int I_{n,\mu,s}[\sigma,\varrho]\mathcal{Z}_{k}^{n+1}
				=\int I_{n,\mu,s}[\sigma,\varrho_{0}]\mathcal{Z}_{k}^{n+1}+\int I_{n,\mu,s}[\sigma,\varrho_{1}]\mathcal{Z}_{k}^{n+1}
				=:\tilde{h}_1+\tilde{h}_2.
			\end{split}
		\end{equation}
		Let us estimate each integral in \eqref{pp-1}. Similar to the argument of \eqref{qk1}-\eqref{qk2}, we get that
		\begin{equation*}
			|\tilde{h}_1|=o(\mathscr{Q}^{\min\{\frac{\mu}{n-2s},1\}}),
		\end{equation*}
		due to the orthogonality condition. Moreover, Lemma \ref{rr-1-00} implies that
		\begin{equation*}
			\hspace{4mm}
			|\tilde{h}_2|\lesssim\|\hat{f}\|_{\dot{H}^{-s}(\mathbb{R}^n)}+\mathscr{Q}^{1+\min\{\frac{\mu}{n-2s},1\}}.
		\end{equation*}
		Hence the result easily follows.
	\end{proof}
	
	\begin{lem}\label{Ni-2}
		Assume that $s\in(0,\frac{n}{2})$, $n>2s$, $0<\mu<n$ with $0<\mu\leq4s$. If $\delta$ is small we have
		\begin{equation*}
			\Big|\int N(\phi)\mathcal{Z}_{k}^{n+1}\Big|
			=\|\hat{f}\|_{\dot{H}^{-s}(\mathbb{R}^n)}
			+o(\mathscr{Q}^{\min\{\frac{\mu}{n-2s},1\}}).
		\end{equation*}
	\end{lem}
	\begin{proof}
		The conclusion follows by using the H\"{o}lder's inequality and Sobolev's inequality, Lemma \ref{p00} and  Lemma \ref{rr-1-00}.
	\end{proof}

	\begin{lem}\label{Ni-3}
		Assume that $s\in(0,\frac{n}{2})$, $n>2s$, $0<\mu<n$ with $0<\mu\leq4s$. If $\delta$ is small we have
		\begin{equation}\label{xiache}
			\int g\mathcal{Z}_k^{n+1}
			=\widetilde{\alpha}_{n,\mu,s}\sum_{i\neq k}\int W_i^{2_s^\ast-p_s}W_k^{p_s-1}\mathcal{Z}_k^{n+1}
			+\widetilde{\alpha}_{n,\mu,s}(2_s^\ast-1)\int W_k^{2_s^\ast-2}\sum_{i\neq k}^{\kappa}W_i\mathcal{Z}_k^{n+1}
			+o(\mathscr{Q}^{\min\{\frac{\mu}{n-2s},1\}}),
		\end{equation}
		where $\widetilde{\alpha}_{n,\mu,s}$ is defined in Lemma \ref{p1-00}.
	\end{lem}
	\begin{proof}
		We exploit the following decomposition
		\begin{equation}\label{zong}
			\begin{split}
				\int g\mathcal{Z}_k^{n+1}
				=&\int\big(|x|^{-\mu}\ast\big[\sigma^{p_s}-\sum_{i=1}^\kappa W_i^{p_s}\big]\big)\sigma^{p_s-1}\mathcal{Z}_k^{n+1}
				+\widetilde{\alpha}_{n,\mu,s}\sum_{i=1}^\kappa\int W_i^{2_s^\ast-p_s}\big(\sigma^{p_s-1}-\sum_{j=1}^\kappa W_j^{p_s-1}\big)\mathcal{Z}_k^{n+1}      \\
				&+\widetilde{\alpha}_{n,\mu,s}\sum_{i,j=1,i\neq j}^\kappa\int W_i^{2_s^\ast-p_s}W_j^{p_s-1}\mathcal{Z}_k^{n+1}
				=:J_1+J_2+J_3.
			\end{split}
		\end{equation}
		Using Lemma \ref{FPU3}, we verify that
		\begin{equation}\label{J_3}
			\begin{split}
				J_3
				&=\widetilde{\alpha}_{n,\mu,s}\sum_{i\neq k}\int W_i^{2_s^\ast-p_s}W_k^{p_s-1}\mathcal{Z}_k^{n+1}
				+\widetilde{\alpha}_{n,\mu,s}\sum_{i,j=1,i\neq j\neq k}^\kappa\int W_i^{2_s^\ast-p_s}W_j^{p_s-1}\mathcal{Z}_k^{n+1}      \\
				&=\widetilde{\alpha}_{n,\mu,s}\sum_{i\neq k}\int W_i^{2_s^\ast-p_s}W_k^{p_s-1}\mathcal{Z}_k^{n+1}
				+o(\mathscr{Q}^{\min\{\frac{\mu}{n-2s},1\}}).
			\end{split}
		\end{equation}
		Next we estimate $J_2$. If $\mu=4s$, then $J_2\equiv0$. Thus, in the following we consider two cases.
		\par \textbf{Case 1:} $n\geq6s$ and $0<\mu<4s$.
		\par If $W_l=\max W_i$ then we have $\Big|\sigma^{p_s-1}-\sum_{j=1}^\kappa W_j^{p_s-1}\Big|
		\lesssim W_l^{p_s-2}\sum_{j\neq l}^\kappa W_j$. Thus,
		\begin{equation*}
			|J_2|
			=\Big|\widetilde{\alpha}_{n,\mu,s}\sum_{l=1}^{\kappa}\int_{\{W_l=\max W_i\}}\sum_{i=1}^\kappa W_i^{2_s^\ast-p_s}W_l^{p_s-2}\sum_{j\neq l}^\kappa W_j\mathcal{Z}_k^{n+1}\Big|
			\lesssim\sum_{j\neq l}\int W_l^{2_s^\ast-1}W_j
			\lesssim\mathscr{Q}
			=o(\mathscr{Q}^{\min\{\frac{\mu}{n-2s},1\}}).
		\end{equation*}
		\par \textbf{Case 2:} $2s<n<6s$ and $0<\mu<4s$.
		\par We decompose $J_2$ as follows:
		\begin{equation*}
			\begin{split}
				J_2=&\widetilde{\alpha}_{n,\mu,s}\sum_{l=1}^{\kappa}\int_{\{W_l=\max W_i\}}\sum_{i=1}^\kappa W_i^{2_s^\ast-p_s}\big[\sigma^{p_s-1}-W_l^{p_s-1}-(p_s-1)W_l^{p_s-2}\sum_{j\neq l}W_j\big]\mathcal{Z}_k^{n+1}    \\
				&+\widetilde{\alpha}_{n,\mu,s}(p_s-1)\sum_{l=1}^{\kappa}\int_{\{W_l=\max W_i\}}\sum_{i\neq l}^\kappa W_i^{2_s^\ast-p_s}W_l^{p_s-2}\sum_{j\neq l}^{\kappa}W_j\mathcal{Z}_k^{n+1}    \\
				&+\widetilde{\alpha}_{n,\mu,s}(p_s-1)\sum_{l=1}^{\kappa}\int_{\{W_l=\max W_i\}}W_l^{2_s^\ast-2}\sum_{j\neq l}^{\kappa}W_j\mathcal{Z}_k^{n+1}        \\
				&-\widetilde{\alpha}_{n,\mu,s}(p_s-2)\sum_{l=1}^{\kappa}\int_{\{W_l=\max W_i\}}\sum_{i=1}^\kappa W_i^{2_s^\ast-p_s}\sum_{j\neq l}^{\kappa}W_j^{p_s-1}\mathcal{Z}_k^{n+1}
				=:J_{2,1}+J_{2,2}+J_{2,3}+J_{2,4}.
			\end{split}
		\end{equation*}
		Since $2s<n<6s$, we have $\Big|\sigma^{p_s-1}-W_l^{p_s-1}-(p_s-1)W_l^{p_s-2}\sum_{j\neq l}W_j\Big|\lesssim W_l^{p_s-3}\sum_{j\neq l}^{\kappa}W_j^2$. Thus,
		\begin{equation}\label{J_21}
			|J_{2,1}|
			\lesssim\sum_{j\neq l}\int W_l^{2_s^\ast-2}W_j^2
			=o(\mathscr{Q}^{\min\{\frac{\mu}{n-2s},1\}}).
		\end{equation}
		Furthermore, applying Lemma \ref{chufa} we get
		\begin{equation*}
			J_{2,2}
			=\widetilde{\alpha}_{n,\mu,s}(p_s-1)\sum_{l=1}^{\kappa}\int_{\{W_l=\max W_i\}}\sum_{i\neq l}^{\kappa}W_i^{2_s^\ast-p_s+1}W_l^{p_s-2}\mathcal{Z}_k^{n+1}
			+o(\mathscr{Q}^{\min\{\frac{\mu}{n-2s},1\}}),
		\end{equation*}
		which implies that
		\begin{equation}\label{J_22}
			|J_{2,2}|
			\lesssim\sum_{i\neq l}\int W_i^{2_s^\ast-p_s+1}W_l^{p_s-1}
			+o(\mathscr{Q}^{\min\{\frac{\mu}{n-2s},1\}})
			=o(\mathscr{Q}^{\min\{\frac{\mu}{n-2s},1\}}).
		\end{equation}
		For $J_{2,3}$, we have
		\begin{equation}\label{huanghe1}
			\begin{split}
				J_{2,3}
				=&\widetilde{\alpha}_{n,\mu,s}(p_s-1)\int_{\{W_k=\max W_i\}}W_k^{2_s^\ast-2}\sum_{j\neq k}^{\kappa}W_j\mathcal{Z}_k^{n+1}
				+\widetilde{\alpha}_{n,\mu,s}(p_s-1)\sum_{l\neq k}^{\kappa}\int_{\{W_l=\max W_i\}}W_k^{2_s^\ast-2}\sum_{j\neq l}^{\kappa}W_j\mathcal{Z}_k^{n+1}        \\
				&+\widetilde{\alpha}_{n,\mu,s}(p_s-1)\sum_{l\neq k}^{\kappa}\int_{\{W_l=\max W_i\}}W_l^{2_s^\ast-2}\sum_{j\neq l\neq k}^{\kappa}W_j\mathcal{Z}_k^{n+1}.
			\end{split}
		\end{equation}
		It follows from Lemma \ref{FPU1} that
		\begin{equation}\label{huanghe2}
			\Big|\widetilde{\alpha}_{n,\mu,s}(p_s-1)\sum_{l\neq k}^{\kappa}\int_{\{W_l=\max W_i\}}W_k^{2_s^\ast-2}\sum_{j\neq l}^{\kappa}W_j\mathcal{Z}_k^{n+1}\Big|
			\lesssim\sum_{l\neq k}\int_{\{W_l=\max W_i\}}W_k^{2_s^\ast-2}W_l^2
			=o(\mathscr{Q}^{\min\{\frac{\mu}{n-2s},1\}}),
		\end{equation}
		Moreover, thanks to Lemma \ref{FPU3}, we obtain that
		\begin{equation}\label{huanghe3}
			\Big|\widetilde{\alpha}_{n,\mu,s}(p_s-1)\sum_{l\neq k}^{\kappa}\int_{\{W_l=\max W_i\}}W_l^{2_s^\ast-2}\sum_{j\neq l\neq k}^{\kappa}W_j\mathcal{Z}_k^{n+1}\Big|
			\lesssim\sum_{j\neq l\neq k}\int_{\{W_l=\max W_i\}}W_l^{2_s^\ast-2}W_jW_k
			=o(\mathscr{Q}^{\min\{\frac{\mu}{n-2s},1\}}).
		\end{equation}
		Thus, combining \eqref{huanghe1}-\eqref{huanghe3}, we arrive at
		\begin{equation}\label{J_23}
			J_{2,3}
			=\widetilde{\alpha}_{n,\mu,s}(p_s-1)\int W_k^{2_s^\ast-2}\sum_{j\neq k}^{\kappa}W_j\mathcal{Z}_k^{n+1}
			+o(\mathscr{Q}^{\min\{\frac{\mu}{n-2s},1\}}).
		\end{equation}
		For $J_{2,4}$, using Lemma \ref{FPU1} and Lemma \ref{chufa} we have
		\begin{equation}\label{J_24}
			|J_{2,4}|
			\lesssim\sum_{j\neq l}\int_{\{W_l=\max W_i\}}W_i^{2_s^\ast-p_s+1}W_j^{p_s-1}
			=o(\mathscr{Q}^{\min\{\frac{\mu}{n-2s},1\}}).
		\end{equation}
		Taking \eqref{J_21}, \eqref{J_22}, \eqref{J_23} and \eqref{J_24} into account, if $2s<n<6s$ and $0<\mu<4s$ we can derive
		\begin{equation*}
			J_2
			=\widetilde{\alpha}_{n,\mu,s}(p_s-1)\int W_k^{2_s^\ast-2}\sum_{j\neq k}^{\kappa}W_j\mathcal{Z}_k^{n+1}
			+o(\mathscr{Q}^{\min\{\frac{\mu}{n-2s},1\}}),
		\end{equation*}
		which together with Case 1, we can derive that
		\begin{equation}\label{J_2}
			J_2
			=\widetilde{\alpha}_{n,\mu,s}(p_s-1)\int W_k^{2_s^\ast-2}\sum_{j\neq k}^{\kappa}W_j\mathcal{Z}_k^{n+1}
			+o(\mathscr{Q}^{\min\{\frac{\mu}{n-2s},1\}}).
		\end{equation}
		
		Similar to $J_2$, we compute $J_1$ as follows
		\begin{equation}\label{J_1}
			J_1
			=\widetilde{\alpha}_{n,\mu,s}(2_s^\ast-p_s)\int W_k^{2_s^\ast-2}\sum_{j\neq k}^{\kappa}W_j\mathcal{Z}_k^{n+1}.
		\end{equation}
		Therefore, the conclusion follows by putting \eqref{zong}, \eqref{J_3}, \eqref{J_2} and \eqref{J_1} together.
	\end{proof}
	
	As a consequence of Lemma \ref{Ni-3}, we deduce the following important estimate.
	\begin{lem}\label{QQ-1}
		Assume that $s\in(0,\frac{n}{2})$, $n>2s$, $0<\mu<n$ with $0<\mu\leq4s$. If $\delta$ is small we have
		\begin{equation}\label{Qf}
			\mathscr{Q}^{\min\{\frac{\mu}{n-2s},1\}}
			\lesssim\|\hat{f}\|_{\dot{H}^{-s}(\mathbb{R}^n)}.
		\end{equation}
	\end{lem}
	\begin{proof}
		We are going to show that the statement of \eqref{Qf} holds by induction. For each $2\leq\varsigma\leq\kappa$, let us introduce the induction hypothesis
		$$(\Pi_{\varsigma}):\hspace{4mm}\sum_{k=\varsigma}^{\kappa}\sum_{i=1}^{k-1}Q_{ik}^{\min\{\frac{\mu}{n-2s},1\}}\lesssim
		\big\|\hat{f}\big\|_{{\dot{H}^{-s}(\mathbb{R}^n)}}+o(\mathscr{Q}^{\min\{\frac{\mu}{n-2s},1\}}).$$
		By Lemma \ref{guji}, one has
		\begin{equation}\label{guming1}
			\widetilde{\alpha}_{n,\mu,s}\sum_{i\neq k}\int W_i^{2_s^\ast-p_s}W_k^{p_s-1}\mathcal{Z}_k^{n+1}
			\approx -Q_{ik}
		\end{equation}
		and
		\begin{equation}\label{guming2}
			\widetilde{\alpha}_{n,\mu,s}(2_s^\ast-1)\sum_{i\neq k}^{\kappa}\int W_k^{2_s^\ast-2}W_i\mathcal{Z}_k^{n+1}
			\approx -Q_{ik}^\frac{\mu}{n-2s},
		\end{equation}
		when $Q_{ik}\ll1$. It is easy to see that $Q_{ik}+Q_{ik}^\frac{\mu}{n-2s}\approx Q_{ik}^{\min\{\frac{\mu}{n-2s},1\}}$. Therefore, taking $k=\kappa$ in \eqref{dache} and \eqref{xiache} and using \eqref{hangzhou}, \eqref{guming1}, \eqref{guming2}, Lemma \ref{Ni-1} and Lemma \ref{Ni-2} we have
		\begin{equation*}
			\begin{split}
				\sum_{i=1}^{\kappa-1}Q_{i\kappa}^{\min\{\frac{\mu}{n-2s},1\}}
				&\approx-\sum_{i=1}^{\kappa-1}\int W_i^{2_s^\ast-p_s}W_\kappa^{p_s-1}\mathcal{Z}_\kappa^{n+1}
				-\sum_{i=1}^{\kappa-1}\int W_\kappa^{2_s^\ast-2}W_i\mathcal{Z}_\kappa^{n+1}
				+o(\mathscr{Q}^{\min\{\frac{\mu}{n-2s},1\}})    \\
				&\lesssim\|\hat{f}\|_{\dot{H}^{-s}(\mathbb{R}^n)}
				+o(\mathscr{Q}^{\min\{\frac{\mu}{n-2s},1\}}).
			\end{split}
		\end{equation*}
		The above estimate tells us $(\Pi_{\kappa})$ holds true. Now we can assume that the statement of $(\Pi_{\varsigma+1})$ holds, then we claim that $(\Pi_{\varsigma})$ is true. Indeed, choosing $k=\varsigma$ in this case, it follows from \eqref{dache}, \eqref{xiache}, \eqref{hangzhou}, \eqref{guming1}, \eqref{guming2}, Lemma \ref{Ni-1} and Lemma \ref{Ni-2} that
		\begin{equation*}
			\begin{split}
				\sum_{i=1}^{\varsigma-1}Q_{i\varsigma}^{\min\{\frac{\mu}{n-2s},1\}}
				&\lesssim\sum_{i=\varsigma+1}^{\kappa}\int W_i^{2_s^\ast-p_s}W_\varsigma^{p_s-1}\mathcal{Z}_\varsigma^{n+1}
				+\sum_{i=\varsigma+1}^{\kappa}\int W_\varsigma^{2_s^\ast-2}W_i\mathcal{Z}_\varsigma^{n+1}
				+\|\hat{f}\|_{\dot{H}^{-s}(\mathbb{R}^n)}
				+o(\mathscr{Q}^{\min\{\frac{\mu}{n-2s},1\}})     \\
				&\lesssim\sum_{i=\varsigma+1}^{\kappa}Q_{i\varsigma}^{\min\{\frac{\mu}{n-2s},1\}}
				+\|\hat{f}\|_{\dot{H}^{-s}(\mathbb{R}^n)}
				+o(\mathscr{Q}^{\min\{\frac{\mu}{n-2s},1\}})
				\lesssim\|\hat{f}\|_{\dot{H}^{-s}(\mathbb{R}^n)}
				+o(\mathscr{Q}^{\min\{\frac{\mu}{n-2s},1\}}).
			\end{split}
		\end{equation*}
		Then the claim follows by induction. Hence, we have the bound
		$$\mathscr{Q}^{\min\{\frac{\mu}{n-2s},1\}}\lesssim\big\|\hat{f}\big\|_{\dot{H}^{-s}(\mathbb{R}^n)}+o(\mathscr{Q}^{\min\{\frac{\mu}{n-2s},1\}}).$$
		As $\delta$ is small, we have concluded the proof.
	\end{proof}
	
	\begin{proof}[Proof of Theorem \ref{Figalli}]
		We first note that, by Lemma \ref{p00} and Lemma \ref{rr-1-00}, it follows immediately that
		\begin{equation}\label{popo}
			\|\varrho\|_{\dot{H}^s(\mathbb{R}^n)}
			\leq\|\varrho_0\|_{\dot{H}^s(\mathbb{R}^n)}+\|\varrho_1\|_{\dot{H}^s(\mathbb{R}^n)}
			\lesssim\mathcal{K}_{n,\mu,s}(\mathscr{Q}^{\min\{\frac{\mu}{n-2s},1\}})
			+\mathscr{Q}^{1+\min\{\frac{\mu}{n-2s},1\}}
			+\|\hat{f}\|_{\dot{H}^{-s}(\mathbb{R}^n)}.
		\end{equation}
		Combining with \eqref{Qf} and using the fact that $\mathcal{K}_{n,\mu,s}(x)$ is non-decreasing near $0$, we deduce that
		$$\big\|\varrho\big\|_{\dot{H}^s(\mathbb{R}^n)}
		\lesssim\mathcal{K}_{n,\mu,s}\big(\|\hat{f}\big\|_{\dot{H}^{-s}(\mathbb{R}^n)}\big).$$
		which concludes the proof of \eqref{tnu}. Finally, \eqref{moreover} follows by \eqref{Qf} and the following estimate
		\begin{equation*}
			\int_{\mathbb{R}^n}\Big(|x|^{-\mu}\ast W_i^{p_s}\Big)
			W_i^{p_s-1}W_j
			=\widetilde{\alpha}_{n,\mu,s}\int W_i^{2_s^{\ast}-1}W_j\approx Q_{ij}\leq\mathscr{Q}.
		\end{equation*}
	\end{proof}
	
	\begin{proof}[Proof of Corollary \ref{Figalli2}]
		It follows from Theorem \ref{emm} and Theorems \ref{Figalli} that there exists $\varepsilon>0$ such that
		$$\big\|(-\Delta)^s u-\left(|x|^{-\mu}\ast u^{p_s}\right)u^{p_s-1}\big\|_{\dot{H}^{-s}(\mathbb{R}^n)}\leq\varepsilon.$$
		Therefore for any $\delta>0$,
		$$
		\big\|u-\sum_{i=1}^{\kappa}W_i\big\|_{\dot{H}^s(\mathbb{R}^n)}\leq\delta,
		$$
		where $\{W_i\}_{1\leq i\leq\kappa}$ is a $\delta$-interacting family of Talenti bubbles, which concludes the proof.
	\end{proof}

	\section{A sharp example}\label{example}
	In this section, we shall construct an example showing that our quantitative estimate \eqref{tnu} is sharp for $n=6s$ and $\mu=4s$. It is important to note that for $n=6s$, $0<\mu<4s$ or $n\neq6s$, $0<\mu\leq4s$ with $\kappa\geq2$, the estimate is not sharp. In particular, Yang and Zhao in \cite{YZ25} showed that if $n\geq6-\mu$, $\mu\in(0,n)$ with $0<\mu\leq4$ satisfy the following assumption
	\begin{equation*}
		\frac{n^2-6n}{n-4}<\mu<4\hspace{4mm}\text{and}\hspace{4mm}n\geq6-\mu,
	\end{equation*}
	then there holds
	\begin{equation*}
		\Big\|u-\sum_{i=1}^{\kappa}W_i\Big\|_{\dot{H}^1(\mathbb{R}^n)}\lesssim
		\left\lbrace
		\begin{aligned}
			&	\Gamma(u)^\frac{n+2}{2(n-2)},\hspace{8mm}if\hspace{2mm}0<\mu<\frac{n+\mu-2}{2},\\
			&\Gamma(u)^\frac{2n-\mu-4}{n-2},\hspace{8mm}if\hspace{2mm}\frac{n+\mu-2}{2}\leq\mu<4.\\
		\end{aligned}
		\right.
	\end{equation*}
	
	\par WLOG, let us assume that $\kappa=2$. By Lemma \ref{xianyuxian}, choosing $\mathscr{R}$ large enough, we can find a solution $\varrho$ and a family of scalars $\{c_a^i\}$ such that
	\begin{equation}\label{weilan}
		\left\{\begin{array}{l}
			\displaystyle (-\Delta)^s \varrho-\Big(|x|^{-4s}\ast \big(\sigma+\varrho\big)^2\Big)
			\big(\sigma+\varrho\big)+\widetilde{\alpha}_{6s,4s,s}\sum_{i=1}^2W_i^2
			\displaystyle =
			\sum_{i=1}^2\sum_{a=1}^{n+1}c_{a}^{i}I_{6s,4s,s}[W_{i},\mathcal{Z}^{a}_i]\hspace{4.14mm}\mbox{in}\hspace{1.14mm} \mathbb{R}^{6s},\\
			\displaystyle \int_{\mathbb{R}^{6s}} I_{6s,4s,s}[W_{i},\mathcal{Z}^{a}_i]\phi=0,\hspace{4mm}i=1,2; ~a=1,\cdots,n+1.
		\end{array}
		\right.
	\end{equation}
	Here $\sigma=W_1+W_2$ and $\mathcal{Z}^{a}_i$ are the corresponding ones in \eqref{qta} for $W_1$ and $W_2$. It follows from Lemma \ref{cll}, Lemma \ref{xianyuxian} and Lemma \ref{p00} that
	\begin{equation}\label{nanshida}
		\sum_{i=1}^2\sum_{a=1}^{n+1}|c_{a}^{i}|\lesssim\mathscr{Q},\hspace{4mm}\|\varrho\|_\ast\leq c\hspace{4mm}\text{and}\hspace{4mm}\|\varrho\|_{\dot{H}^s(\mathbb{R}^{6s})}\lesssim\mathscr{Q}|\log\mathscr{Q}|^\frac{1}{2}.
	\end{equation}
	Now let $u:=W_1+W_2+\varrho$. Then we have
	\begin{equation*}
		(-\Delta)^su-\big(|x|^{-4s}\ast u^2\big)u
		=\sum_{i=1}^2\sum_{a=1}^{n+1}c_{a}^{i}I_{6s,4s,s}[W_{i},\mathcal{Z}^{a}_i]=:f.
	\end{equation*}
	By the Sobolev embedding, $|\mathcal{Z}_i^a|\lesssim W_i$ and \eqref{nanshida}, it is easy to see that
	\begin{equation*}
		\|f\|_{\dot{H}^{-s}(\mathbb{R}^{6s})}\lesssim\|f\|_{L^\frac{3}{2}(\mathbb{R}^{6s})}\lesssim\sum_{i=1}^2\sum_{a=1}^{n+1}|c_a^i|\|W_i^2\|_{L^\frac{3}{2}(\mathbb{R}^{6s})}\lesssim\mathscr{Q}.
	\end{equation*}
	Then we can follow the argument in \cite[Lemma 7.1]{DSW21} to derive that
	\begin{lem}\label{shubiao}
		For $\mathscr{R}$ large enough, one has
		\begin{equation*}
			\|\varrho\|_{\dot{H}^s(\mathbb{R}^{6s})}\gtrsim\|\hat{f}\|_{\dot{H}^{-s}(\mathbb{R}^{6s})}\Big|\log\|\hat{f}\|_{\dot{H}^{-s}(\mathbb{R}^{6s})}\Big|^\frac{1}{2}.
		\end{equation*}
	\end{lem}
	\begin{proof}[Proof of Theorem \ref{xiuzhengdai}]
		According to Lemma \ref{shubiao}, it suffices to show that
		\begin{equation*}
			\inf\limits_{\xi_1,\xi_2,\cdots,\xi_\kappa\in\mathbb{R}^{6s},
				\lambda_1,\lambda_2,\cdots,\lambda_{\kappa}\in\mathbb{R}^+}\big\|u-\sum_{i=1}^2 W[\xi_i,\lambda_i]\big\|_{\dot{H}^s(\mathbb{R}^{6s})}
			\gtrsim\|\varrho\|_{\dot{H}^s(\mathbb{R}^{6s})}.
		\end{equation*}
		It is well-known that the minimization problem on the left-hand side can be attained by some
		\begin{equation*}
			\widetilde{W}_1:=W[\xi_1,\lambda_1]\hspace{4mm}\text{and}\hspace{4mm}\widetilde{W}_2:=W[\xi_2,\lambda_2].
		\end{equation*}
		Denote $\widetilde{\sigma}=\widetilde{W}_1+\widetilde{W}_2$ and $\widetilde{\varrho}=u-\widetilde{\sigma}$. We need to show $\|\widetilde{\varrho}\|_{\dot{H}^s(\mathbb{R}^{6s})}\gtrsim\|\varrho\|_{\dot{H}^s(\mathbb{R}^{6s})}$. Since $\widetilde{\sigma}$ is the minimizer, then
		\begin{equation*}
			\|u-\widetilde{\sigma}\|_{\dot{H}^s(\mathbb{R}^{6s})}
			\leq\|u-\sigma\|_{\dot{H}^s(\mathbb{R}^{6s})}
			=\|\varrho\|_{\dot{H}^s(\mathbb{R}^{6s})}
			\lesssim\mathscr{Q}|\log\mathscr{Q}|^\frac{1}{2}.
		\end{equation*}
		Recall that $\langle v,w\rangle_{\dot{H}^s(\mathbb{R}^{6s})}=\int_{\mathbb{R}^{6s}}(-\Delta)^\frac{s}{2}v\cdot(-\Delta)^\frac{s}{2}w$. Hence $\|\sigma-\widetilde{\sigma}\|_{\dot{H}^s(\mathbb{R}^{6s})}\lesssim\mathscr{Q}|\log\mathscr{Q}|^\frac{1}{2}$. This implies that (up to some reordering of $\xi_1$ and $\xi_2$)
		\begin{equation*}
			\lambda_i=1+o(1),\hspace{2mm}\xi_i=(-1)^i(\mathscr{R}+o(1))e_1,\hspace{4mm}i=1,2.
		\end{equation*}
		Here $o(1)$ means a quantity that goes to zero when $\mathscr{R}\rightarrow\infty$. Denote
		\begin{equation*}
			\varepsilon:=\sum_{i}^2|\lambda_i-1|+|\xi_i-(-1)^i\mathscr{R}e_1|.
		\end{equation*}
		It is easy to see that $(\xi,\lambda)\rightarrow W[\xi,\lambda]$ is a smooth map from $\mathbb{R}^{6s}\times(0,\infty)$ to $\dot{H}^{s}(\mathbb{R}^{6s})$. Using the Taylor's expansion, there exist $A_1,A_2\in\dot{H}^s(\mathbb{R}^{6s})$ and $\|A_1\|_{\dot{H}^s(\mathbb{R}^{6s})}=O(\varepsilon^2)$, $\|A_2\|_{\dot{H}^s(\mathbb{R}^{6s})}=O(\varepsilon^2)$ such that
		\begin{equation*}
			\widetilde{W}_1-W_1=\sum_{a=1}^n\mathcal{Z}_1^a(\xi_1+\mathscr{R}e_1)_a+\mathcal{Z}_1^{n+1}(\lambda_1-1)+A_1
		\end{equation*}
		and
		\begin{equation*}
			\widetilde{W}_2-W_2=\sum_{a=1}^n\mathcal{Z}_2^a(\xi_2-\mathscr{R}e_1)_a+\mathcal{Z}_2^{n+1}(\lambda_2-1)+A_2,
		\end{equation*}
		where $\mathcal{Z}_1^a,\mathcal{Z}_2^a$ are defined in \eqref{qta} with respect to $W_1$ and $W_2$. Consequently, $\|W_1-\widetilde{W}_1\|_{\dot{H}^s(\mathbb{R}^{6s})}\approx\varepsilon,~\|W_2-\widetilde{W}_2\|_{\dot{H}^s(\mathbb{R}^{6s})}\approx\varepsilon$ and using Lemma \ref{armidale}, we get
		\begin{equation*}
			\Big|\langle W_1-\widetilde{W}_1,W_2-\widetilde{W}_2\rangle_{\dot{H}^s(\mathbb{R}^{6s})}\Big|\approx o(\varepsilon^2)
			+\varepsilon^2\sum_{a=1}^{n+1}\Big|\langle\mathcal{Z}_1^a,\mathcal{Z}_1^a\rangle_{\dot{H}^s(\mathbb{R}^{6s})}\Big|=o(\varepsilon^2).
		\end{equation*}
		Combining the above estimates, we have $\|\sigma-\widetilde{\sigma}\|_{\dot{H}^s(\mathbb{R}^{6s})}\approx\varepsilon$ because
		\begin{equation*}
			\|\sigma-\widetilde{\sigma}\|_{\dot{H}^s(\mathbb{R}^{6s})}^2
			=\|W_1-\widetilde{W}_1\|_{\dot{H}^s(\mathbb{R}^{6s})}^2
			+\|W_2-\widetilde{W}_2\|_{\dot{H}^s(\mathbb{R}^{6s})}^2
			+2\langle W_1-\widetilde{W}_1,W_2-\widetilde{W}_2\rangle_{\dot{H}^s(\mathbb{R}^{6s})}.
		\end{equation*}
		By the orthogonal conditions in \eqref{weilan}, we have $\varrho$ is orthogonal to $\mathcal{Z}_1^a$ and $\mathcal{Z}_2^a$ in $\dot{H}^s(\mathbb{R}^{6s})$ for $a=1,\cdots n+1$. Thus,
		\begin{equation*}
			\langle\sigma-\widetilde{\sigma},\varrho\rangle_{\dot{H}^s(\mathbb{R}^{6s})}
			=\int_{\mathbb{R}^{6s}}(-\Delta)^\frac{s}{2}(A_1+A_2)\cdot(-\Delta)^\frac{s}{2}\varrho
			\lesssim o(1)\|\sigma-\widetilde{\sigma}\|_{\dot{H}^s(\mathbb{R}^{6s})}\|\varrho\|_{\dot{H}^s(\mathbb{R}^{6s})}.
		\end{equation*}
		Since $\widetilde{\varrho}=\varrho+\sigma-\widetilde{\sigma}$, the above inequality implies that
		\begin{equation*}
			\begin{split}
				\|\widetilde{\varrho}\|_{\dot{H}^s(\mathbb{R}^{6s})}
				&=\|\varrho\|_{\dot{H}^s(\mathbb{R}^{6s})}
				+\|\sigma-\widetilde{\sigma}\|_{\dot{H}^s(\mathbb{R}^{6s})}
				+2\langle\sigma-\widetilde{\sigma},\varrho\rangle_{\dot{H}^s(\mathbb{R}^{6s})}       \\
				&\gtrsim\|\varrho\|_{\dot{H}^s(\mathbb{R}^{6s})}
				+\|\sigma-\widetilde{\sigma}\|_{\dot{H}^s(\mathbb{R}^{6s})}
				\gtrsim\|\varrho\|_{\dot{H}^s(\mathbb{R}^{6s})}.
			\end{split}
		\end{equation*}
		From this, the assertion follows.
	\end{proof}
	
	\appendix
	\section{Technical Lemmata}
	In this appendix, we give some crucial estimates that have been used in the previous sections.
	\begin{lem}\label{FPU2}
		Let $\tilde{s}>\tilde{t}>1$ and $\tilde{s}+\tilde{t}=2_s^{\ast}$. It holds
		$$\int W_i^t\inf(W_i^{\tilde{s}},W_i^{\tilde{t}})=O\big(Q_{ij}^{\frac{n}{n-2s}}|\log{Q_{ij}}|\big).$$
	\end{lem}
	\begin{proof}
		The proof of can be found in \cite[E4]{Bahri-1989}.
	\end{proof}

	\begin{lem}\label{armidale}
		For the $\mathcal{Z}_i^a$ defined in \eqref{qta}, there exist some constants $\gamma^a=\gamma^a(n)>0$ such that
		\begin{equation*}
			\int W_i^\frac{4s}{n-2s}\mathcal{Z}_i^a\mathcal{Z}_i^b=
			\left\lbrace
			\begin{aligned}
				&0\hspace{7.5mm}if\hspace{2mm}a\neq b,\\&
				\gamma^a\hspace{6mm}if\hspace{2mm}1\leq a=b\leq n+1.
			\end{aligned}
			\right.
		\end{equation*}
		If $i\neq j$ and $1\leq a,b\leq n+1$, then we have
		\begin{equation*}
			\Big|\int W_i^\frac{4s}{n-2s}\mathcal{Z}_i^a\mathcal{Z}_j^b\Big|
			\lesssim Q_{ij}.
		\end{equation*}
	\end{lem}
	\begin{proof}
		See the proof in \cite[F1-F6]{Bahri-1989}. Moreover, it is known that $\gamma^1=\cdots=\gamma^n$.
	\end{proof}

	\begin{lem}\label{guji}
		For any two bubbles, there exist two positive constants $c_1$ and $c_2$ such that
		\begin{equation*}
			\int W_i^{2_s^\ast-2}W_j\mathcal{Z}_i^{n+1}
			=-c_1Q_{ij}+o(Q_{ij}),
		\end{equation*}
		\begin{equation*}
			\int W_j^{2_s^\ast-p_s}W_i^{p_s-1}\mathcal{Z}_i^{n+1}
			=-c_2Q_{ij}^\frac{\mu}{n-2s}
			+o(Q_{ij}^\frac{\mu}{n-2s}).
		\end{equation*}
	\end{lem}
	\begin{proof}
		See the proof in \cite[F16]{Bahri-1989}.
	\end{proof}

	\begin{lem}\label{FPU1} Given $s\in(0,\frac{n}{2})$, $n>2s$, $\mu\in(0,n)$ with $0<\mu\leq4s$, let $W_i$ and $W_j$ be two bubbles. Then for any fixed $\varepsilon>0$ and any nonnegative exponents such that $\tilde{s}+\tilde{t}=2_s^\ast$, it holds that
		\begin{equation*}
			\int W_i^{\tilde{s}}W_j^{\tilde{t}}\approx
			\left\lbrace
			\begin{aligned}
				&Q_{ij}^{\min(\tilde{s},\tilde{t})},\quad\hspace{6mm}\hspace{3mm}\quad|\tilde{s}-\tilde{t}|\geq\varepsilon,\\&
				Q_{ij}^{\frac{n}{n-2s}}|\log Q_{ij}|,\quad\quad \tilde{s}=\tilde{t},
			\end{aligned}
			\right.
		\end{equation*}
		where the quantity
		\begin{equation*}	Q_{ij}=\min\Big(\frac{\lambda_i}{\lambda_j}+\frac{\lambda_j}{\lambda_i}+\lambda_i\lambda_j|\xi_i-\xi_j|^2\Big)^{-\frac{n-2s}{2}}.
		\end{equation*}
	\end{lem}
	\begin{proof}
		It is similar to that of Proposition B.2 in \cite{FG20}, so is omitted.
	\end{proof}

	\begin{lem}\label{p1-00}
		We have that
		\begin{equation*}
			|x|^{-\mu}\ast W_i^{p_s}
			=\int_{\mathbb{R}^n}\frac{W_i^{p_s}(y)}{|x-y|^{\mu}}dy
			=\widetilde{\alpha}_{n,\mu,s}W_i^{2_s^{\ast}-p_s}(x),
		\end{equation*}
		where  $\widetilde{\alpha}_{n,\mu,s}=I(\gamma)S^{\frac{(n-\mu)(2s-n)}{4(n-\mu+2s)}}C_{n,\mu}^{\frac{2s-n}{2(n-\mu+2s)}}[n(n-2s)]^{\frac{n-2s}{4}}.$
	\end{lem}
	\begin{proof}
		The conclusion follows by the Fourier transforms of the kernels of Riesz and Bessel potentials (cf. \cite{DHQWF}).
	\end{proof}

	\begin{lem}\label{FPU3}
		Given $s\in(0,\frac{n}{2})$, $n>2s$, $\mu\in(0,n)$ with $0<\mu\leq4s$ and any non-negative exponents such that $t_1+t_2+t_3=2_s^\ast$. Let $W_i$, $W_j$ and $W_k$ be three bubbles with $\delta$-interaction that is
		\begin{equation*}
			\mathscr{Q}=\max\big\{Q_{ij}, Q_{ik}, Q_{jk}\big\}\leq\delta.
		\end{equation*}
		Then we have
		\begin{equation*}
			\int W_i^{t_1}W_j^{t_2}W_k^{t_3}
			=o\bigg(\mathscr{Q}^{\min\{t_1,t_2,t_3\}}\bigg).
		\end{equation*}
	\end{lem}
	\begin{proof}
		We can assume that $t_1\leq t_2\leq t_3$. Let $\theta_1=\frac{t_2}{t_2+t_3}$ and $\theta_2=\frac{t_3}{t_2+t_3}$, we have
		\begin{equation*}
			\int W_i^{t_1}W_j^{t_2}W_k^{t_3}
			=\int\bigg(W_i^{t_1}W_j^{t_2}\bigg)^{\theta_1}\bigg(W_i^{t_1}W_k^{t_3}\bigg)^{\theta_2}\bigg(W_jW_k\bigg)^\frac{t_2t_3}{t_2+t_3}.
		\end{equation*}
		By the H\"older inequality and Lemma \ref{FPU1}, we get
		\begin{equation*}
			\begin{split}
				\int W_i^{t_1}W_j^{t_2}W_k^{t_3}
				&\leq\bigg(\int\big(W_i^{t_1}W_j^{t_2}\big)^\frac{2_s^\ast}{t_1+t_2}\bigg)^\frac{\theta_1(t_1+t_2)}{2_s^\ast}\bigg(\int\big(W_i^{t_1}W_k^{t_3}\big)^\frac{2_s^\ast}{t_1+t_3}\bigg)^\frac{\theta_2(t_1+t_3)}{2_s^\ast}\bigg(\int \big(W_jW_k\big)^\frac{2_s^\ast}{2}\bigg)^\frac{2t_2t_3}{2_s^\ast(t_2+t_3)}     \\
				&\leq
				\left\lbrace
				\begin{aligned}
					&\mathscr{Q}^{t_1}\mathscr{Q}^\frac{2_s^\ast}{2}|\log \mathscr{Q}|,\hspace{15mm}if\hspace{2mm}t_1<t_2\leq t_3,\\&
					\mathscr{Q}^{t_2}\mathscr{Q}^\frac{t_2t_3}{t_2+t_3}|\log\mathscr{Q}|^\frac{2t_2}{2_s^\ast},\hspace{6mm}if\hspace{2mm}t_1=t_2<t_3,\\&
					\mathscr{Q}^\frac{n}{n-2s}|\log\mathscr{Q}|,\hspace{18mm}if\hspace{2mm}t_1=t_2=t_3,    \\
				\end{aligned}
				\right.
			\end{split}
		\end{equation*}
		which implies the desired estimate.
	\end{proof}

	\begin{lem}\label{wwc101}
		Assume $s\in(0,\frac{n}{2})$, $n>2s$, $0<\mu<n$ with $0<\mu\leq4s$. There exists a $\Gamma_1^a>0$ such that
		\begin{equation*}
			\begin{split}
				\int I_{n,\mu,s}[W_{j},\mathcal{Z}^{a}_j]\mathcal{Z}_{j}^{b}=
				\left\lbrace
				\begin{aligned}
					&0,\hspace{7.3mm}\hspace{2mm}a\neq b,\\
					&\Gamma_0^a,\hspace{5mm}\hspace{2mm}1\leq a=b\leq n+1.
				\end{aligned}
				\right.
			\end{split}
		\end{equation*}
		If $i\neq j$, there holds
		\begin{equation*}
			\Big|\int\Big(|x|^{-\mu}\ast W_{i}^{p
				_s}\Big)W_{i}^{p_s-2}\mathcal{Z}_{i}^{a}\mathcal{Z}_{j}^{b}+\int\Big(|x|^{-\mu}\ast (W_{i}^{p_s-1}\mathcal{Z}_{i}^{a})\Big)W_{i}^{p_s-1}\mathcal{Z}_{j}^{b}\Big|\lesssim Q_{ij}, \hspace{2mm}\mbox{if}\hspace{2mm}1\leq a,b\leq n+1.
		\end{equation*}
		Here $\Gamma_0^1=\cdots=\Gamma_0^n$ and $\Gamma^{n+1}_0$ are composed of some $\Gamma$ functions. Moreover,
		we have that the following equality holds
		\begin{equation*}
			\begin{split}
				p_s\int
				&\Big(|x|^{-\mu}\ast W_k^{p_s-1}\partial_{\lambda_k} W_k\Big)
				W_k^{p_s-1}W_i+(p_s-1)\int
				\Big(|x|^{-\mu}\ast W_k^{p_s}\Big)
				W_k^{p_s-2}W_i\partial_{\lambda_k} W_k\\&
				=\int\Big(|x|^{-\mu}\ast W_i^{p_s}\Big)
				W_i^{p_s-1}\partial_{\lambda_k} W_k.
			\end{split}
		\end{equation*}
	\end{lem}
	\begin{proof}
		It holds:
		$$
		(-\Delta)^s \mathcal{Z}_j^a=p_s\Big(|x|^{-\mu}\ast W_{j}^{p_s-1}\mathcal{Z}_j^a\Big)W_{j}^{p_s-1}+(p_s-1)\Big(|x|^{-\mu}\ast W_{j}^{p_s}\Big)W_{j}^{p_s-2}\mathcal{Z}_j^a,
		$$
		for $1\leq a\leq n+1$.
		Therefore, it is easy to check that
		\begin{equation*}
			\begin{split}
				p_s&\int\Big(|x|^{-\mu}\ast W_{j}^{p_s-1}\mathcal{Z}_j^a\Big)W_{j}^{p_s-1}\mathcal{Z}_j^b+(p_s-1)\int\Big(|x|^{-\mu}\ast W_{j}^{p_s}\Big)W_{j}^{p_s-2}\mathcal{Z}_j^a\mathcal{Z}_j^b\\&=\int(-\Delta)^\frac{s}{2}\mathcal{Z}_j^a\cdot(-\Delta)^\frac{s}{2}\mathcal{Z}_j^b=C\int U_j^{2_s^{\ast}-2}\big(\frac{1}{\lambda_{j}}\frac{\partial U[\xi,\lambda_j]}{\partial \xi^a}\Big|_{\xi=\xi_j}\big)\big(\frac{1}{\lambda_{j}}\frac{\partial U[\xi,\lambda_j]}{\partial \xi^b}\Big|_{\xi=\xi_j}\big).
			\end{split}
		\end{equation*}
		Moreover, Lemma \ref{p1-00} gives us that
		\begin{equation*}
			|x|^{-\mu}\ast W_k^{p_s}
			=\widetilde{\alpha}_{n,\mu,s}W_k^{2_s^{\ast}-p_s}(x).
		\end{equation*}
		Hence one can follow from the same argument as in the proof of \cite{FG20}.
		Note that it holds
		$$(-\Delta)^s(\partial_{\lambda_k}W_k)=p_s\Big(|x|^{-\mu}\ast W_{k}^{p_s-1}\partial_{\lambda_k} W_k\Big)W_{k}^{p_s-1}+(p_s-1)\Big(|x|^{-\mu}\ast W_{k}^{p_s}\Big)W_{k}^{p_s-2}\partial_{\lambda_k} W_k.$$
		Then, the conclusion follows by simple integration by parts. The Lemma is obtained.
	\end{proof}
	
	\begin{lem}\label{chufa}
		Given $s\in(0,\frac{n}{2})$, $n>2s$, $\mu\in(0,n)$ with $0<\mu\leq4s$ and any non-negative exponents such that $\tilde{s}+\tilde{t}=2_s^\ast$. Let $W_i$ and $W_j$ be two bubbles with $\delta$-interaction. Then we have
		\begin{equation*}
			\int_{\{W_i\geq W_j\}}W_i^{\tilde{s}}W_j^{\tilde{t}}\lesssim
			\left\lbrace
			\begin{aligned}
				&Q_{ij}^{\tilde{t}},\hspace{8mm}if\hspace{2mm}\tilde{s}>\tilde{t},\\&
				Q_{ij}^\frac{n}{n-2s},\hspace{4mm}if\hspace{2mm}\tilde{s}<\tilde{t}.
			\end{aligned}
			\right.
		\end{equation*}
	\end{lem}
	\begin{proof}
		The only difference is that the exponents have been modified by the parameter $s$. Thus one can follow the same proof as in \cite[Lemma A.6]{DHP} and we omit it.
	\end{proof}

	\begin{lem}\label{B3}
		We have
		\begin{equation*}
			\int\frac{1}{|\hat{y}-x|^{n-2s}}\frac{1}{(1+|x|^2)^{\varsigma_0/2}}dx
			\lesssim\left\lbrace
			\begin{aligned}
				&\frac{1}{(1+|\hat{y}|^2)^{\varsigma_0/2-1}},\hspace{5mm}\hspace{2mm}2<\varsigma_0<n,\\&
				\frac{1+\log\sqrt{1+|\hat{y}|^2}}{(1+|\hat{y}|^2)^{(n-2s)/2}},\hspace{2mm}\hspace{2mm}\varsigma_0=n,
				\\&\frac{1}{(1+|\hat{y}|^2)^{(n-2s)/2}},\hspace{4mm}\hspace{2mm}\varsigma_0>n.
			\end{aligned}
			\right.
		\end{equation*}
	\end{lem}
	\begin{proof}
		The conclusion follows from the same proof as in Appendix B of \cite{FG20} except minor modifications.
	\end{proof}
	
	In order to study the higher order term $N(\phi)$, we introduce the auxiliary functions as follows:
	\begin{equation*}
		\begin{split}
			&N_1(\phi):=\big(|x|^{-\mu}\ast\sigma^{p_s}\big)\phi^{p_s-1},\hspace{4mm}
			N_2(\phi):=\big(|x|^{-\mu}\ast\sigma^{p_s-1}\phi\big)\sigma^{p_s-2}\phi,\hspace{4mm} N_3(\phi):=\big(|x|^{-\mu}\ast\sigma^{p_s-1}\phi\big)\phi^{p_s-1},     \\
			&N_4(\phi):=\big(|x|^{-\mu}\ast\sigma^{p_s-2}\phi^2\big)\sigma^{p_s-1},\hspace{4mm}
			N_5(\phi):=\big(|x|^{-\mu}\ast\sigma^{p_s-2}\phi^2\big)\sigma^{p_s-2}\phi,\hspace{4mm}
			N_6(\phi):=\big(|x|^{-\mu}\ast\sigma^{p_s-2}\phi^2\big)\phi^{p_s-1},    \\
			&N_7(\phi):=\big(|x|^{-\mu}\ast\phi^{p_s}\big)\sigma^{p_s-1},\hspace{
			4mm}
			N_8(\phi):=\big(|x|^{-\mu}\ast\phi^{p_s}\big)\sigma^{p_s-2}\phi,\hspace{4mm}
			N_9(\phi):=\big(|x|^{-\mu}\ast\phi^{p_s}\big)\phi^{p_s-1}.
		\end{split}
	\end{equation*}
	Applying the elementary inequality, we can derive the following estimate.
	\begin{lem}\label{gaojiexiang}
		Recalling that $N(\phi)$ is defined in \eqref{u-2}, then we have
		\begin{equation*}
			N(\phi)
			\lesssim N_1(\phi)+N_2(\phi)+N_3(\phi)+N_4(\phi)+N_5(\phi)+N_6(\phi)+N_7(\phi)+N_8(\phi)+N_9(\phi).
		\end{equation*}
	\end{lem}

\end{document}